\documentclass[12pt]{amsart}
\usepackage{a4wide}
\usepackage{amssymb}
\usepackage{amsthm}
\usepackage{amsmath}
\usepackage{amscd}
\usepackage{verbatim}
\usepackage[all]{xy}

\usepackage{longtable, lscape}
\newcommand{\note}[1]{\marginpar{\small \sf#1}}

\numberwithin{equation}{section}

\theoremstyle{plain}
\newtheorem{theorem}{Theorem}[section]
\newtheorem{corollary}[theorem]{Corollary}
\newtheorem{lemma}[theorem]{Lemma}
\newtheorem{proposition}[theorem]{Proposition}

\newtheorem{conjecture}[theorem]{Conjecture}
\theoremstyle{definition}

\theoremstyle{remark}
\newtheorem{remark}[theorem]{Remark}

\newcommand{\A}{\mathbb{A}}
\newcommand{\R}{\mathbb{R}}
\newcommand{\Q}{\mathbb{Q}}
\newcommand{\Z}{\mathbb{Z}}

\newcommand{\C}{\mathbb{C}}

\renewcommand{\H}{\mathbb{H}}

\newcommand{\F}{\mathbb{F}}
\newcommand{\D}{\mathbb{D}}

\newcommand{\zxz}[4]{\begin{pmatrix} #1 & #2 \\ #3 & #4 \end{pmatrix}}

\newcommand{\kzxz}[4]{\left(\begin{smallmatrix} #1 & #2 \\ #3 & #4\end{smallmatrix}\right) }

\newcommand{\calA}{A}

\newcommand{\calC}{\mathcal{C}}
\newcommand{\calD}{\mathcal{D}}

\newcommand{\calF}{\mathcal{F}}

\newcommand{\calH}{\mathcal{H}}

\newcommand{\calO}{\mathcal{O}}

\newcommand{\calX}{\mathcal{X}}
\newcommand{\calZ}{\mathcal{Z}}

\newcommand{\frakg}{\mathfrak g}
\newcommand{\frakk}{\mathfrak k}
\newcommand{\frakp}{\mathfrak p}

\newcommand{\bs}{\backslash}
\newcommand{\norm}{\operatorname{N}}
\newcommand{\vol}{\operatorname{vol}}
\newcommand{\tr}{\operatorname{tr}}

\newcommand{\Gl}{\operatorname{GL}}

\newcommand{\Sp}{\operatorname{Sp}}

\newcommand{\GSpin}{\operatorname{GSpin}}
\newcommand{\Mp}{\operatorname{Mp}}

\newcommand{\Uni}{\operatorname{U}}
\newcommand{\Hom}{\operatorname{Hom}}
\newcommand{\Aut}{\operatorname{Aut}}
\newcommand{\Mat}{\operatorname{Mat}}

\newcommand{\End}{\operatorname{End}}

\newcommand{\Sym}{\operatorname{Sym}}

\newcommand{\sig}{\operatorname{sig}}

\newcommand{\sym}{\text{\rm sym}}

\newcommand{\univ}{\text{\rm univ}}
\newcommand{\reg}{\text{\rm reg}}
\newcommand{\main}{\text{\rm main}}
\newcommand{\red}{\text{\rm red}}

\newcommand{\GL}{\operatorname{GL}}
\newcommand{\SO}{\operatorname{SO}}

\newcommand{\Ch}{\operatorname{Ch}}
\newcommand{\Cha}{\widehat{\operatorname{Ch}}}

\newcommand{\ord}{\operatorname{ord}}

\newcommand{\Ei}{\operatorname{Ei}}
\newcommand{\Lie}{\operatorname{Lie}}
\newcommand{\Ht}{\operatorname{ht}}
\newcommand{\pr}{\operatorname{pr}}
\newcommand{\uk}{\mathbf k}
\newcommand{\ff}{\operatorname{if}}
\newcommand{\IM}{\operatorname{Im}}
\newcommand{\diag}{\operatorname{diag}}
\newcommand{\co}{\mathcal  O}
\newcommand{\cha}{\operatorname{char}}
\newcommand{\Ind}{\operatorname{Ind}}
\newcommand{\Diff}{\operatorname{Diff}}

\newcommand{\RZ}{\operatorname{RZ}}
\newcommand{\newV}{\mathbb{V}}
\newcommand{\newH}{\mathbb{H}}
\newcommand{\newL}{ \mathbb{L}}
\newcommand{\newK}{\mathbb{K}}
\newcommand{\newphi}{\varphi}
\newcommand{\kay}{\bar{\mathbb F}_p}
\newcommand{\K}{\mathbb K}
\newcommand{\HH}{\tilde H}
\newcommand{\GSp}{\operatorname{GSp}}

\begin{document}

\title[Arithmetic degrees of special cycles]{Arithmetic degrees of special cycles and derivatives of Siegel Eisenstein series}

\author[Jan H.~Bruinier and Tonghai Yang]{Jan
Hendrik Bruinier and Tonghai Yang}
\address{Fachbereich Mathematik,
Technische Universit\"at Darmstadt, Schlossgartenstrasse 7, D--64289
Darmstadt, Germany}
\email{bruinier@mathematik.tu-darmstadt.de}
\address{Department of Mathematics, University of Wisconsin Madison, Van Vleck Hall, Madison, WI 53706, USA}
\email{thyang@math.wisc.edu}


\thanks{The first author is partially supported by DFG grant BR-2163/4-2 and the LOEWE research unit USAG.
The second author is partially supported by  NSF grant DMS-1762289.
}


\begin{abstract}
Let $V$ be a rational quadratic space of signature $(m,2)$.
A conjecture of Kudla relates the arithmetic degrees of top degree special cycles on an integral model of a Shimura variety associated with $\SO(V)$ to the coefficients of the central derivative of an incoherent Siegel Eisenstein series of genus $m+1$.
We prove this conjecture for the coefficients of non-singular index $T$ when $T$ is not positive definite. We also prove it when $T$ is positive definite and
the corresponding special cycle has dimension $0$. To obtain these results, we establish new local arithmetic Siegel-Weil formulas at the archimedean and non-archimedian places.
\end{abstract}

\maketitle


\section{Introduction}
\label{sect:intro}

The classical Siegel-Weil formula connects the arithmetic of quadratic forms with Eisenstein series for symplectic groups \cite{Si}, \cite{We}, \cite{KR1}. In particular, it yields explicit formulas for the representation numbers of integers by the genus of a quadratic form in terms of generalized divisor sum functions.

The Siegel-Weil formula also has important geometric applications.
For instance, it  leads to formulas for the degrees of special cycles on orthogonal Shimura varieties in terms of Fourier coefficients of Eisenstein series. To describe this,
we let $(V,Q)$ be a rational quadratic space of signature $(m,2)$.
To simplify the exposition, we assume throughout the introduction that $m$ is even, the general case is treated in the body of this paper.
Denote by $H=\SO(V)$ the special orthogonal group of $V$, and let $\calD$ be the corresponding hermitian symmetric space, realized as the Grassmannian of oriented negative definite planes in $V(\R)$.
For a compact open subgroup $K\subset H(\A_f)$ we consider the Shimura variety
\[
X_K=H(\Q)\bs (\calD\times H(\A_f)/ K) .
\]
It is a quasi-projective variety of dimension $m$, which has a canonical model over $\Q$.
Every positive definite subspace $U\subset V$ of dimension $n$ induces an embedding of groups $\SO(U^\perp)\to H$ and thereby a special cycle $Z(U)$ of codimension $n$ on $X_K$.
Moreover,
for every positive definite symmetric matrix $T\in \Sym_n(\Q)$ and every $K$-invariant Schwartz function $\varphi \in S(V^n(\A_f))$ there is a composite codimension $n$ cycle
\[
Z(T,\varphi)
\]
on $X_K$,
which is a certain linear combination of the $Z(U)$ for which $U$ has Gram matrix $2T$.
The classes of these cycles in the cohomology $H^{2n}(X_K,\C)$ and in the Chow group $\Ch^n(X_K)$ are important geometric invariants. Kudla and Millson also defined cycle classes $Z(T,\varphi)$ for positive semi-definite $T\in \Sym_n(\Q)$. They showed that the generating series of the cohomology classes of these cycles is a Siegel modular form of genus $n$ and weight $1+m/2$ (see \cite{KM3}, \cite{Ku:Duke}), generalizing the celebrated work of Hirzebruch-Zagier for Hilbert modular surfaces \cite{HZ}. The analogous statement for the classes in the Chow group was conjectured by Kudla and recently proved in \cite{Zh}, \cite{BW}.

It is natural to ask for more precise information about the automorphic properties of the generating series.
For the special cycles of maximal codimension, that is for $n=m$,
this question can be answered by means of the Siegel-Weil formula.
If $X_K$ has $r$ connected components, the compactly supported cohomology $H_c^{2m}(X_K,\C)$ is isomorphic to $\C^r$ via the degree maps on the connected components.
If $X_K$ is compact, Kudla showed \cite[Theorem 10.3]{Ku:Duke} that the generating series of the degrees of the special cycles is given by a Siegel Eisenstein series of genus $n$ and weight $\kappa=1+m/2$, that is,
\begin{align}
\label{eq:degid}
\sum_T \deg(Z(T,\varphi)) \cdot q^T =C\cdot E\big(\tau,1/2,\lambda(\varphi)\otimes \Phi_{\kappa}\big).
\end{align}
Here $\tau=u+iv$ is a variable in the Siegel upper half plane $\H_n$, and $C$ is an explicit normalizing constant which is independent of $\varphi$.
Moreover, $\lambda(\varphi)$ denotes a certain section of the induced representation $I(s,\chi_V)$ of $\Sp_n(\A_f)$ associated with $\varphi$, and $\Phi_{\kappa}$ denotes the standard section of weight $\kappa$ of the corresponding induced representation of $\Sp_n(\R)$, see Section \ref{sect2.1}. If $X_K$ is non-compact, the Eisenstein series is usually non-holomorphic and the treatment of the non-holomorphic contributions needs extra care, see e.g.~\cite{Fu:Compositio}, \cite{FM}.

The proof of this result involves the Schwartz-forms
$\varphi_{KM}^n(x,z)\in S(V^n(\R))\otimes A^{2n}(\calD)$ constructed by Kudla-Millson \cite{KM3}, which are Poincar\'e dual to special codimension $n$-cycles.  Since they transform with weight $\kappa$ under the maximal compact subgroup $U(n)\subset \Sp_n(\R)$, the theta series
\[
\theta^n_{KM}(\tau,\varphi,z,h)= \det(v)^{-\kappa/2}\sum_{x\in V^n(\Q)} \varphi(h^{-1}x)\cdot \big(\omega(g_\tau)\varphi_{KM}^n(x,z)\big)
\]
is a smooth (non-holomorphic) Siegel modular form of weight $\kappa$ in $\tau$. Here  $h \in  H(\A_f)$,  $\omega$ denotes the Weil representation of $\Sp_n$ and $g_\tau=\kzxz{1}{u}{}{1}
\kzxz{a}{}{}{{}^ta^{-1}}\in \Sp_n(\R)$ with $a\,{}^ta=v$.
The $T$-th Fourier coefficient of this theta series represents the
de Rham cohomology class of $Z(T,\varphi)$.
When   $n=m$, the generating series for the degrees is obtained by integrating $\theta^n_{KM}(\tau,\varphi, z,h)$ over $X_K$.
This can be evaluated by means of the Siegel-Weil formula, leading to \eqref{eq:degid}. As a consequence,
it can be shown that the intersection number of two special cycles
$Z(T_1,\varphi_1)$ and $Z(T_2,\varphi_2)$ of complementary codimensions $n_1$ and $n_2$ is given by the corresponding Fourier coefficient of the block diagonal restriction to $\H_{n_1}\times \H_{n_2}$ of the Eisenstein series (see \cite[Section~10]{Ku:Duke}). By means of the doubling method connections to special values of $L$-functions can be obtained.

\medskip

Kudla initiated a program connecting the Arakelov geometry of special cycles on integral models of orthogonal (and unitary) Shimura varieties to Siegel (Hermitian)  modular forms, see e.g.~\cite{Ku1}, \cite{KRY-book}. In particular, in this setting arithmetic degrees of special cycles are conjecturally connected to derivatives of Siegel Eisenstein series.
We describe some aspects of this program which are important for the present paper.

We consider arithmetic cycles in the sense of Gillet-Soul\'e (see \cite{GS}, \cite{SABK}), which are given by pairs consisting of a cycle on an integral model of $X_K$ and a Green current for the cycle.
For $x\in V(\R)$, Kudla constructed a Green function
\[
z\mapsto \xi_0(x,z)
\]
on $\calD$. It has a logarithmic singularity
along the special divisor determined by $x$, see \eqref{eq:defxi}.
More generally,
if $x=(x_1,\dots,x_n)\in V^n(\R)$ such that the moment matrix $T=Q(x)=\frac{1}{2}((x_i,x_j))_{i,j}$ is invertible, one obtains a Green current for the special codimension $n$ cycle
$\calD_x= \{z\in \calD\mid \; (z,x_1)=\dots (z,x_n)=0\}$ by taking the star product
\[
\xi^n_0(x,z)=\xi_0(x_1,z)*\dots *\xi_0(x_n,z).
\]
It satisfies the current equation
\[
dd^c [\xi^n_0(x,z)] + \delta_{\calD_x} = [\varphi^n_{KM,0}(x,z)],
\]
where $\varphi^n_{KM,0}(x,z)= \varphi^n_{KM}(x,z)\cdot e^{2\pi \tr Q(x)}$ is essentially the Poincar\'e dual form considered above, and $\delta_{\calD_x}$ is the Dirac current given by integration over $\calD_x$.
For the rest of this introduction we assume that
 $T\in \Sym_n(\Q)$ is invertible. Then we obtain a
Green current for the cycle $Z(T,\varphi)$ on $X_K$ by
\[
G(T,\varphi,v,z,h) = \sum_{\substack{x\in V^n(\Q)\\Q(x)=T}}
\varphi(h^{-1}x)\cdot \xi^n_0(xv^{1/2},z),
\]
where $z\in \calD$ and $h\in H(\A_f)$.

To describe the integral models of $X_K$ and the special cycles we are working with, we
assume for convenience that $V$ contains a unimodular even lattice $L$. This assumption can
and will be relaxed when considering local integral models later on, see Remark \ref{rem:maintheo1}. We let $K = \SO(\hat L)$ be the stabilizer of $\hat L=L\otimes\hat \Z $ in $H(\A_f )$, and let $\varphi =\varphi_L= \operatorname{char}(\hat L^n)$
be the characteristic function of $\hat L^n$.
By work of Kisin, Vasiu, and Madapusi Pera, the Shimura variety $X_K$ has a canonical integral model $\calX_K$, which is a smooth stack over $\Z$,  see \cite{Ki}, \cite{MP}, \cite[Theorem 4.2.2]{AGHM}.
There is a polarized abelian scheme $\calA$
of dimension $2^{m+1}$ over $\calX_K$, which is equipped with an action of the Clifford algebra $C(L)$ of $L$. For an $S$-valued point of $\calX_K$ there is a space of special endomorphisms
\[
V(\calA_S)\subset \End_{C(L)}(\calA_S)
\]
on the pull-back $A_S$ of $A$,
which is endowed with a positive definite even quadratic form $Q$, see
\cite[Section~4]{AGHM}. It can be used to define an integral model of $\calZ(T, \varphi)$ of $Z(T, \varphi)$ as the sub-stack of $\calX_K$ whose $S$-valued points have an $x\in V(A_S)^n$
with $Q(x)=T$.
The pair
\[
\widehat \calZ(T,\varphi,v) = \big(\calZ(T,\varphi),G(T,\varphi,v) \big)
\]
determines a class in an arithmetic Chow group.
Through the Green current it depends
on $v = \Im(\tau)$. In analogy with the geometric situation described earlier, we would like to
understand the classes of these cycles and their relations.

As before we focus on the case of top degree cycles, which is here the case $n = m + 1$.
If $T$ is not positive definite, then  $\calZ(T,\varphi)$
vanishes, but the arithmetic cycle $\widehat \calZ(T,\varphi,v)$
has non-trivial current part. On the other hand, if $T$ is positive definite, then $\widehat \calZ(T,\varphi,v)$
has trivial current part, and the cycle is entirely supported in positive characteristic. In
fact, if it is non-trivial then it is supported in the fiber above one single prime $p$. The
dimensions of the irreducible components were recently determined by Soylu \cite{Cihan-thesis}.
In particular, he showed that $\mathcal Z(T, \varphi)(\kay)$ is finite if and only if the reduction of $T$ modulo $p$
is of  rank $n-1$, $n-2$, or of rank $n-3$ (plus a technical condition).
We refer  to Theorem \ref{theo:Soylu} for details.
Throughout this paper we consider the cases when either  $T$ is not positive definite, or $T$ is positive definite and $\calZ(T,\varphi)$ has dimension $0$. 

According to \cite[Theorem 7.4]{MP}, there exists a regular toroidal compactification $\overline\calX_K$ of $\calX_K$ with generic fiber $\overline  X_K$. The cycle $\widehat \calZ(T,\varphi,v)$ defines a class in $\Cha^n_\C(\overline\calX_K)$.
Recall that there exists an arithmetic degree map
\[
\widehat \deg: \Cha^n_\C(\overline\calX_K)\longrightarrow\C
\]
which is given as a sum of local degrees
\[
\widehat \deg(\calZ,G) = \sum_{p\leq \infty} \widehat\deg_p(\calZ,G),
\]
where for an arithmetic cycle $(\calZ,G)$ the local degrees are defined as
\[
\widehat\deg_p(\calZ,G)= \begin{cases} \displaystyle
\sum_{x\in \calZ(\bar\F_p)}\frac{ \Ht_p(x)}{|\Aut(x)|}\cdot \log(p) ,&\text{if $p<\infty$,}\\[3ex]
\displaystyle
\frac{1}2 \int_{\overline \calX_K(\C)} G,&\text{if $p=\infty$.}
\end{cases}
\]
Here $\Ht_p(x)$ denotes the length of the \'etale local ring $\calO_{\calZ,x}$ of $\calZ$ at the point $x$. Kudla conjectured the following description of the arithmetic degrees of special cycles
in terms of derivatives of Siegel Eisenstein series of genus $n$, see \cite{Ku1}, \cite{Ku:MSRI}.

\begin{conjecture}[Kudla]
\label{conj:ku}
Assume that $n=m+1$ and that $T\in \Sym_n(\Q)$ is invertible.
Then
\[
\widehat \deg \big( \widehat \calZ(T,\varphi,v)\big) \cdot q^T= \hat C\cdot E'_T(\tau,0, \lambda(\varphi)\otimes \Phi_{\kappa}),
\]
where $\hat C$ denotes a constant which is independent of $T$ and $\varphi$ (see Theorem \ref{theo:ArithSW-infinite}), $E_T(\tau,s,\Phi)$ denotes the $T$-th Fourier coefficient of a Siegel Eisenstein series $E(\tau,s,\Phi)$, and the derivative is taken with respect to $s$.
%
\end{conjecture}

Note that for $T$ positive definite with $\calZ(T,\varphi)$ of higher dimension the arithmetic degree has to be defined more carefully as in \cite{Te}, but we do not consider this here.
The conjecture can be further generalized to include the cases where $T$ is singular, leading to an identity between the generating series of the arithmetic degrees of the $\widehat\calZ(T,\varphi,v)$ and the central derivative of the Eisenstein series
$E(\tau,s, \lambda(\varphi)\otimes \Phi_{\kappa})$ analogous to
\eqref{eq:degid}, which can be viewed as an {\em arithmetic Siegel-Weil formula}. The full conjecture is known for $m=0$ and for the $m=1$ case of Shimura curves,
see \cite{KRY-tiny}, \cite{KRY-book}.


To state our results on Conjecture~\ref{conj:ku}, we
let $\mathcal C =\bigotimes_{p\le \infty} \mathcal C_p$ be the incoherent quadratic space over $\A$ for which $\mathcal C_f =\otimes_{p <\infty} \mathcal C_p \cong V(\A_f)$ and $\mathcal C_\infty$ is positive definite of dimension $m+2$. The Eisenstein series appearing in Conjecture \ref{conj:ku} is naturally associated with the Schwartz function on $S(\calC^n)$ given by the tensor product of $\varphi$ and the Gaussian on $\calC_\infty^n$ via the intertwining operator $\lambda$.
Hence it is incoherent and vanishes at $s=0$. The conjecture gives a formula for the leading term of the Taylor expansion in $s$ at $s=0$.
Following Kudla \cite{Ku1}, define the Diff set associated with $\calC$ and $T$ as
\begin{equation}
\Diff(\mathcal C, T) =\{ p \le \infty: \;\mathcal C_p  \hbox{ does not represent } T\}.
\end{equation}
Then  $\Diff(\mathcal C, T)$ is a finite set of odd order, and $\infty \in \Diff(\mathcal C, T)$ if and only if $T$ is not positive definite.

\begin{theorem}
\label{maintheo1}
Assume that $T \in \Sym_n(\Q)$ is invertible. Then Conjecture \ref{conj:ku} holds in the following cases:
\begin{enumerate}

\item If $|\Diff(\mathcal C, T)| >1$. In this case both sides of the equality vanish.

\item If  $\Diff(\mathcal C, T) =\{ \infty\}$. In this case
$T$ is not positive definite, and the only contribution comes from the archimedian place,  i.e.,
\[
\widehat \deg \big( \widehat \calZ(T,\varphi,v)\big) \cdot q^T= \widehat \deg_\infty \big( \widehat \calZ(T,\varphi,v)\big) \cdot q^T = \hat C\cdot E'_T(\tau,0, \lambda(\varphi)\otimes \Phi_{\kappa}).
\]

\item If $\Diff(\mathcal C, T) =\{p\}$ for a finite prime $p\neq 2$ and  $\mathcal Z(T,\varphi)(\kay)$ has dimension $0$. In such a case, the only contribution comes from the prime $p$, i.e,
\[
\widehat \deg \big( \widehat \calZ(T,\varphi,v)\big) \cdot q^T= \widehat \deg_p \big( \widehat \calZ(T,\varphi,v)\big) q^T = \hat C\cdot E'_T(\tau,0, \lambda(\varphi)\otimes \Phi_{\kappa}).
\]
\end{enumerate}
\end{theorem}

\begin{remark}
\label{rem:maintheo1}
Since the cycle $\widehat \calZ(T,\varphi,v)$ is supported in a single fiber, all assertions of Theorem \ref{maintheo1} can be reformulated in terms of `local' models of $\overline X_K$. We will actually prove the local analogues in much greater generality.

Generalizing (2) we will show that if $\infty \in \Diff(\mathcal C, T)$ then
\begin{align}
\label{eq:refi1}
 \widehat \deg_\infty \big( \widehat Z(T,\varphi,v)\big) \cdot q^T = \hat C\cdot E'_T(\tau,0, \lambda(\varphi)\otimes \Phi_{\kappa}).
\end{align}
Since this is an assertion only about the complex fiber $\overline X_K$, we will be able to prove it for any compact open subgroup $K\subset H(\A_f)$ (in particular $V$ does not have to contain an even unimodular lattice) and any $K$-invariant Schwartz function $\varphi\in S(V(\A_f)^n)$,
see Theorem~\ref{theo:ArithSW-infinite}.

To generalize (3) we consider the canonical integral model $\calX_{K,(p)}$ of $\overline X_K$ over the localization $\Z_{(p)}$. In this setting we will show that if $p\in \Diff(\mathcal C, T)$ for a finite prime $p\neq 2$ and $\calZ(T,\varphi)$ is $0$-dimensional, then
\begin{align}
\label{eq:refi2}
\widehat \deg_p \big( \widehat \calZ(T,\varphi,v)\big) q^T = \hat C\cdot E'_T(\tau,0, \lambda(\varphi)\otimes \Phi_{\kappa}).
\end{align}
This will be proved under the assumption that $K$ is the stabilizer of a $\Z_p$-unimodular lattice $L\subset V$ and for $\varphi =\cha(\hat L^n)$, see Theorem \ref{thm:arithswfin}.
\end{remark}


To prove Theorem \ref{maintheo1}, we decompose the Fourier coefficients of the Eisenstein series into local factors.
If
$\Phi=\otimes_v \Phi_v$ is a factorizable section of the induced representation, then
\begin{align*}
E_T(g,s,\Phi)= \prod_{v\leq \infty} W_{T,v}(g,s,\Phi_v),
\end{align*}
where $W_{T,v}(g,s,\Phi_v)$ is the local Whittaker function given  by \eqref{eq:Whittaker}.
It is a basic fact that  the local Whittaker function  $W_{T, p}(g_p, 0, \lambda_p(\varphi_p))$ vanishes for every $p \in \Diff(\mathcal C, T)$.

This implies assertion (1) of Theorem \ref{maintheo1} in a rather direct way. Indeed, if $|\Diff(\mathcal C, T) |>1$ then
the right hand  side of the conjectured identity is automatically zero. To prove that the left hand side also vanishes, we consider for a prime $p\leq \infty$ the neighboring global quadratic space  $V^{(p)}$ at $p$ associated with $\calC$, which is the quadratic space over $\Q$ with local components $V_q^{(p)} \cong \mathcal C_q$ for all $q \ne p$
and such that $V_p^{(p)}$ and $\mathcal C_p$ have  the same dimension and quadratic character but different Hasse invariants (for $p=\infty$ we also require that $V_\infty^{(\infty)}$ has signature $(m,2)$, and hence $V^{(\infty)}=V$).

When $\mathcal Z(T,\varphi)(\kay)$ is non-empty for a prime $p<\infty$, then one can show (see the proof of Proposition \ref{prop:Counting} for example) that $V^{(p)}$ represents $T$.
This implies $\Diff(\mathcal C, T) =\{ p\}$. Similarly, the proof of Theorem \ref{theo:ArithSW-infinite} shows that if
the Green current $G(T,\varphi,v)$ is non-vanishing, then
$V^{(\infty)}$ represents $T$ and hence  $\Diff(\mathcal C, T) =\{ \infty\}$.

In the situation of part (2) of Theorem \ref{maintheo1}, when $\Diff(\mathcal C, T) =\{ \infty\}$, the local Whittaker function $W_{T,\infty}(g,0,\Phi_\kappa)$ vanishes, and hence
\[
E'_T(\tau,0, \lambda(\varphi)\otimes \Phi_{\kappa}) =
\prod_{p< \infty} W_{T,p}(g,0,\lambda(\varphi_p))\times W'_{T,\infty}(g,0,\Phi_\kappa).
\]
The derivative of the archimedian Whittaker function
is given by the following {\em arithmetic local Siegel-Weil formula} for the archimedian local height function
\[
\Ht_\infty(x)= \frac{1}2 \int_{\calD}\xi^n_0(x,z)
\]
on $V^n(\R)$,
which is our second main result (see also Theorem \ref{thm:alsw}).  The contributions from the non-archimedian  places can be computed by means of the \emph{local Siegel-Weil} formula (see Propositions \ref{prop:localSW} and \ref{prop:Measure}) .

\begin{theorem}
\label{thm:intro-alsw}
Let $x\in V^n(\R)$ such that the $Q(x)=T$
is invertible.
Then we have
\begin{align}
\label{eq:intro-lsw}
\Ht_\infty(xv^{1/2}) \cdot q^T
=-B_{n,\infty} \det(v)^{-\kappa/2}\cdot
W_{T, \infty}'(g_\tau,0,\Phi_{\kappa}),
\end{align}
where
$$
B_{n, \infty}=\frac{e(\frac{n^2+n-4}{8}) (n-1)! \prod_{k=1}^{n-1} \Gamma(\frac{n-k}2)}
       {2^{n-2} (2\pi)^{\frac{n(n+3)}4}}.
$$
\end{theorem}


In the special case $m=0$ Theorem \ref{thm:intro-alsw}
was proved in \cite{KRY-tiny},
for $m=1$ in \cite{Ku1}, and for $m=2$ in \cite{YZZ}.
For the related case of Shimura varieties associated to unitary groups of signature $(m,1)$ it was proved in \cite{Liu}. But the argument of \cite{Liu} does not seem to generalize to the case of orthogonal groups considered in the present paper. Recently, Garcia and Sankaran employed Quillen's theory of super-connections to obtain a different proof  of Theorem \ref{thm:intro-alsw}, see \cite{GaSa}.

In all these works it is first noticed that
because of the equivariance of $\xi^n_0(x,z)$ with respect to the action of $H(\R)$, the local height function $\Ht_\infty(x)$ only depends on $T=Q(x)$.
Then a crucial step consists in proving that
$\Ht_\infty(T):=\Ht_\infty(x)$ is invariant under the action of $\SO(n)$ on $\Sym_n(\R)$ (respectively $\Uni(n)$ on $\operatorname{Herm}_n(\C)$) by conjugation. Hence it suffices to prove the claimed identity for diagonal $T$. In this case the star product reduces to a single integral, which
can be related to the derivative of the Whittaker function by a direct (but rather involved) computation.

Our approach is different. For general nonsingular $T$, we consider the recursive formula for the star product (see \eqref{eq:wr}) and compute its `main term' by means of the classical archimedian local Siegel Weil formula
(see Theorem~\ref{thm:pre-alsw}).
The result turns out to be the sum of
a main term, which is the desired right hand side of \eqref{eq:intro-lsw}, plus a boundary term, given by the derivative of a genus $n-1$ Whittaker function. By an inductive argument, the boundary term cancels against the remaining terms of the star product. This approach does not require proving $\SO(n)$-invariance of the local height function at the outset. We obtain this invariance a posteriori from the obvious invariance of the Whittaker function.

Finally, we describe our approach to part (3) of Theorem \ref{maintheo1}. When $\Diff(\mathcal C, T) =\{ p\}$ for a finite prime $p\neq 2$, the local Whittaker function $W_{T,p}(g,0,\Phi_\kappa)$ vanishes, and hence
\[
E'_T(\tau,0, \lambda(\varphi)\otimes \Phi_{\kappa}) = W'_{T,p}(1,0,\lambda(\varphi_p))\times
\prod_{\substack{q< \infty\\ q\neq p}} W_{T,q}(1,0,\lambda(\varphi_q))\times W_{T,\infty}(\tau,0,\Phi_\kappa).
\]
The derivative of the local Whittaker function at $p$
is given by the following {\em arithmetic local Siegel-Weil formula}, which is our third main result paralleling Theorem \ref{thm:intro-alsw}. The terms away from  $p$ can again be computed by means of the local Siegel-Weil formula.
Recall that
$\varphi =\varphi_L
\in  S( V^n (\A_f ))$
is the characteristic function of $\hat L^n$.


\begin{theorem}
\label{theo:p-ASW} Let $p\ne 2$ be a prime number and assume  that $\mathcal Z(T,\varphi)(\kay)$ is finite. Then for $x \in  \mathcal Z(T,\varphi)(\kay)$, the local height $\Ht_p(x)$ is independent of the choice of $x$ and is given by
 $$
\Ht_p(x) \cdot \log p = \frac{W_{T, p}'(1, 0, \lambda(\varphi))}{W_{T^u, p}(1, 0, \lambda(\varphi))},
$$
where $T^u$ is any unimodular matrix in $\Sym_n(\Z_p)$ (i.e., $\det T^u \in \Z_p^\times$).
\end{theorem}

This theorem will be restated and proved as Theorem \ref{theo:LASW-finite}.  As in the archimedian case the proof is given by an  inductive argument.  According to Soylu's condition mentioned above, 
$\mathcal Z(T,\varphi) (\kay)$ being finite implies that $T$ is $\Z_p$-equivalent to $\diag(T_1, T_2)$ where $T_1$ is $\Z_p$-unimodular of rank $n-3$.  On the Whittaker function side, we will prove the following recursive formula (see Proposition \ref{prop:Surprise}):
\begin{equation}
\frac{W_{T, p}'(1, 0, \lambda(\varphi_{L_p}))}{W_{T^u, p}(1, 0, \lambda(\varphi_{L_p}))}
=\frac{W_{T_2, p}'(1, 0, \lambda(\varphi_{L_{2, p}}))}{W_{T_2^u, p}(1, 0, \lambda(\varphi_{L_{2, p}}))}.
\end{equation}
Here $T^u$ and $T_2^u$ are unimodular  symmetric $\Z_p$ matrices of order $n$ and $3$ respectively, and $L_{2, p}$ is a unimodular $\Z_p$-quadratic lattice of rank $4$ with
$$
L_p = M_{1, p} \oplus L_{2, p}
$$
for a unimodular $\Z_p$-quadratic lattice $M_{1, p}$ whose quadratic matrix is given by $T_1$.

This suggests a similar recursion for the local height function  $\Ht_p(x)$.
Soylu proved that the abelian variety associated with $x$ is  supersingular. The local height function depends only on the associated $p$-divisible group, and it can be computed using the $p$-adic uniformization of the supersingular locus by a Rapoport-Zink space (see Section \ref{sect:LASW-finite}).
The required recursion formula for the local height function is proved by employing recent work of Li and Zhu (\cite[Lemma 3.1.1]{LZ}, see Corollary \ref{cor:LiZhu}).

By the recursion formulas, the proof of Theorem \ref{theo:p-ASW} is reduced to the case $n=3$ in the local situation, where
$L=L_{2, p}$
is a unimodular $\Z_p$-lattice of rank $4$ and
$T=T_{2, p}$ is a symmetric $\Z_p$-matrix of rank $3$, and $x=(x_1, x_2, x_3)$ determines a point in $\mathcal Z(T,\varphi)(\kay)$ on the associated Rapoport-Zink space. This turns out to be exactly the local case considered by Kudla and Rapoport in their work on (twisted) Hilbert modular surfaces \cite{KRHilbert}.

\medskip

This paper is organized as follows.
Section 2 contains some preliminaries and basic facts about Whittaker functions. Moreover, we state the classical local Siegel-Weil formula with an explicit formula for the constant of proportionality. In Section 3 we derive a variant of the archimedian local Siegel Weil formula for integrals of certain Schwartz functions over the hermitian symmetric space of the orthogonal group, again with explicit 
constant of
proportionality. The main result Theorem~\ref{prop:LocalSiegel-WeilR} is one of the key ingredients in the proof of Theorem~\ref{thm:intro-alsw}.
In Section 4 we investigate the asymptotic behavior of the archimedian Whittaker function as one of the radial parameters goes to infinity. Our analysis relies on Shimura's work on confluent hypergeometric functions \cite{Sh}. The main result,  Theorem \ref{thm:etalimnew}, which is of independent interest, is the second main ingredient in the proof of Theorem~\ref{thm:intro-alsw}.
Section 5 is devoted to the proof of the archimedian arithmetic Siegel-Weil formula Theorem~\ref{thm:intro-alsw}.
In Section 6 we recall some facts about the Rapoport-Zink space for $\GSpin$ groups from \cite{HP} and \cite{Cihan-thesis} and prove the non-archimedian local arithmetic Siegel-Weil formula,  Theorem~\ref{theo:p-ASW}. Finally, Section 7 contains the proofs of our main global results, Theorem~\ref{maintheo1} and
the refinements described in Remark \ref{rem:maintheo1}.

We thank Jose Burgos Gil, Stephan Ehlen, Jens Funke, Ben Howard, Steve Kudla, Chao Li, Cihan Soylu, and Torsten Wedhorn for many helpful comments and conversations related to this paper. We also thank the referee for his/her careful reading of our manuscript
and for the insightful comments.


\section{The local Siegel-Weil Formula} \label{sect:localSW}


In this section we introduce the basic local setup and recall the local Siegel Weil formula, see Theorem \ref{theo:LocalSiegel-Weil}.
We make the involved constant explicit in Proposition \ref{prop:localSW}.

\subsection{The basic local set-up and local Whittaker functions}
\label{sect2.1}

Let $F$ be a local field  or the ring of adeles of a number field, and let $\psi$ be a non-trivial additive character of $F$ (or adele class character).  Let $P=NM$ be the standard Siegel parabolic subgroup of the symplectic group $\Sp_n(F)$ given by
\begin{align*}
M&=\left\{ m(a)=\zxz{a}{}{}{{}^ta^{-1}}\mid\; a\in \Gl_n(F)\right\},\\
N&=\left\{ n(b)=\zxz{1}{b}{}{1}\mid \; b\in \Sym_n(F)\right\}.
\end{align*}
We also denote
\[
w= \zxz {0} {-I_n} {I_n} {0}.
\]
Let $\Mp_{n, F}$ be the  metaplectic cover of $\Sp_n(F)$, identified with $\Mp_{n, F}=\Sp_n(F) \times \{ \pm 1 \}$ via the normalized Rao cocycle  $c_R(g_1, g_2)$ given  in \cite{Rao} (with the minor correction in \cite[Page 379]{KuSplit}):
$$
[g_1, \epsilon_1][g_2, \epsilon_2]=[g_1g_2, \epsilon_1 \epsilon_2 c_R(g_1, g_2)].
$$
For $g \in \Sp_n(F)$, we will simply denote $g=[g, 1]$.

Let $(V, Q)$ be a non-degenerate quadratic space over $F$ of dimension $l$. Then there is a Weil representation $\omega=\omega_{V, \psi}$ of $\Mp_{n, F}$ on $S(V^n)$ given by \cite[Page 400]{KuSplit}. In  particular,
\begin{align*}
\omega(n(b)) \phi(x) &= \psi(\tr(Q(x)b)) \phi(x),
\\
\omega(m(a))\phi(x) &= \chi_V(\det a) \gamma(\det a, \frac{1}2\psi)^{-l} (\det a, -1)_F^{\frac{l(l-1)}2} |\det a|^{\frac{l}2} \phi(xa),
\\
 \omega(w) \phi(x) &=\gamma(V^n) \int_{V^n} \phi(y) \psi(-\tr(x, y)) d_\psi y,
\end{align*}
where $d_\psi y$ is the self-dual Haar measure on $V$ with respect to $\psi$, and $\gamma(V^n) = \gamma(\psi\circ \det Q)^{-n}$. Here $\gamma(\psi)$ and $\gamma(a, \psi)$ (for $a \in F^\times$)  are the local Weil indices defined in \cite[Appendix]{Rao}, and
$$
\det Q =2^{-l} \det V =\det (\frac{1}2(e_i, e_j)) \in F^\times/(F^{\times})^2,
$$
for an $F$-basis $\{e_1, \dots, e_l\}$ of $V$. Finally,
\[
\chi_V(a) =(a, (-1)^{\frac{l(l-1)}2} \det V)_F
\]
is the quadratic character associated to $V$.
It is well-known that the Weil representation factors through $\Sp_n(F)$ when $l$ is even.  Since
$$
\gamma(a, \frac{1}2 \psi)^2 = (a, -1)_F,
$$
the formula for $\omega(m(a))\phi$ above  works for both even and odd $l$.
From now on, let $G =\Sp_n(F)$ or $\Mp_{n, F}$ depending on whether $n$ is odd or even, and let $P$ be the standard Siegel parabolic subgroup or the preimage of the standard Siegel parabolic subgroup.
If $F=\R$,  we let $K_G \subset G$ be the maximal compact subgroup given by either
\begin{align*}
\left\{ k=\zxz{a}{b}{-b}{a}\mid \;\uk=a+ib\in \Uni(n)\right\}\cong \Uni(n)
\end{align*}
or the inverse image of $\Uni(n)$ under the covering map (when  $G=\Mp_{n, \R}$).

For a character $\chi$ of $F^\times$, let $I(s, \chi) =\Ind_{P}^G \chi | \det |^s$ be the induced representation. A section $\Phi \in  I(s, \chi)$ satisfies
$$
\Phi(n(b) m(a)g, s)
 = \begin{cases}
 \chi(\det a) |\det a|^{s+\rho_n} \Phi(g, s) &\ff G=\Sp_n(F),
 \\
  \chi(\det a) \gamma(\det a, \frac{1}2 \psi)^{-1} |\det a|^{s+\rho_n} \Phi(g, s) &\ff G=\Mp_{n, F},
 \end{cases}
$$
where
\[
\rho_n = \frac{n+1}2.
\]
For a symmetric matrix $T \in \Sym_n(F)$, the Whittaker function of $\Phi$ with respect to $T$ is defined to be
\begin{equation} \label{eq:Whittaker}
W_{T}(g, s, \Phi) =\int_{\Sym_n(F)} \Phi(w n(b) g, s) \psi(-\tr(T b) )\,d_\psi b,
\end{equation}
where $d_\psi b$ is the self-dual Haar measure on $\Sym_n(F)$ with respect to the pairing $(b_1,b_2)\mapsto \psi(\tr (b_1b_2))$.
It has the transformation behavior
\begin{align}
\label{eq:i1}
W_T(n(b) m(a)g,s,\Phi)&=\psi(\tr(Tb))\chi(a)^{-1}
|a|^{\rho_n-s} \cdot
 \begin{cases}
  W_{{}^{t}aTa}(g,s,\Phi)&\ff G=\Sp_n(F),
 \\
 \gamma(a, \frac{1}2\psi) W_{{}^{t}aTa}(g,s,\Phi)&\ff G=\Mp_{n, F}.
 \end{cases}
\end{align}
Here we have shortened $\gamma(a, \frac{1}2\psi) = \gamma(\det a,\frac{1}2\psi)$  (and similarly for $\chi(a)$).  We remark that $\gamma(a, \frac{1}2\psi) =1$ when $\det a >0$ and $F=\R$.

 Let $s_{l, n} = \frac{l-n-1}2$. Then there is a $G$-equivariant map
\begin{align}
\label{eq:deflambda}
\lambda:  S(V^n) \rightarrow I(s_{l,n}, \chi_V),      \quad \lambda(\phi)(g) =\big(\omega(g) \phi\big)(0).
\end{align}
We will also denote by $\lambda(\phi) $ the associated standard section in $I(s, \chi_V)$ with $\lambda(\phi) (g, s_{l, n}) =\lambda(\phi)(g)$. Assume that $l=n+1$. Then a formal unfolding suggests that there is a Haar measure $dh$ on $H=\SO(V)$ such that  for all $\phi \in S(V^n)$
$$
O_T(\omega(g)\phi, dh)=C \cdot W_{T}(g, 0, \lambda(\phi))
$$
where $C$ is some constant which is independent of $T$ and $\phi$, and
$$
O_T(\phi, dh) = \int_{H(F)} \phi(h^{-1} x) dh
$$
if there is $x\in V^n$ with $Q(x) =T$ (otherwise set the orbital integral to be zero).
This is the content of the so-called local Siegel-Weil formula which we will describe in next two subsections. In particular, we will determine the constant $C$.

\subsection{Kudla's  local Siegel-Weil formula}

In this subsection we review the  local Siegel-Weil formula given in  \cite[Section 5.3]{KRY-book}, following a general result in \cite[Chapter 4]{Ra}. Let the notation be as in Section \ref{sect2.1}, and assume $\dim(V)=n+1$.
%
Let
\begin{equation}
Q: V^n \rightarrow \Sym_n(F),  \quad Q(x) = \frac{1}2 ( (x_i, x_j))
\end{equation}
be the moment map. Let $V_{\reg}^n$ be the subset of $x \in V^n$ with $\det Q(x) \ne 0$, and let $\Sym_{n}^{\reg}(F)$ be the subset of $T \in  \Sym_n(F)$ with $\det T\ne 0$. Then $Q$ induces a regular map from $V_{\reg}^n$ to $\Sym_n^{\reg}(F)$.

Put $a(n) =\frac{n(n+1)}2$. We let $\alpha$ be a gauge form on $V^n$, that is, a generator of $(\wedge^{2a(n)}V^n )^*$ (a top level differential of the topological vector space $V^n$), and let $\beta$ be a gauge form on $\Sym_n(F)$, i.e., a generator of $\left(\wedge^{a(n)}(\Sym_n(F))\right)^*$.

Fix an  $x=(x_1, \dots, x_n) \in V_{\reg}^n$ with $Q(x) =T$. If we  identify  the tangent space $T_x(V_{\reg}^n) $ with $V^n$, then the differential $dQ_x$ is given by
$$
dQ_x(v) = \frac{1}2((x, v) + (v, x)) \in \Sym_n(F), \quad  v \in V^n.
$$
Let
\begin{equation}
j_x:  \Sym_n(F) \rightarrow  V^n,  \quad j_x(u) =\frac{1}2 x Q(x)^{-1} u.
\end{equation}
Then $dQ_x \circ j_x (u) =u$, and  we have the decomposition
$$
T_x(V_{\reg}^n) = \IM(j_x) \oplus \ker(dQ_x).
$$
Now choose any $u=(u_1, \dots, u_{a(n)}) \in  (\Sym_n(F))^{a(n)}$ with $\beta(u) \ne 0$. We define an $a(n)$-form $\nu \in (\wedge^{a(n)}V^n)^*$ on $V_{\reg}^n$ as follows:  for any $t =(t_1, \dots, t_{a(n)}) \in (V^n)^{a(n)}$, we put
\begin{equation} \label{eq:KudlaForm}
\nu(t) = \alpha(j_x(u), t) \beta(u)^{-1}.
\end{equation}
This quantity is independent of the choice of $u$.
Then \cite[Lemma 5.3.1]{KRY-book} asserts that
\begin{align}
\alpha &= Q^*(\beta) \wedge \nu, \\
\label{eq:invariant}
\nu &=(h, g)^* \nu ,
\end{align}
for $ h \in \SO(V)$ and $g \in \GL_n$,
where $\SO(V) \times \GL_n$ acts on $V^n$ via $(h, g) x = hx g^{-1}$. Moreover,   $\nu$ defines a gauge form on $Q^{-1}(T)$ if we identify $ \ker dQ_x$ with   the tangent space $T_x(Q^{-1}(T))$ of $Q^{-1}(T)$. Finally, using the isomorphism
\begin{equation}
i_x: \SO(V) \rightarrow Q^{-1}(T),   \quad i_x(h) =h x
\end{equation}
(here $\dim V= n+1$ is critical to insure that the pointwise stabilizer $H_x$ of $x$ is trivial),
we obtain a gauge form $i_x^*(\nu)$ on $\SO(V)$, which we will still denote by  $\nu$ for simplicity. The key point (see \cite[Lemma 5.3.2]{KRY-book}) is that  this gauge form $\nu$ does not depend on $T$ or $x$, which can be  seen by \eqref{eq:invariant}.

  This gauge form $\nu$  gives a Haar measure $dh=d_\nu h$ on $\SO(V)$. Let $d_\alpha x$ be the Haar measure on $V^n$ associated to $\alpha$ and $d_\beta T$ be the Haar measure on $\Sym_n(F)$ associated to $\beta$, and let $d_\psi x$ and $d_\psi T$ be  the self-dual Haar measures  on $V^n$ and  $\Sym_n(F)$ with respect to $\psi$, respectively. Then there are constants $c(\alpha, \psi)$ and $c(\beta, \psi)$ such that
 \begin{equation} \label{eq:MeasureConstant}
 d_\alpha x = c(\alpha,\psi) d_\psi x,  \quad d_\beta T = c(\beta, \psi) d_\psi T.
 \end{equation}
Finally, we can state Kudla's  local Siegel-Weil formula, which is \cite[Proposition 5.3.3]{KRY-book} (although only stated for $n=2$ there, the proof goes through for general $n$ without any change).

\begin{theorem}[Local Siegel-Weil formula]
\label{theo:LocalSiegel-Weil} Given a gauge form $\alpha$ on $V^n$ and  a  gauge form $\beta$ on $\Sym_n(F)$, let $d_\nu h$ be the Haar measure on $H(F)$ associated to $\alpha$ and $\beta$ as above.
Then one has  for any $\phi \in S(V^n)$, $T \in \Sym_n^{\reg}(F)$,  and $g\in G$,
$$
 O_T(\omega(g)\phi, d_\nu h) =C(V, \alpha, \beta, \psi)\cdot   W_T(g, 0, \lambda(\phi)) ,
$$
where
$$
C(V, \alpha, \beta, \psi) = \frac{ c(\alpha, \psi)}{\gamma(V^{n}) c(\beta, \psi)},
$$
and $\gamma(V^n)=\gamma(V)^n$  by \cite[Lemma 3.4]{KuSplit}.
\end{theorem}

We remark that our $C(V, \alpha, \beta, \psi)$ is the reciprocal  of the same notation in \cite{KRY-book}.

\subsection{Explicit construction}

Let $\underline{e} =(e_1, \dots, e_{n+1})$ be an ordered basis of $V$ and put $J=Q(\underline{e}) = \frac{1}2 ( (e_i, e_j)) \in \Sym_{n+1}(F)$.
When $F$ is $p$-adic, let $L=\oplus \mathcal O_F e_j$ be the associated $\mathcal O_F$-lattice. Using this basis, we identify $V$ with $F^{n+1}$ (column vectors)  and $V^n$ with  $M_{n+1, n}$.

Let  $E_{ij}$  denote a matrix whose $(ij)$-entry is one and all other entries are zero (we do not identify the size of the matrix). Then  $\{ E_{ij},  1\le i \le n+1, 1 \le j \le n\}$ is a basis of $V^n$. Let
 $de_{ij}$ be its dual basis, and  let $\alpha =\bigwedge_{ij} de_{ij}$ be the gauge form on $V^n$ (up to sign, which does not affect the associated Haar measure) with
\begin{equation}
\alpha ( (E_{ij})) =\alpha (E_{11}, E_{12}, \dots, E_{n+1, n}) =1.
\end{equation}
Notice  that $Y_{ij} =E_{ij} + E_{ji}$ is a basis of $\Sym_n(F)$ ($1 \le i \le j \le n$), and let $dy_{ij} $ be its dual basis.
Let $\beta = \bigwedge_{ij} dy_{ij}$, then (up to sign)
\begin{equation}
\beta(Y_{11}, Y_{12}, \dots, Y_{n, n}) =1.
\end{equation}

\begin{proposition}
\label{prop:localSW} Let $J$,  $\alpha$, and $\beta$ be given  as above, and let $d_\nu h$ be the associated Haar measure on $H(F)=\SO(V)(F)$.
We take   $\psi(x)=e(x)=e^{2 \pi i x}$ when $F=\R$ and  assume that $\psi$ is unramified  when $F$ is $p$-adic. Then
$$
O_T(\omega(g)\phi, d_\nu h) = C(J) \cdot W_{T}(g, 0, \lambda(\phi))
$$
for all $\phi\in S(V^n)$ and $g \in G$.  Here
$$
C(J)= \gamma(V^{n})^{-1} |2|_F^{n+\frac{n(n-1)}4}  |\det (2J)|_F^{-n/2}.
$$
Finally,  when $F=\R$ and $V$ has signature $(p, q)$, then  $\gamma(V^{n}) =e(\frac{n(q-p)}{8})$.
\end{proposition}
By the proposition, we see that $d_\nu h$ depends only on $|\det(2J)|_F$.
For this reason, we will sometime denote $d_\nu h$ by $d_J h$ or $d_L h$ in the $p$-adic case. We also write $C(L)=C(J)$ in the $p$-adic case as $\det(2J) =\det L$.
\begin{proof}First assume that $F$ is $p$-adic.  Let $\co_F$ be   ring of integers  of $F$. Let $L=\oplus \co_F e_i  = \co_F^n  \subset V=F^n$ and $f=\cha(L^n) =\cha(M_{n+1, n}(\co_F) )\in S(V^n)$. Then
the Fourier transforms of $f$ with respect to $d_\alpha x$  and $d_\psi x$ are given by
\begin{align*}
\hat{f}_\alpha(X) &= \int_{M_{n+1, n}(\co_F)} \psi(-\tr (2\, {}^tX J Z)) \prod  dz_{ij} = \cha(L^{\prime, n})(X),
\\
\hat f_\psi(X)  &=\int_{M_{n+1, n}(\co_F)} \psi(-\tr (2 \, {}^tX J Z)) d_\psi Z = \cha(L^{\prime, n})(X) \vol(L^n, d_\psi x),
\end{align*}
where $L'$ is the dual lattice of $L$ with respect to $\psi$.  Since $d_\psi x$ is the self-dual Haar measure on $V$ with respect to $\psi$, one has
$$
\vol(L, d_\psi x) 
=|\det (2J)|_F^{\frac{1}2}  .
$$
Consequently, $c(\alpha, \psi) =   |\det (2J)|_F^{-n/2}$.

 Next, for  $t =(t_{ij}) \in \Sym_{n}(F)$, $d_{\beta} t = |2|_F^{-n} \prod dt_{ij}$. Let $f =\cha(\Sym_n(\co_F))$, then
\begin{align*}
\hat{f}_\beta(y)
 &=\int_{\Sym_n(\co_F)} \psi(-\tr(y t) ) |2|_F^{-n} \prod_{i} \psi(-y_{ii} t_{ii} )dt_{ii}  \prod_{i<j} \psi(-2 y_{ij} t_{ij}) dt_{ij}
 \\
 &=|2|_F^{-n }\prod_i \cha( \co_F)(y_{ii}) \prod_{i <j} \cha(\frac{1}2 \co_F)(y_{ij}).
\end{align*}
On the other hand, if $d_\psi t$ is the self-dual Haar measure on $\Sym_n(F)$  with respect to $\psi$, then
$$
\vol( \Sym_n(\co_F), d_\psi t) = |2|_F^{\frac{n(n-1)}{4}} ,
$$
and
\begin{align*}
\hat{f}_\psi(y) &= \int_{\Sym_n(\co_F)} \psi(-\tr(y t) )d_\psi t
\\
 &=\vol( \Sym_n(\co_F), d_\psi t)  \prod \cha(\co_F)(y_{ii}) \prod_{i <j} \cha(\frac{1}2 \co_F)(y_{ij}).
\end{align*}
So  $c(\beta, \psi) = |2|_F^{-n -\frac{n(n-1)}4} $. Now it is clear that $C(V, \alpha,\beta, \psi) =C(J)$ as claimed.

 Now assume that $F=\R$ and $\psi(x) =e(x)$. To compute the quantity $c(\alpha, \psi)$,  we write $J={}^t P \diag(a_1, \dots, a_{n+1})P$ and denote $|J| ={}^t P \diag(|a_1|, \dots, |a_{n+1}|) P$. We consider the Schwartz function on $M_{n+1, n}(\R)$ given by
$$
f(x) = e^{-2 \pi  \tr (x^t |J| x) }= e^{-2 \pi \sum |a_i|\tilde{x}_{ij}^2},
$$
where we write $Px = (\tilde{x}_{ij})$.
Then its Fourier transformation with respect to $d_\alpha x =\prod dx_{ij}$ is
\begin{align*}
\hat{f}_\alpha(x) = \int_{M_{n+1, n}(\R)} f(y) e(\tr (2\, {}^tx J y)) \prod dy_{ij} = 2^{-\frac{n(n+1)}2} |\det J|^{-\frac{n}2} f(x),
\end{align*}
and so $c(\alpha, \psi) =  |\det 2J|^{-\frac{n}2}$ as claimed.

 To compute $c(\beta, \psi)$, notice that
$d_\beta T=2^{-n} \prod_i dt_{ii} \prod_{i <j} dt_{ij}$ for $T=(t_{ij}) \in \Sym_n(\R)$, and consider the Schwartz function on $\Sym_n(\R)$
given by
$$
f(T) =e^{-\pi(\sum_i t_{ii}^2 + 2\sum_{i <j} t_{ij}^2)}.
$$
Then its Fourier transformation with respect to $d_\beta T$ is
\begin{align*}
\hat{f}_\beta(b) &= \int_{\Sym_n(\R)} f(T) \psi(-\tr (Tb))\, d_\beta T
 \\
  &=2^{-n} \prod_j  e^{-\pi b_{jj}^2} \int_\R e^{-\pi (t_{jj} + i b_{jj})^2} dt_{jj}
    \prod_{j < k} e^{-2 \pi b_{jk}^2} \int_{\R} e^{-2 \pi (t_{jk} + i b_{jk})^2 }dt_{jk}
    \\
  &= 2^{-n -\frac{n(n-1)}4} f(b).
\end{align*}
This shows the equality $c(\beta, \psi) = 2^{-n -\frac{n(n-1)}4}$. We have again $C(V, \alpha, \beta, \psi) =C(J)$ as claimed.  The formula for $\gamma(V^n)$ is given by $\beta_V(w)$ in \cite[(3.4)]{KuSplit}.
\end{proof}

The following proposition shows how to compute the Haar measure $d_\nu h$ in some cases and will be used in Section \ref{sect:ArithSW}.

\begin{proposition}
\label{prop:Measure}
Let $F$ be a $p$-adic local field with $p \ne 2$ and a uniformizer $\pi$, and let $\psi$ be an unramified additive character of $F$.   For a lattice $L$ over $\calO_F$, let  $K_L =\SO(L)$ be the stabilizer of $L$ in  $\SO(V)(F)$, where $V=L\otimes_{\mathcal O_F} F$. Let $d_L h$ be the Haar measure on $H(F)= \SO(V)(F)$ defined above.
\begin{enumerate}
\item[(1)]
When $L$ is unimodular of rank $n+1$, we have
$$
\frac{\vol(K_L, d_L h)}{C(L)} = W_T(1, 0, \lambda(\phi_L))
$$
for any unimodular  symmetric matrix $T \in \Sym_{n}(\mathcal O_F)$. Here $\phi_L=\cha(L^n)$.
\item[(2)]
Assume $L=L_1 \oplus L_0$ with $L_1$ unimodular of rank $n-1$ and $L_0=(\co_E, \pi \norm_{E/F})$, where $E$ is the unique unramified quadratic field extension.  Let $T=\diag(T_1, \pi)$ with $T_1 =\frac{1}2( (e_i, e_j))$ for some $\co_F$-basis $\{e_1, \dots, e_{n-1}\}$ of $L_1$. Then
$$
\frac{\vol(K_L, d_L h)}{C(L)} = W_T(1, 0, \lambda(\phi_L)).
$$
\end{enumerate}
In both cases, $C(L) =C(J) $ is given by Proposition  \ref{prop:localSW}.
\end{proposition}

\begin{proof}
We prove (2) using Proposition  \ref{prop:localSW} with $\phi=\phi_L$,  and leave the slightly easier (1) to the reader. Choose a basis $\{e_n, e_{n+1}\}$ of $L_0$ so that $Q(a e_n + be_{n+1}) =\pi (a^2 +\epsilon b^2)$ for some $\epsilon \in \co_F^\times$. Let $e = (e_1, \dots, e_n) \in L^n$, then $Q( e) =T$. We claim that
$$
K_T:=\{ h \in H(F):\,  h e  \in  L^n\} =K_L.
$$
Clearly, $K_L\subset K_T$. We just need to prove that $he_{n+1} \in L$ for $h \in K_T$. In this case,
 $hL_1 \subset L$ is unimodular, so
$L= hL_1 \oplus M_0$,  where $M_0=L\cap (hL_1)^\perp$ is a rank $2$ lattice  with $\det M_0=\det L_0= \epsilon \pi^2$. Write
 $$
 M_0=\mathcal O_F \tilde e_n + \mathcal O_F \tilde e_{n+1},  \quad  Q(x \tilde e_{n} + y \tilde e_{n+1}) = \epsilon_1 \pi^{a_1} x^2 + \epsilon_2 \pi^{a_2} y^2$$
  with $\epsilon_1 \epsilon_2 =\epsilon$ and non-negative integers $a_i$ satisfying $a_1 +a_2=2$. Since
$$
h e_n = x \tilde e_n + y \tilde e_{n+1} \in L \cap (h L_1)^\perp =M_0,
$$
we have $x, y \in \co_F $ and
$$\pi = \norm(e_n) = \norm(he_n) = \epsilon_1 \pi^{a_1} x^2 + \epsilon_2 \pi^{a_2} y^2,
$$
 which implies $a_1=a_2 =1$. Now write
$$
h e_{n+1} = a  \tilde e_n + b \tilde e_{n+1}, \quad a,  b \in F.
$$
Then  $$\epsilon \pi= \norm(e_{n+1}) =\norm(h e_{n+1})=  \pi (\epsilon_1  a^2 + \epsilon_2  b^2),$$
  i.e.,
$$
\epsilon_1 \epsilon =\norm_{E/F}( \epsilon_1 a + \sqrt{-\epsilon} b),
$$
which implies $ \epsilon_1 a + \sqrt{-\epsilon} b \in E^1$, which is integral over $\co_F$. So  $a, b \in \mathcal O_F$ and $h(e_{n+1}) \in L$.  This proves $K_T=K_L$.  Applying the local Siegel-Weil formula to $\phi=\phi_L$, we have then
$$
 \int_{H(F)} \phi_L(h^{-1}e) d_L x = C(L) \cdot W_T(1, 0, \lambda(\phi_L).
$$
The left hand side is equal to $\vol(K_T, d_Lx)$. So we have
$$
\vol(K_L, d_L h) = C(L) \cdot W_T(1, 0, \lambda(\phi_L)
$$
as claimed.
\end{proof}

We remark that the Whittaker functions involved in the above proposition have explicit formulas, see Section \ref{sect:ArithSW}.

Now we describe $\nu$ and $d_\nu h=d_J h$ more explicitly by choosing the basis $\underline{e}$ and thus $J$ nicely, i.e., we assume $J=\diag(a_1, \dots, a_{n+1})$. It will be used in next section.

Let $\mathfrak h= \mathfrak s \mathfrak o(V)$ be  the Lie algebra of $\SO(V)$. In terms of coordinates with respect to the basis $\underline{e}$, one has
$X=(x_{ij}) \in  \mathfrak h$ if and only if  $ {}^tX J + J X =0$, i.e., $a_i x_{ij} + a_j x_{ji}=0$. Hence we have the following lemma.

\begin{lemma}
Let  $X_{ij} = a_j E_{ij} - a_i E_{ji}$ for $1 \le  i <j \le n+1$. Then $\{ X_{ij} \}$ gives a basis of $\mathfrak h$ as an $F$-vector space.
\end{lemma}

\begin{proposition} \label{prop:GaugeForm}  Let the notation be as above. Then one has
$$
\nu(X_{12}, X_{13}, \dots,  X_{n, n+1}) = \pm 1.
$$
\end{proposition}
\begin{proof} We choose
$$
x = (e_1, \dots, e_n) = \begin{pmatrix} I_n \\ 0 \end{pmatrix}  \in V^n
$$
and $T= Q(x) = \diag(a_1, \dots, a_n)$. Then
$$
j_x(Y_{ij}) = \frac{1}2 x T^{-1} Y_{ij} = \frac{1}2 (a_i^{-1} E_{ij} + a_j^{-1} Y_{ji}).
$$
Recall that $i_x: H \rightarrow Q^{-1}(T), h \mapsto h x$.  Hence   the associated map on tangent space $di_x: \mathfrak h \rightarrow V^{a(n)}$ is given by
$$
di_x (X_{ij})= X_{ij} x
=  \begin{cases}
 X_{ij} &\ff  j\le n,
 \\
 -a_{i} E_{n+1, i}  &\ff j=n+1.
\end{cases}
$$
Therefore,
$$
\left(\bigwedge_{1 \le i<j \le n+1} (di_x)(X_{ij})\right)  \wedge \left(\bigwedge_{1 \le j \le n} j_x(Y_{ij})\right)
 = \pm  \bigwedge_{\substack{ 1 \le i \le n \\ 1 \le j \le  n+1}} E_{ij},
$$
and
\begin{align*}
\nu(\wedge X_{ij} ) &= \alpha\left( (\bigwedge_{1 \le i<j \le n+1} (di_x)(X_{ij}))  \wedge (\bigwedge_{1 \le j \le n} j_x(Y_{ij}))\right) \beta(\wedge Y_{ij})^{-1}
\\
&=\pm \alpha(\bigwedge_{\substack{ 1 \le i \le n \\ 1 \le j \le  n+1}} E_{ij}) =\pm 1.
\end{align*}
This concludes the proof of the proposition.
\end{proof}


\section{The local Siegel-Weil formula on a hermitian symmetric domain}

\label{sect:3}

Let $V$ be a quadratic space over $\R$ of signature $(m, 2)$, and let $H=\SO(V)$. Let $\calD$ be the corresponding hermitian domain, which we realize as the Grassmannian of oriented negative $2$-planes in $V$. 
The purpose of this section is to prove Proposition~\ref{prop:LocalSiegel-WeilR}, a variant of the archimedian local Siegel-Weil formula involving an integral over $\calD$.
Throughout this section we fix the additive character $\psi(x)=e(x)$ of $\R$ and assume that $n=m+1$. Recall that $\rho_n=\frac{n+1}2$.

Let $e,f\in V$ be isotropic vectors such that $(e,f)=1$, and let $V_0= (\R e+\R f)^\perp\subset V$. Then $V_0$ has signature $(m-1,1)$ and we have the Witt decomposition $V=V_0+\R e+\R f$.
The hermitian symmetric domain $\mathcal D$ can also  be realized as the tube domain
\begin{align}
\label{eq:defH}
\mathcal H = \{ z = x+ i y \in V_{0, \C}: \, Q(y) <0\},
\end{align}
via the isomorphism
\[
\calH\to \calD,\quad z\mapsto \R \Re(w(z))+\R\Im(w(z)),
\]
where
\[
w(z)= z+ e-Q(z)f\in V_\C.
\]
Then
$H(\R)$ acts on $\mathcal H$ by linear fractional transformations, characterized by
$$
hw(z) =j(h, z) \cdot w(h z)
$$
where $j(h,z)$ denotes the automorphy factor
$H(\R)\times \calH\to \C^\times$, $j(h,z)= (h w(z),f)$.

The map $z\mapsto w(z)$ can be viewed as a section of the tautological bundle over $\calD$. The Petersson norm of this section is $-\frac{1}{2}(w(z),\overline w(z) ) =-(y,y)$.
Hence
\begin{align}
\label{eq:defom}
\Omega= dd^c \log(-(y,y))
\end{align}
defines an invariant $(1,1)$-form on $\calH\cong \calD$, the first Chern form of the dual of the tautological bundle on $\calD$ equipped with the Petersson metric.  Here $d^c=\frac{1}{4\pi i}(\partial-\bar\partial)$.
According to
\cite[Proposition~4.11]{Ku2},
in the coordinates of $\calH$,
it is given by
\begin{equation} \label{eq:Omega}
\Omega= d d^c \log \left(-\frac{1}2 (\omega(z), \omega(\bar z)) \right)
    = -\frac{1}{2 \pi i} \left( -\frac{(y, dz) \wedge (y, d\bar z)}{(y, y)^{2} } + \frac{(dz, d\bar z)}{2 (y, y)}  \right).
\end{equation}
Moreover, it can be obtained from the Kudla-Millson form $\varphi_{KM}(x,z)$ (see \eqref{eq:defphikm}) by
\begin{align}
\label{eq:omega1}
\Omega=\varphi_{KM}(0,z),
\end{align}
an identity which we will only need in Section \ref{sect:LocalASW}.
Notice that $-\Omega$ is a K\"ahler form, and therefore
$(-\Omega)^m$ is a positive invariant top degree form  on $\mathcal H$.

 \begin{proposition}[Local Siegel-Weil formula on $\calD$]
\label{prop:LocalSiegel-WeilR}
Let $\phi_\infty(x, z) \in S(V_\R^n)\otimes C^\infty(\calD)$ with $\phi_\infty(hx, hz) = \phi_\infty(x, z)$ for all $z \in \mathcal D$, $x \in V_\R^n$ and $h \in H(\R)$.
Then $\lambda(\phi_\infty)$ is independent of $z$, and
$$
\int_{\mathcal D} \phi_\infty(x, z)\, \Omega^m =B_{n, \infty}\cdot  W_{T}(1, 0, \lambda(\phi_\infty)) ,
$$
with $T=Q(x)$ and
\begin{align*}
B_{n, \infty}
&=\frac{e(\frac{n^2+n-4}{8}) (n-1)! \prod_{k=1}^{n-1} \Gamma(\frac{n-k}2)}
       {2^{n-2} (2\pi)^{\frac{n(n+3)}4}}.
\end{align*}
 In particular, one  has  $B_{2, \infty} =\frac{i}{4 \sqrt{2} \pi^2}$ and
\begin{equation} \label{eq:Bquotient}
\frac{B_{n, \infty}}{B_{n-1, \infty} }= i^n \frac{\Gamma(\rho_n)}{(2 \pi)^{\rho_n}}.
\end{equation}
\end{proposition}

The basic idea of the proof is simple and natural: we relate the gauge form on the tangent space $\mathfrak p$ of $\D$ with the differential form  $\Omega^m$ precisely.  The actual calculation is a little long and technical, and can be skipped on first reading. We will also provide an alternative proof in Section \ref{sect:5.1}.

\subsection{The differential $\Omega^m $ and the gauge form $\nu$}
\label{sect:3.1}
Let $\underline{e}=(e_1, e_2, \dots, e_m, e_n, e_{n+1} )$ be an ordered basis of $V$ with  quadratic matrix
$$
J=\frac{1}2 ( (e_i, e_j))
=\kzxz {I_m} {0} {0} {-I_2}.
$$
We write $V_+ $ (respectively $V_-$) for be the subspace  generated by the $e_i$ with $1\le i \le m$ (respectively $i =n, n+1$). Let $K_{\pm} =\SO(V_{\pm})$, then $K_\infty=K_+\times K_-$ is a maximal connected compact subgroup of $H(\R)$. In the notation of the last section, $a_i=1$ for $1 \le i \le m$ and $-1$ for $i=n, n+1$.

Let $\mathfrak h= \mathfrak s \mathfrak o(V)$ be  the Lie algebra of $\SO(V)$. Then $X \in \mathfrak h$ if and only if
$X =\kzxz {X_1} {X_2} {{}^tX_2} {X_3}$ with $ {}^tX_1=-X_1 \in M_{m}$, $ {}^tX_3=-X_3 \in M_2$, and $X_2 \in  M_{m, 2}$.
In other words, one has a decomposition
$$
\mathfrak h=\mathfrak k_+ \oplus \mathfrak k_- \oplus \mathfrak p,
$$
where $\mathfrak k_\pm$  is the Lie algebra of $K_\pm=\SO(V_\pm)$
given by matrices satisfying ${}^tX_1=-X_1$ (respectively
${}^tX_3=-X_3$).

It is easy to see that  the gauge form $\nu$ given in Proposition \ref{prop:GaugeForm}  has the following decomposition (up to sign):
\begin{equation} \label{eq:LieDecomposition}
\nu = \nu_+ \wedge \nu_- \wedge \nu_{\mathfrak p},
\end{equation}
where $\nu_+$, $\nu_-$, and $\nu_{\mathfrak p}$ are  the gauge forms  on $K_+$, $K_-$, and on $H(\R)/K_\infty$, which are characterized by
\begin{align} \label{eq2.2}
&\nu_+(X_{12}, \dots, X_{m-1, m}) =1, \notag
\\
 &\nu_-(X_{n, n+1})=1, \quad \hbox{ and }
 \\
 &\nu_{\mathfrak p}(X_{1n}, X_{1,n+1}, \dots, X_{mn}, X_{m, n+1})=1. \notag
\end{align}

Now we deal with the relation between $\Omega^m$ and $\nu_{\mathfrak p}$. We use a tube domain realization for $\calD$ as above. To this end we define a different basis $\underline{e}'$ of $V$ as follows.
Let $e=\frac{1}2(e_1 + e_{n+1})$, $f= \frac{1}2(e_1-e_{n+1})$, and $\underline{e}'=(e_n,  e_2\, \dots, e_m, e,f)$. Its associated matrix is
$$
J'= \begin{pmatrix}
-1 &0 &0 &0
\\
0 &I_{m-1} &0 &0
\\
0 &0 &0 &\frac{1}2
\\
0&0 &\frac{1}2&0
\end{pmatrix}.
$$
We put
\begin{align}
\label{eq:V0coord}
V_0 &=\oplus_{i=2}^n \R e_i \cong \R^m,
\\
\nonumber
z &= \sum z_i e_i = {}^t(z_n, z_2, \dots, z_m)
\end{align}
with quadratic form $Q(z) =\sum \epsilon_i z_i^2$ with $\epsilon_i =\pm 1$ depending on whether $i <n$ or $i \ge n$.
Then $V=  V_0 \oplus \R e \oplus \R f$ is a Witt decomposition as considered before. We write $\calH$ for the corresponding tube domain realization of $\calD$ as in \eqref{eq:defH}. We will also identity $V$ with $\R^{n+1}$ and $V^n$ with $M_{n+1, n}(\R)$ with respect to the basis $\underline{e}'$:
\begin{align*}
  v&=  \sum_{2 \le i \le n}  z_i e_i + v_0 e + v_1 f = {}^t( z_n,  z_2,\dots, z_m,v_0,  v_1) = [v]_{\underline{e}'},\\
  x &= (\tilde x_1, \dots, \tilde x_n) =[x]_{\underline{e}'} \in M_{n+1, n}(\R).
 \end{align*}
Similarly, we will use $[v]_{\underline{e}}$ and $[x]_{\underline{e}}$ to denote the coordinates of $v$ and $x$  with respect to the basis $\underline{e}$ when necessary. For $\gamma \in H(\R)$, we denote $[\gamma]_{\underline{e}}$ and $[\gamma]_{\underline{e}'} $ for its coordinates with respect to the bases $\underline{e}$ and $\underline{e}'$ respectively. Then one has
\begin{equation} \label{eq:transfer}
[v]_{\underline{e}'} = A [v]_{\underline{e}}, \quad [\gamma]_{\underline{e}'} = A [\gamma]_{\underline{e}} A^{-1}, \quad  \hbox{ and  }
A = \begin{pmatrix}
0 &0 &1&0
\\
0 &I_{m-1} &0 &0
\\
1 &0 &0 &1
\\
1 &0 &0 &-1
\end{pmatrix}.
\end{equation}

We now compute the action of $H(\R)$ on $\calH$ more explicitly.
For $h \in H(\R)$, write
$$
[h]_{\underline{e}'} = \begin{pmatrix}
\tilde h_{11} &\tilde H_{12} &\tilde h_{13} &\tilde h_{14}
\\
\tilde H_{21} &\tilde H_{22} &\tilde H_{23} &\tilde H_{24}
\\
\tilde h_{31} &\tilde H_{32} &\tilde h_{33} &\tilde h_{34}
\\
\tilde h_{41} &\tilde H_{42} &\tilde h_{43} &\tilde h_{44}
\end{pmatrix},
$$
where all the $\tilde H_{ij}$ are matrices, all  $\tilde h_{ij}s$ are numbers, and $\tilde H_{22}$ is a square matrix of order $m-1$.
Then  for $z={}^t(z_n, z_2, \dots , z_m) = \begin{pmatrix} z_n \\ \underline z \end{pmatrix} \in \mathcal H$, we have
\begin{align*}
h (z) &= j(h, z)^{-1} \begin{pmatrix} z_n \tilde{h}_{11} + \tilde H_{12} \underline z + \tilde h_{13} - \tilde h_{14} Q(z)\\  z_n \tilde H_{21}+ \tilde H_{22} \underline z + \tilde H_{23} - \tilde H_{24} Q(z) \underline z \end{pmatrix}, \\
j(h, z) &= z_n \tilde{h}_{31} + \tilde H_{32} \underline z + \tilde h_{33} - \tilde h_{34} Q(z).
\end{align*}
Fixing the base point $z=   i e_n    \in \mathcal H$, we have the isomorphism
\begin{align}
l_z: \,  H(\R)/K_\infty \cong \mathcal H,  \quad
 h \mapsto h(z)= \begin{pmatrix} \tilde h_{11}i+\tilde h_{13} -\tilde h_{14} \\ \tilde H_{21}i +\tilde H_{23} -\tilde H_{24} \end{pmatrix}(\tilde h_{31} +\tilde h_{33}i -\tilde h_{34})^{-1}.
\end{align}
This induces an isomorphism between  $\mathfrak  p$ and  the tangent space $T_z(\mathcal H) \cong V_{0, \C}$ (extending to the tangent bundle, too):
\begin{align}
dl_z: \,   \mathfrak p  \cong V_{0, \C},  \quad
 X\mapsto X(z) =\begin{pmatrix}\tilde x_{13} -\tilde x_{14} +\tilde x_{31}  \\ \tilde X_{23} -\tilde X_{24} \end{pmatrix}
               + i \begin{pmatrix} \tilde x_{11} -\tilde x_{33} +\tilde x_{34}  \\ \tilde X_{21} \end{pmatrix},
\end{align}
where $\tilde x_{i}$ and $\tilde X_{ij}$ are the coordinates of  $X$ with respect to $\underline{e}'$ just as for $h$. In terms of the coordinates with respect to $\underline{e}$, one has
$$
[X]_{\underline{e}} =
 \begin{pmatrix}
   0 &0 & x_{13}  &x_{14}
   \\
   0  &0_{m-1, m-1}   &X_{23} &X_{24}
   \\
   x_{13}  & X_{23}^t &0  &0
   \\
  x_{14}   &X_{24}^t&0 &0
  \end{pmatrix}  \in \mathfrak p,
$$
and by a direct
direct calculation using (\ref{eq:transfer}) we obtain
\begin{equation}
dl_z (X) = \begin{pmatrix} -i (x_{14} + i x_{13}) \\ X_{24} + i X_{23}  \end{pmatrix}.
\end{equation}
So we have proved the following lemma.

\begin{lemma} \label{lem:Omega}  The isomorphism  $dl_z$ induces
$$
(dl_z)^* (dx_2\wedge dy_2 \wedge \cdots \wedge dx_n \wedge dy_n) = \pm   \nu_{\mathfrak p},
$$
where $(dx_j +i dy_j)$ is the dual $\C$-basis of the basis $(e_j)_{2 \le j \le n}$ of $V_{0, \C}$.
\end{lemma}

Recall the formula \eqref{eq:Omega} for the $H(\R)$-invariant $(1,1)$-form $\Omega$.

\begin{lemma}
Using the above notation, we have
$$
(-\Omega)^m=\frac{m!}{(2 \pi )^m} \frac{\bigwedge_{j=2}^n dx_j \wedge dy_j}{(-Q(y))^m}.
$$
\end{lemma}

\begin{proof}
Using the coordinates of \eqref{eq:V0coord}, one sees
\begin{align*}
\Omega&= -\frac{1}{2 \pi i} \left( -\frac{1}{Q(y)^2} \sum \epsilon_i \epsilon_j y_i y_j dz_i \wedge d\bar{z}_j  + \frac{1}{2 Q(y)} \sum \epsilon_i  dz_i \wedge d \bar{z}_i\right)
\\
&=-\frac{1}{2 \pi i} \left(-\frac{1}{Q(y)^2} \alpha + \frac{1}{2Q(y)} \beta\right),
\end{align*}
where $\alpha$ and $\beta$ have the obvious meanings. Notice that
$$
\alpha^2 =\sum_{i, j,k, l=2}^{n} \epsilon_i \epsilon_j \epsilon_k \epsilon_l y_iy_j y_k y_l \cdot \alpha(i, j, k, l)
$$
with
$$
 \alpha(i, j, k, l) =dz_i \wedge d\bar{z}_j \wedge dz_k \wedge d\bar{z}_l.
$$
Since $\alpha(i, j, k, l) = -\alpha(k, j, i, l)$, we have
$\alpha^2=0$. This implies
$$
\Omega^m =\left(-\frac{1}{4 \pi i Q(y)}\right)^m (\beta^m -\frac{2m}{Q(y)} \alpha \wedge \beta^{m-1}).
$$
It is easy to check that
\begin{align*}
\beta^m &=- m! \bigwedge_{j=2}^n dz_j \wedge d\bar z_j,
\\
 \beta^{m-1} &= -(m-1)!  \sum_{l=2}^n \epsilon_l \beta_l,
\end{align*}
where $\beta_l $ is  $\bigwedge_{j=2}^n dz_j \wedge d\bar z_j$ with $dz_l \wedge d \bar z_l$ missing. So
$$
\alpha \wedge \beta^{m-1} = -(m-1)! \sum_{l=2}^n \epsilon_l y_l^2 \bigwedge_{j=2}^n dz_j \wedge d\bar z_j = -(m-1)! Q(y) \bigwedge_{j=2}^n dz_j \wedge d\bar z_j,
$$
and therefore
$$
\Omega^m=\left(-\frac{1}{4 \pi i Q(y)}\right)^m m! \bigwedge_{j=2}^n dz_j \wedge d\bar z_j=\frac{m!}{(2 \pi Q(y))^m} \bigwedge_{j=2}^n dx_j \wedge dy_j
$$
as claimed.
\end{proof}

It is well-known that  $\frac{\bigwedge_{j=2}^n dx_j \wedge dy_j}{(-Q(y))^m}$ is the $H(\R)$-invariant Haar measure on $\mathcal H$ associated to $\bigwedge_{j=2}^n dx_j \wedge dy_j$. So we obtain  the following proposition from the above two lemmas.

\begin{proposition} \label{prop:Omega}
Let the notation be as above and $z=ie_n \in \mathcal H$. Then
$$
(d l_z)^* (-\Omega)^m =\pm \frac{m!}{(2 \pi )^m}  \nu_{\mathfrak p}.
$$
\end{proposition}

\begin{proof}[Proof of Proposition \ref{prop:LocalSiegel-WeilR}]
First,  let $z =ie_n \in \mathcal H$ as Proposition~\ref{prop:Omega}, and let $\nu=\nu_+ \wedge \nu_- \wedge \nu_{\mathfrak p}$ be as in (\ref{eq:LieDecomposition}), and let $d_\nu h$, $dh_+$, $dh_-$, and $d_{\mathfrak p}h$ be the associated Haar measures. Then, by
 Proposition~\ref{prop:Omega}, we have
\begin{align*}
\int_{H(\R)} \phi_\infty(h^{-1}x, z) \,d_\nu h
 &=  \vol(K_+, dh_+) \vol(K_-, dh_-) \int_{H(\R)/K_\infty} \phi_\infty(x, hz) \,d_{\mathfrak p} h
 \\
  &= \vol(K_+, dh_+) \vol(K_-, dh_-) \frac{ (2\pi)^m}{m!} \int_{\mathcal H} \phi_\infty(x, z) (-\Omega)^m
  \\
  &= (-1)^m \vol(K_+, dh_+) \vol(K_-, dh_-) \frac{ (2\pi)^m}{m!} \int_{\mathcal D} \phi_\infty(x, z) \,\Omega^m.
\end{align*}
On the other hand, Proposition \ref{prop:localSW} gives
$$
\int_{H(\R)} \phi_\infty(h^{-1}x, z) \,d_\nu h =C(\diag(I_m, -I_2))\cdot    W_{T, \infty} (1, 0, \lambda(\phi_\infty)).
$$
Consequently,
$$
B_{n, \infty}=\frac{(-1)^m m! C(\diag(I_m, -I_2))}{  (2 \pi)^m \vol(K_+, dh_+) \vol(K_-, dh_-)}.
$$
Applying Proposition \ref{prop:localSW} to $K_+=\SO(V_+)$ and $\phi_\infty =e^{-2\pi \tr Q_+(x)}$, one sees by  Proposition~\ref{prop:whitt0} that
\begin{align*}
\vol(K_+, dh_+) e^{-2 \pi (m-1) }   &=C(I_{m}) W_{I_{m-1}, \infty}(1, 0, \Phi_{\frac{m}2})
\\
 &=C(I_m) \frac{ (-2\pi i)^{\frac{m(m-1)}2}}{\Gamma_{m-1}(\frac{m}2)} e^{-2 \pi (m-1)},
\end{align*}
where $\Gamma_n(s)$ is given by (\ref{eq:siegel}).
We obtain
$$
\vol(K_+, dh_+)= C(I_m) \frac{ (-2\pi i)^{\frac{m(m-1)}2}}{\Gamma_{m-1}(\frac{m}2)}
  =\frac{ 2^{m-1} \pi^{\frac{m(m+1)}4} }{\prod_{k=0}^{m-1} \Gamma(\frac{m-k}2)}.
$$
Similarly, $\vol(K_-, dh_-) =2 \pi$. Plugging these formulas into that of $B_{n, \infty}$, one proves Proposition \ref{prop:LocalSiegel-WeilR}.
\end{proof}

We remark that the above calculation of $\vol(K_+, dh_+)$ has the following well-known formula as a consequence.

\begin{corollary}
Let $l\ge 1$ be an integer, and let
$$
\SO_l(\R)=\{ g \in \GL_l(\R):\, g \, {}^tg =I_n, \det g =1\}
$$
be the standard special orthogonal group. Let $\nu_l$ be the gauge form defined as $\nu_+$ for $K_+=\SO_m(\R)$, and let $dh_l$ be the associated  Haar measure on $\SO_l(\R)$.  Then
$$
\vol(\SO_l(\R), dh_l) =\frac{ 2^{l-1} \pi^{\frac{l(l+1)}4} }{\prod_{k=0}^{l-1} \Gamma(\frac{l-k}2)}.
$$
\end{corollary}

\section{Asymptotic properties of Whittaker functions}
\label{sect:4}


Throughout this section we consider the local field $F=\R$, the additive character $\psi(x)=e(x)$, and the group $G=\Sp_n(\R)$ or $\Mp_{n,\R}$.
We investigate the asymptotic behavior of the archimedian Whittaker function for $G$ as one of the radial parameters of the Levi subgroup $M$ goes to $\infty$. The main results are
Theorem~\ref{thm:etalimnew} and Corollary \ref{cor:wderivasy}.
Our analysis is based on Shimura's work on confluent hypergeometric functions \cite{Sh}.
%
We fix a quadratic character $\chi$ of $\R^\times$ and an half integer $\kappa \in \rho_n+\Z$ (not necessarily equal to $\rho_n$) satisfying the compatibility condition
$$
\begin{cases}
&(-1)^\kappa = \chi(-1)\quad  \ff n \equiv 1 \pmod 2,
\\
& \kappa \equiv \frac{1}2\chi(-1)  \quad \ff n \equiv 0 \pmod 2.
\end{cases}
$$
We also fix a matrix $T \in \Sym_n(\R)$.

\subsection{Basic properties of archimedian Whittaker functions}

Let  $\Phi=\Phi_\kappa\in I_n(s,\chi)$ be  the weight $\kappa$ standard section, that is, the unique function in $I_n(s,\chi)$ whose restriction to $K_G$ is the character $\det(\uk)^{\kappa}$. Then
the Whittaker integral \eqref{eq:Whittaker} can be expressed in terms of Shimura's confluent hypergeometric function.
As in \cite[Lemma 9.3]{Ku1}, the following result can be proved. 

\begin{lemma}
\label{lem:whitshim}
Assume that $\det(T)\neq 0$.
If $a\in \GL_n^+(\R)$ and $y=a\,{}^ta$, then
\[
 W_T(m(a),s,\Phi_\kappa)=C_{n,\infty} \cdot
|a|^{s+\rho_n}\xi(y,T,\alpha,\beta),
\]
where
\note{}
\begin{align*}
\xi(y,T,\alpha,\beta)&= \int_{\Sym_n(\R)} \det(x+iy)^{-\alpha}
\det(x-iy)^{-\beta} e(-\tr(Tx))\, dx
\end{align*}
denotes Shimura's confluent hypergeometric function of matrix argument
\cite[(1.25)]{Sh} with
\begin{align*}
\alpha&=\frac{1}{2}\left( s+\rho_n+\kappa\right),\\
\beta&=\frac{1}{2}\left( s+\rho_n-\kappa\right).
\end{align*}
Here $dx=\bigwedge_{i\leq i} dx_{ij}$ is the Lebesgue measure on $\Sym_n(\R)\cong\R^{\frac{n(n+1)}{2}}$, and $C_{n,\infty}= 2^{\frac{n(n-1)}{4}}$.
\end{lemma}

The normalizing factor $C_{n,\infty}$ comes from comparing the measures $dn$ and $dx$.
Recall that the Siegel gamma function of genus $n$ is defined by
\begin{align}
\label{eq:siegel}
\Gamma_n(s):&=\int\limits_{\substack{x\in \Sym_n(\R)\\ x>0}} e^{-\tr(x)} \det(
x)^{s-\rho_n}\, dx\\
\nonumber
&=\pi^{\frac{n(n-1)}{4}} \prod_{k=0}^{n-1} \Gamma(s-\frac{k}{2}).
\end{align}
Following Shimura, we define another special function by
\begin{align}
\label{eq:eta}
\eta(y,T, \alpha,\beta)= \int\limits_{\substack{u\in \Sym_n(\R)\\ u>- T\\u>T}}
 e^{-\tr(uy)} \det(u+T)^{\alpha-\rho_n}\det(u-T)^{\beta-\rho_n}\, du.
\end{align}
For all regular $T$, by \cite[Remark 4.3]{Sh}, the integral converges when $\Re(\alpha)>\rho_n-1$ and  $\Re(\beta)>\rho_n-1$.
According to \cite[(1.29)]{Sh}, we have
\begin{align}
\label{eq:xieta}
\xi(y,T,\alpha,\beta)&= \frac{i^{n(\beta-\alpha)}2^{-n(\rho_n-1)}(2\pi)^{n\rho
_n}}{\Gamma_n(\alpha)\Gamma_n(\beta)}\cdot \eta(2y,\pi T,\alpha,\beta),
\end{align}
and therefore
\begin{align}
\label{eq:weta}
W_T(m(a),s,\Phi_\kappa)&=c_n(\alpha,\beta) \cdot
|a|^{s+\rho_n}\cdot  \eta(2y,\pi T,\alpha,\beta),\quad\text{where}\\
\nonumber
c_n(\alpha,\beta)&=C_{n,\infty}\cdot \frac{i^{n(\beta-\alpha)}2^{-n(\rho_n-1)}(2\pi)^{n\rho_n
}}{\Gamma_n(\alpha)\Gamma_n(\beta)}.
\end{align}

\begin{lemma}
\label{lem:etatrafo}
If $S\in \GL_n^+(\R)$, we have
\begin{align*}
\eta({}^tS gS,h,\alpha,\beta)&= |S|^{2(\rho_n-\alpha-\beta)}
\eta( g,Sh\,{}^tS,\alpha,\beta),\\
W_T(m(a), s, \Phi_\kappa)&= |S|^{\rho_n-s} W_{{}^tSTS}(m(S^{-1}a), s, \Phi_\kappa).
\end{align*}
 \end{lemma}

\begin{proof}
The first assertion follows from (3.1.K) of \cite{Sh}. The second assertion follows from this by means of \eqref{eq:weta}.
\end{proof}

The special values of Eisenstein series and Whittaker functions at $s=0$ will be of particular interest. Here we collect the facts that we will require.

\begin{proposition}
\label{prop:whitt0}
Assume that $\det(T)\neq 0$ and that $\kappa=\rho_n$.
\begin{enumerate}
\item[(i)] If $\sig(T)=(n-j,j)$ with $0\leq j\leq n$, then
\[
\ord_{s=0} W_T(m(a),s,\Phi_\kappa)\geq \Big\lfloor \frac{j+1}{2}\Big\rfloor .
\]
\item[(ii)] If $\sig(T)=(n,0)$, then
\[
W_T(m(a),0,\Phi_\kappa) = \frac{(-2\pi i )^{n\kappa}2^{-\frac{n(n-1)}{4}}}{\Gamma_n(\kappa)} (\det y)^{\kappa/2} e^{-2\pi \tr Ty}.
\]
\end{enumerate}
\end{proposition}

\begin{proof}
According to \cite[Theorem 4.2]{Sh}, the function
\[
\Gamma_{n-j}(\beta-\frac{j}{2})^{-1}\Gamma_j(\alpha-\frac{n-j}{2})^{-1}
\eta(2y,\pi T, \alpha,\beta)\]
is holomorphic for $(\alpha,\beta)\in \C^2$.
Hence, in view of \eqref{eq:weta}, $W_T(m(a),s,\Phi_\kappa)$ is equal to a holomorphic function in a neighborhood of $s=0$ times the gamma factor
\[
\frac{\Gamma_{n-j}(\beta-\frac{j}{2})}{\Gamma_n(\beta)}.
\]
Therefore, the first assertion follows from \eqref{eq:siegel} by working out the vanishing order of this gamma factor.

To prove (ii), we use  (4.35.K), (4.12.K),  and (4.6.K) of \cite{Sh} to see that for $\sig(T)=(n,0)$ we have
\[
\eta(g,h,\rho_n, \beta)=\Gamma_n(\beta)\det(g)^{-\beta} e^{-\tr(gh)}.
\]
By means of \eqref{eq:weta} we get
\begin{align*}
W_T(m(a),0,\Phi_\kappa)&=c_n(\rho_n,\beta)|a|^{s+\rho_n} \eta(2y,\pi T,\rho_n,\beta) \mid_{s=0}\\
&=\frac{(-2\pi i )^{n\rho_n}2^{-\frac{n(n-1)}{4}}}{\Gamma_n(\rho_n)} (\det y)^{\rho_n/2} e^{-2\pi \tr Ty},
\end{align*}
where $y=a\,{}^ta$.
%
This proves the proposition.
\end{proof}



\begin{remark}  \label{rem2.11}
Assume that $\sig(T)=(n,0)$ and that $\kappa=\rho_n$.
Then with the constant $B_{n,\infty}$ of Proposition \ref {prop:LocalSiegel-WeilR} we have
$$
 B_{n, \infty} \cdot  W_{T}(1, 0, \Phi_\kappa) =-2e^{-2 \pi \tr T}.
$$
\end{remark}

Later we will also need the following lemmas.

\begin{lemma}
\label{lem:trafo}
If $f$ is a measurable function on $\Sym_n(\R)$ and $a\in \GL_n(\R)$, we have
\[
\int_{\Sym_n(\R)} f(ab\,{}^ta)\, db = |a|^{-2\rho_n} \int_{\Sym_n(\R)} f(b)\, db.
\]
\end{lemma}

\begin{lemma}
\label{lem:gaussian}
If $S\in \Sym_n(\R)$ is positive definite,
we have
\[
\int\limits_{\substack{v\in \R^{n}}}
e^{-{}^tv \,S v } \, dv = \pi^{n/2}\det(S)^{-1/2}.
\]
Here $dv$ denotes the Lebesgue measure on $\R^n$.
\end{lemma}

\begin{lemma}
\label{lem:block}
Let $u=\zxz{u_1}{u_{12}}{{}^t u_{12}}{u_2}$ be a symmetric block matrix.
Then the following are equivalent:
\begin{enumerate}
\item $u>0$,
\item $u_1>0$ and $u_2> {}^t u_{12} \,u_1^{-1} u_{12}$,
\item $u_2>0$ and $u_1> u_{12}\, u_2^{-1} \,{}^t u_{12}$.
\end{enumerate}
In this case we have
\[
\det(u) = \det(u_1)\det (u_2- {}^t u_{12} u_1^{-1} u_{12})= \det(u_2)\det(u_1-u_{12} u_2^{-1} \,{}^t u_{12}).
\]
\end{lemma}

\begin{proof}
This is direct consequence of the Jacobi decompositions
\begin{align*}
\zxz{u_1}{u_{12}}{{}^t u_{12}}{u_2} &= \zxz{1}{0}{{}^t(u_1^{-1}u_{12})}{1}
\zxz{u_1}{0}{0}{u_2- {}^t u_{12} u_1^{-1} u_{12}}\zxz{1}{u_1^{-1}u_{12}}{0}{1},\\
\zxz{u_1}{u_{12}}{{}^t u_{12}}{u_2} &= \zxz{1}{u_{12} u_2^{-1}}{0}{1}
\zxz{u_1-u_{12} u_2^{-1} \,{}^t u_{12}}{0}{0}{u_2}\zxz{1}{0}{{}^t(u_{12} u_2^{-1})}{1},
\end{align*}
whenever the inverses make sense.
See also \cite[Lemma 2.1]{Sh}.
\end{proof}

\subsection{Asymptotic properties}
\label{sect:4.2}

Here we investigate the asymptotic behavior of the Whittaker function $W_T(g,s,\Phi_\kappa)$.
We assume that $T\in \Sym_n(\R)$ with $\det(T)\neq 0$, and
$a\in \GL_n^+(\R)$. We put $y=a\,{}^ta$.
Recall that
\[
\eta(y,T, \alpha,\beta)= \int\limits_{\substack{u\in \Sym_n(\R)\\ u>- T\\u>T}}
 e^{-\tr(uy)} \det(u+T)^{\alpha-\rho_n}\det(u-T)^{\beta-\rho_n}\, du.
\]
We write $T$ and the variable of integration $u$ in block form as
\begin{align}
\label{eq:varblock}
T=\zxz{T_1}{T_{12}}{{}^tT_{12}}{T_2}, \qquad u=\zxz{u_1}{u_{12}}{{}^tu_{12}}{u_2}
\end{align}
with $T_1\in \R$, $T_2\in \Sym_{n-1}(\R)$, and $T_{12}\in \R^{1\times(n-1)}$, and analogously for $u$.

\begin{theorem}
\label{thm:etalimnew}
Let  $y =\kzxz {y_1} {y_{12}} {{}^t y_{12}} {y_2}\in \Sym_{n}(\R)$ be a positive definite matrix in block form as in \eqref{eq:varblock}. If $T_1\leq 0$ we have
\begin{align*}
&\lim_{y_1\to \infty} e^{T_1y_1} y_1^\beta \cdot \eta^{(n)}(y,T,\alpha,\beta)=0.
\end{align*}
If $T_1>0$ we have
\begin{align*}
&\lim_{y_1\to \infty} e^{T_1y_1} y_1^\beta \cdot \eta^{(n)}(y,T,\alpha,\beta)\\
&=e^{-2T_{12}{}^ty_{12}+\tr(\tilde{T}_2-T_2)y_2}
\cdot \Gamma(\beta+1-\rho_n) \pi^{\frac{n-1}{2}} (2T_1)^{\alpha-\rho_
n}\eta^{(n-1)}(y_2,\tilde{T}_2,\alpha-\frac{1}{2},\beta),
\end{align*}
where
\begin{align*}
\tilde{T}=\zxz{\tilde{T}_1}{0}{0}{\tilde{T}_2}=\zxz{T_1}{0}{0}{T_2-{}^tT_{12}T_1^{-1}T_{12}}.
\end{align*}
Here we have added a superscript to $\eta$ to indicate in which genus  it is considered.
\end{theorem}

\begin{remark}
In the case $n=1$ the function $\eta^{(0)}$ is to be interpreted as the constant function with value $1$. Then the theorem states
\begin{align}
\label{eq:etaasy1}
&\lim_{y_1\to \infty} e^{T_1y_1} y_1^\beta \cdot \eta^{(1)}(y,T,\alpha,\beta)=\begin{cases} 0,&\text{if $T_1 < 0$,}\\
\Gamma(\beta)(2T_1)^{\alpha-\rho_1},&\text{if $T_1>0$.}
\end{cases}
\end{align}
On the other hand,
for $T\in \R^\times$ and $y\in \R_{>0}$ we have
\begin{align*}
\eta^{(1)}(y,T,\alpha,\beta) =
e^{-|T|y} \cdot |2T|^{\alpha+\beta-1}\cdot
\begin{cases}
\Gamma(\alpha) U(\alpha,\alpha+\beta,
2|T|y),&\text{if $T<0$,}\\
\Gamma(\beta) U(\beta,\alpha+\beta,2|T|y),&\text{if $T>0$,}
\end{cases}
\end{align*}
where $U(a,b,z)$ denotes Kummer's confluent hypergeometric function, see \cite[(13.1.3)]{AS}. The asymptotic behavior of the Kummer function
$U(a,b,y)= y^{-a}+O(y^{-a-1})$ as $y\to \infty$
(see e.g.~\cite[(13.5.2)]{AS}) matches with \eqref{eq:etaasy1}.
\end{remark}

\begin{proof}[Proof of Theorem \ref{thm:etalimnew}.]
\emph{Step 1.}
We first consider the case where $\Re(\alpha)>\rho_n$ and $\Re(\beta)>\rho_n-1/2$.
We put
\[
y'= \zxz{y_1^{-1/2}}{}{}{1} y \zxz{y_1^{-1/2}}{}{}{1} = \zxz{1}{y_{12}/y_1^{1/2}}{{}^ty_{12}/y_1^{1/2}}{y_2}.
\]
By means of Lemma \ref{lem:trafo} we rewrite the integral as follows:
\begin{align*}
\eta(y,T,\alpha,\beta)&= e^{-\tr Ty}
 \int\limits_{\substack{u\in \Sym_n(\R)\\ u+2T>0\\u>0}}
 e^{-\tr(uy)} |u+2T|^{\alpha-\rho_n}|u|^{\beta-\rho_n}\, du\\
&=y_1^{-\beta} e^{-\tr Ty}\\
&\phantom{=}{}\times
 \!\!\int\limits_{\substack{\kzxz{u_1/y_1}{u_{12}/y_1^{1/2}}{{}^tu_{12}/y_1^{1/2}}{u_2} +2T>0\\u>0}}\!\!
 e^{-\tr(uy')} \left|\kzxz{u_1/y_1}{u_{12}/y_1^{1/2}}{{}^tu_{12}/y_1^{1/2}}{u_2}+2T\right|^{\alpha-\rho_n}\cdot
|u|^{\beta-\rho_n}\, du.
\end{align*}
Here and throughout the proof of the theorem we briefly write $|u|$ for the determinant of $u$.
In view of Lemma~\ref{lem:block}, we obtain
\begin{align}
\label{eq:lebes}
&\eta(y,T,\alpha,\beta)\\
\nonumber
&=
y_1^{-\beta} e^{-\tr Ty}
 \int\limits_{\substack{u_2+2T_2>0\\u>0}}\!\!
\chi(u,y_1,T) e^{-\tr(uy')} \left|\kzxz{u_1/y_1}{u_{12}/y_1^{1/2}}{{}^tu_{12}/y_1^{1/2}}{u_2}+2T\right|^{\alpha-\rho_n}\cdot
|u|^{\beta-\rho_n}\, du,
\end{align}
where $\chi(u,y_1,T)$ denotes the characteristic function
\[
\chi(u,y_1,T) = \begin{cases}
1,&\text{if \,$\frac{u_1}{y_1}+2T_1 -(\frac{u_{12}}{\sqrt{y_1}}+2T_{12})(u_2+2T_2)^{-1}\,{}^t(\frac{u_{12}}{\sqrt{y_1}}+2T_{12})>0$,}\\[1ex]
0,&\text{if \,$\frac{u_1}{y_1}+2T_1 -(\frac{u_{12}}{\sqrt{y_1}}+2T_{12})(u_2+2T_2)^{-1}\,{}^t(\frac{u_{12}}{\sqrt{y_1}}+2T_{12})\leq 0$.}
\end{cases}
\]

We now compute the desired limit as $y_1\to \infty$ \emph{assuming} that
the integration can be interchanged with the limit.
After that we will come back to the justification of the interchange.
We have
\begin{align*}
&\lim_{y_1\to \infty} e^{T_1y_1} y_1^\beta \cdot \eta^{(n)}(y,T,\alpha,\beta)\\
&=e^{-2T_{12}{}^ty_{12}-\tr T_2 y_2} \\
&\phantom{=}{}\times
 \int\limits_{\substack{u_2+2T_2>0\\u>0}}\!\!
\lim_{y_1\to \infty}\chi(u,y_1,T) e^{-\tr(uy')} \left|\kzxz{u_1/y_1}{u_{12}/y_1^{1/2}}{{}^tu_{12}/y_1^{1/2}}{u_2}+2T\right|^{\alpha-\rho_n}\cdot
|u|^{\beta-\rho_n}\, du\\
&=e^{-2T_{12}{}^ty_{12}-\tr T_2 y_2}
 \int\limits_{\substack{\kzxz{0}{0}{0}{u_2}+2T>0\\u>0}}\!\!
 e^{-u_1-\tr u_2y_2} \big|\kzxz{0}{0}{0}{u_2}+2T\big|^{\alpha-\rho_n}\cdot
|u|^{\beta-\rho_n}\, du.
\end{align*}
If $T_1\leq 0$, then the domain of integration is empty and the integral vanishes as claimed.

If $T_1>0$, and $n=1$, then the remaining integral reduces to the Euler integral for the Gamma function, which implies the assertion in this case. If $T_1>1$ and $n>1$, then we use Lemma~\ref{lem:block} again (but now the other of the two formulas) to write
\begin{align*}
\left|\kzxz{0}{0}{0}{u_2}+2T\right|^{\alpha-\rho_n}
&= 2T_1 \cdot |u_2+2T_2-{}^t(2T_{12})(2T_1)^{-1}(2T_{12})|\\
&= 2\tilde{T}_1 \cdot |u_2+2\tilde{T}_2|.
\end{align*}
Inserting this in the integral, we obtain
\begin{align}
\label{eq:intref}
&\lim_{y_1\to \infty} e^{T_1y_1} y_1^\beta \cdot \eta^{(n)}(y,T,\alpha,\beta)\\
\nonumber
&=e^{-2T_{12}{}^ty_{12}-\tr T_2 y_2} (2T_1)^{\alpha-\rho_n}
 \int\limits_{\substack{u_2+2\tilde{T}_2>0\\u>0}}\!\!
 e^{-u_1-\tr u_2y_2} \big|u_2+2\tilde{T}_2\big|^{\alpha-\rho_n}\cdot
|u|^{\beta-\rho_n}\, du\\
\nonumber
&=e^{-2T_{12}{}^ty_{12}-\tr T_2 y_2} (2T_1)^{\alpha-\rho_n}
\\
\nonumber
&\phantom{=}{}\times
 \int\limits_{\substack{u_2+2\tilde{T}_2>0\\u_2>0\\u_1-u_{12}u_2^{-1}\, {}^tu_{12}>0}}\!\!
 e^{-u_1-\tr u_2y_2} \big|u_2+2\tilde{T}_2\big|^{\alpha-\rho_n}\cdot
|u_2|^{\beta-\rho_n}(u_1-u_{12}u_2^{-1} \,{}^tu_{12})^{\beta-\rho_n}\, du.
\end{align}
We carry out the $u_1$-integration and employ Lemma \ref{lem:gaussian}
to get
\begin{align*}
&\lim_{y_1\to \infty} e^{T_1y_1} y_1^\beta \cdot \eta^{(n)}(y,T,\alpha,\beta)\\
&=e^{-2T_{12}{}^ty_{12}-\tr T_2 y_2} (2T_1)^{\alpha-\rho_n}\Gamma(\beta-\rho_n+1)
\\&\phantom{=}{}\times
 \int\limits_{\substack{u_2+2\tilde{T}_2>0\\u_2>0}}\!\!
e^{-\tr u_2y_2}\cdot\big|u_2+2\tilde{T}_2\big|^{\alpha-\rho_n}\cdot
|u_2|^{\beta-\rho_n}
\int\limits_{u_{12}\in \R^{1\times(n-1)}}
 e^{-u_{12}u_2^{-1}\, {}^tu_
{12}}
 \, du_{12}\,du_2\\
&=e^{-2T_{12}{}^ty_{12}-\tr T_2 y_2} (2T_1)^{\alpha-\rho_n}\pi^{\frac{n-1}{2}}\Gamma(\beta-\rho_n+1)
\\&\phantom{=}{}\times
 \int\limits_{\substack{u_2+2\tilde{T}_2>0\\u_2>0}}\!\!
e^{-\tr u_2y_2}\cdot\big|u_2+2\tilde{T}_2\big|^{\alpha-\rho_n}\cdot
|u_2|^{\beta+\frac{1}{2}-\rho_n}
\,du_2.
\end{align*}
Shifting the variable of integration, we get
\begin{align*}
&\lim_{y_1\to \infty} e^{T_1y_1} y_1^\beta \cdot \eta^{(n)}(y,T,\alpha,\beta)\\
&=\exp\big(-2T_{12}{}^ty_{12}-\tr T_2 y_2+\tr\tilde{T}_2y_2\big) (2T_1)^{\alpha-\rho_n}\pi^{\frac{n-1}{2}}\Gamma(\beta-\rho_n+1)
\\&\phantom{=}{}\times
 \int\limits_{\substack{u_2+\tilde{T}_2>0\\u_2-\tilde{T}_2>0}}\!\!
e^{-\tr u_2y_2}\cdot\big|u_2+\tilde{T}_2\big|^{\alpha-\rho_{n-1}-\frac{1}{2}}\cdot
|u_2-\tilde{T}_2|^{\beta-\rho_{n-1}}
\,du_2.
\end{align*}
Since the latter integral is $\eta^{(n-1)}(y_2,\tilde{T}_2,\alpha-\frac{1}{2},\beta)$, we obtain the claimed formula.

\emph{Step 2.}
We now justify the interchange of the integral and the limit in \eqref{eq:lebes}.
To this end we bound the integrand in
\begin{align}
\label{eq:itb}
 \int\limits_{\substack{u_2+2T_2>0\\u>0}}\!\!
\chi(u,y_1,T) e^{-\tr(uy')} \left|\kzxz{u_1/y_1}{u_{12}/y_1^{1/2}}{{}^tu_{12}/
y_1^{1/2}}{u_2}+2T\right|^{\alpha-\rho_n}\cdot
|u|^{\beta-\rho_n}\, du
\end{align}
by an integrable function which is independent of $y_1$.
On the domain of integration and where $\chi(u,y,T)$ is non-zero,
the quantity
\begin{align*}
&\left|\kzxz{u_1/y_1}{u_{12}/y_1^{1/2}}{{}^tu_{12}/
y_1^{1/2}}{u_2}+2T\right|\\
&=|u_2+2T_2| \cdot \left(\frac{u_1}{y_1}+2T_1 -(\frac{u_{12}}{\sqrt{y_1}}+2T_{12})(u_2+2T_2)^{-1}\,{}^t(\frac{u_{12}}{\sqrt{y_1}}+2T_{12})\right)
\end{align*}
is bounded by
\[
|u_2+2T_2| \cdot (u_1+2|T_1|)
\]
from above when $y_1>1$. It is bounded by $0$ from below.
Moreover, for $y_1>4y_{12}y_2^{-1}\,{}^ty_{12}$ we have
\[
y'>\frac{1}{2}\zxz{1}{0}{0}{y_2}
\]
and therefore
\[
e^{-\tr(uy')} \leq e^{-u_1/2-\tr(u_2y_2)}.
\]
Hence, for such $y_1$ the integrand in \eqref{eq:itb} is bounded by
\[
e^{-u_1/2-\tr(u_2y_2)}(u_1+2|T_1|)^{\Re(\alpha)-\rho_n}|u_2+2T_2|^{\Re(\alpha)-\rho_n} \cdot |u|^{\Re(\beta)-\rho_n}
\]
on the domain of integration.
Here we have also used that $\Re(\alpha)>\rho_n$.
There exists a constant $C=C(\alpha)>0$ such that this is bounded by
\[
C\cdot e^{-u_1/4-\tr(u_2y_2)}|u_2+2T_2|^{\Re(\alpha)-\rho_n} \cdot |u|^{\Re(\beta)-\rho_n}
\]
locally uniformly in $\alpha$. Note that this function is independent of $y_1$. Consequently, by the dominated convergence theorem, the interchange of the integral and the limit in \eqref{eq:lebes} follows if
\begin{align}
I(y_2,T_2,\alpha,\beta)=\int\limits_{\substack{u_2+2T_2>0\\u>0}}
e^{-u_1/4-\tr(u_2y_2)}|u_2+2T_2|^{\Re(\alpha)-\rho_n} \cdot |u|^{\Re(\beta)-\rho_n}\,du
\end{align}
converges. But this integral is of the same form as the one on the right hand side of \eqref{eq:intref}.  The same computation shows that the integral is given by  a constant times
\[
e^{\frac{1}{2}\tr(y_2 T_2)}\cdot\eta^{(n-1)}(\frac{1}{2}y_2,T_2, \Re(\alpha)-\frac{1}{2},\Re(\beta)).
\]
This shows the convergence for $\Re(\alpha)>\rho_n$ and $\Re(\beta)>\rho_n-1/2$.

\emph{Step 3.}
We now show that the formulas of the theorem actually hold for all $\alpha,\beta\in \C$. If $T$ is positive definite, we use the functional equation of $\eta$ under $(\alpha,\beta)\mapsto(\rho_n-\beta,\rho_n-\alpha)$ and argue as in \cite[pp.~281]{Sh}.  For general $T$, we then apply the integral representation (4.24) in \cite[pp.~289]{Sh} to deduce the assertion.
\end{proof}

%


If $a\in \R^\times$, we let
\begin{align}
\label{eq:da}
d(a)=\begin{pmatrix}
a&&&\\
&1&&\\
&&\ddots&\\
&&&1
\end{pmatrix}\in \Gl_n(\R).
\end{align}
Theorem \ref{thm:etalimnew} implies the following asymptotic behavior of the Whittaker function.

\begin{corollary}
\label{thm:wasy}
Let $\kappa=\rho_n$.
For general invertible $T$ we have
\begin{align*}
&\lim_{a\to \infty} e^{2\pi (T_1 a^2+T_1^{-1}T_{12}{}^tT_{12})} a^{-\rho_n} \cdot W_T(m(d(a)),s,\Phi_\kappa)\\
&=\begin{cases}
\displaystyle \frac{i^{-n}(2\pi)^{\rho_n} }{\Gamma(\frac{s}{2}+\rho_n)}
(\pi T_1)^{s/2}\cdot
W_{\tilde{T}_2}(1,s,\Phi_{\kappa-1/2}) ,&\text{if $T_1>0$,}\\[2ex]
0,&\text{if $T_1\leq 0$.}
\end{cases}
\end{align*}
Here the Whittaker function on the left hand side is in genus $n$ and the one on the right hand side in genus $n-1$ (which is clear from the size of the matrices in the subscripts).
\end{corollary}

\begin{proof}
Let $a>0$. According to \eqref{eq:weta}, we have
\begin{align*}
a^{-\rho_n} \cdot W_T(m(d(a)),s,\Phi_\kappa)&=c_n(\alpha,\beta) \cdot
a^{s}\cdot  \eta^{(n)}(2d(a)^2,\pi T,\alpha,\beta),\\
\nonumber
c_n(\alpha,\beta)&= \frac{i^{n(\beta-\alpha)}2^{-n(\rho_n-1)/2}(2\pi)^{n\rho_n
}}{\Gamma_n(\alpha)\Gamma_n(\beta)}.
\end{align*}
If $T_1\leq 0$, we obtain by means of Theorem \ref{thm:etalimnew} that
\begin{align*}
\lim_{a\to \infty} e^{2\pi (T_1 a^2+T_1^{-1}T_{12}{}^tT_{12})} a^{-\rho_n} \cdot W_T(m(d(a)),s,\Phi_\kappa)=0.
\end{align*}
If $T_1>0$, we find
\begin{align*}
&\lim_{a\to \infty} e^{2\pi (T_1 a^2+T_1^{-1}T_{12}{}^tT_{12})} a^{-\rho_n} \cdot W_T(m(d(a)),s,\Phi_\kappa)\\
&=c_n(\alpha,\beta)\Gamma(\beta+1-\rho_n) \pi^{\frac{n-1}{2}}
(2\pi T_1)^{\alpha-\rho_n}2^{-\frac{s}{2}}\eta^{(n-1)}(2,\pi\tilde{T}_2,\alpha-\frac{1}{2},\beta).
\end{align*}
Using
\eqref{eq:weta} in genus $n-1$, we see that
\begin{align*}
W_{\tilde{T}_2}(1,s,\Phi_{\kappa-1/2})&=c_{n-1}(\frac{s}{2}+\rho_{n-1},\frac{s}{2}) \cdot
\eta^{(n-1)}(2,\pi \tilde{T}_2,\alpha-\frac{1}{2},\beta).
\end{align*}
Inserting this we get
\begin{align*}
&\lim_{a\to \infty}e^{2\pi (T_1 a^2+T_1^{-1}T_{12}{}^tT_{12})}  a^{-\rho_n} \cdot W_T(m(d(a)),s,\Phi_\kappa)\\
&=\frac{c_n(\alpha,\beta)\Gamma(\beta+1-\rho_n)}{ c_{n-1}(\frac{s}{2}+\rho_{n-1},\frac{s}{2}) }\pi^{\frac{n-1}{2}}
(\pi T_1)^{s/2}\cdot W_{\tilde T_2}(1,s,\Phi_{\kappa-1/2}).
\end{align*}
Employing the relations
\begin{align}
\label{eq:mg2}
\Gamma_{n-1}(\beta)\Gamma_1(\beta+1-\rho_n)&=\pi^{-\frac{n-1}{2}}\Gamma_n(\beta),\\
\label{eq:mg3}
\Gamma_{n-1}(\beta+\rho_{n-1})\Gamma_1(\beta+\rho_n)&=\pi^{-\frac{n-1}{2}}\Gamma_n(\beta+\rho_n),
\end{align}
we find
\begin{align*}
\frac{c_n(\alpha,\beta)\Gamma(\beta+1-\rho_n)}{ c_{n-1}(\frac{s}{2}+\rho_{n-1},\frac{s}{2}) } = \frac{i^{-n}2^{\rho_n} \pi}{\Gamma(\frac{s}{2}+\rho_n)}.
\end{align*}
Consequently, we obtain
\begin{align*}
&\lim_{a\to \infty} e^{2\pi (T_1 a^2+T_1^{-1}T_{12}{}^tT_{12})}
a^{-\rho_n} \cdot W_T(m(d(a)),s,\Phi_\kappa)\\
&=\frac{i^{-n}(2\pi)^{\rho_n} }{\Gamma(\frac{s}{2}+\rho_n)}
(\pi T_1)^{s/2}\cdot W_{\tilde T_2}(1,s,\Phi_{\kappa-1/2}).
\end{align*}
This concludes the proof of the corollary.
\end{proof}

\begin{corollary}
\label{cor:wderivasy}
Let $\kappa=\rho_n$.
If $T$ has signature $(n-j,j)$ with $j\geq 1$, then
\begin{align*}
&\lim_{a\to \infty} e^{2\pi (T_1 a^2+T_1^{-1}T_{12}{}^tT_{12})} a^{-\rho_n} \cdot W_T'(m(d(a)),0,\Phi_\kappa)\\
&=\begin{cases} \displaystyle\frac{i^{-n}(2\pi)^{\rho_n}}{\Gamma(\rho_{n})}
\cdot
W'_{\tilde T_2}(1,0,\Phi_{\kappa-1/2}) ,&\text{if $T_1>0$,}\\[2ex]
0,&\text{if $T_1\leq 0$.}
\end{cases}
\end{align*}
\end{corollary}


\section{The archimedian arithmetic Siegel-Weil formula}
\label{sect:LocalASW}


Here we use the archimedidean local Siegel-Weil formula (see Proposition \ref{prop:LocalSiegel-WeilR}), the asymptotic behavior of Theorem \ref{thm:etalimnew}, and some computations in the induced representation to prove Theorem \ref{thm:intro-alsw}.

We use the same setup and notation as in Section \ref{sect:3}.
In particular, $V$ is a quadratic space over $\R$ of signature $(m, 2)$, and $\kappa=\frac{m+2}{2}$. Moreover,   $\calD$ is the hermitian domain associated with $H=\SO(V)$, realized as the Grassmannian of oriented negative $2$-planes in $V$.

\subsection{Green currents and local heights}
For $z\in \calD$ the orthogonal complement $z^\perp$ is positive definte of dimension $m$. If $x\in V$, we denote the orthogonal projetion of $x$ to $z$ and $z^\perp$ by $x_z$ and $x_{z^\perp}$, respectively. The quadratic form
\[
(x,x)_z= (x_{z^\perp},x_{z^\perp}) - (x_z,x_z)
\]
is positive definite on $V$, the so called majorant associated with $z$.
We also put
\[
R(x,z)= -(x_z,x_z)
\]
so that
$(x,x)_z= (x,x)+2R(x,z)$.
For $0\ne x \in V$, we define
$$
\mathcal D_x=\{z\in \mathcal D\mid\; z\perp x\}=\{z\in \mathcal D\mid\; R(x,z)=0\}.
$$
Then $\calD_x$ is a non-trivial analytic  divisor of $\mathcal D$ if $Q(x) >0$, and it is empty  if $Q(x) \le 0$ (which we will view as the zero divisor).
Following \cite{Ku1}
we define the Kudla Green function
\begin{align}
\label{eq:defxi}
\xi(x,z) = -\Ei(-2\pi R(x,z)) \cdot e^{-\pi (x,x)},
\end{align}
where
$\Ei(u)= \int_{-\infty}^u e^t\,\frac{dt}{t}$ is the exponential integral, see
\cite[Chapter 5]{AS}.
If $x\in V$ is fixed, then $\xi(x,z)$ is a smooth function on $\calD\setminus \calD_x$ with a logarithmic singularity along $\calD_x$. It has the equivariance property
$\xi(gx,gz)=\xi(x,z)$ for $g\in H(\R)$.
The differential form
\begin{align}
\label{eq:defphikm}
\varphi_{KM}(x,z)=dd^c\xi(x,z)
\end{align}
extends to a smooth $(1,1)$-form on all of $\calD$, where $d^c=\frac{1}{4\pi i}(\partial-\bar\partial)$.
It is the Kudla-Millson Schwartz form which is Poincar\'e dual for the cycle $e^{-\pi (x,x)}\calD_x$, see \cite{KM2}, \cite{Ku1}. More precisely, as currents on $\calD$, we have the identity
\[
dd^c[\xi(x)]+e^{-\pi (x,x)}\delta_{\calD_x} = [\varphi_{KM}(x)].
\]
Because of the $H(\R)$-equivariance of $\xi(x,z)$, the $(1,1)$ form
\begin{align*}
\Omega=\varphi_{KM}(0,z)
\end{align*}
is $H(\R)$-invariant on $\calD$. In fact, it is equal to the invariant differential form defined earlier in \eqref{eq:defom}.
%

For $x=(x_1,\dots,x_n)\in V^n$ and $z\in \calD$ we also define the  Kudla-Millson Schwartz form in genus $n$ as
\begin{align*}
\varphi_{KM}^n(x,z)&=\varphi_{KM}(x_1,z)\wedge \dots\wedge \varphi_{KM}(x_n,z).
\end{align*}
With respect to the action of $G$ through the Weil representation it transforms under the maximal compact subgroup $K_G$ with the character
$\det(\uk)^{\kappa}$.
If $Q(x)\in \Sym_n(\R)$ is invertible,
the form $\varphi_{KM}^n(x,z)$ is Poincar\'e dual for the special cycle $e^{-\pi \tr (x,x)}\calD_x$, where
\[
\calD_x= \{ z\in \calD\mid \; \text{$z\perp x_i$ for $i=1,\dots,n$}\},
\]
see \cite{KM2}.
We define a Green current for the cycle $e^{-\pi \tr (x,x)}\calD_x$ by taking the star product
\begin{align}
\xi^n(x,z)=\xi(x_1,z)*\dots *\xi(x_n,z)
\end{align}
in the sense of \cite{GS}.
As a current on compactly supported smooth differential forms it satisfies the equation
\begin{align}
\label{eq:currentn}
dd^c[\xi^n(x)]+e^{-\pi \tr (x,x)}\cdot \delta_{\calD_x} = [\varphi_{KM}^n(x)].
\end{align}
When $\calD_x$ is compact, it follows from the growth estimates in \cite[Section 6]{KM2} that $\xi^n(x,z)$ is rapidly decaying and extends to a current on forms of moderate growth with \eqref{eq:currentn} still holding.
A recursive formula for the star product is given by
\begin{align}
\label{eq:wr}
\xi^n(x,z)=\xi(x_1,z)\wedge \varphi_{KM}^{n-1}((x_2,\dots,x_n),z)+e^{-\pi (x_1,x_1)}\delta_{\calD_{x_1}}\wedge \xi^{n-1}((x_2,\dots,x_n),z),
\end{align}
where $\varphi_{KM}^{0}$ has to be interpreted as $1$ and $\xi^{0}$ as $0$. The current $\xi^n(x,z)$ is invariant under permutations of the components of $x$.

Sometimes it is convenient to put
\begin{align*}
\xi_0^n(x,z) &= \xi^n(x,z)\cdot e^{\pi \tr (x,x)},\\
\varphi^n_{KM,0}(x,z)&=\varphi^n_{KM}(x,z)\cdot e^{\pi \tr(x,x)}.
\end{align*}
Then the current equation becomes
\[
dd^c[\xi^n_0(x)]+\delta_{\calD_x} = [\varphi^n_{KM,0}(x)].
\]

Note that the current equation \eqref{eq:currentn} together with
Proposition \ref{prop:whitt0} implies the following geometric local Siegel-Weil formula, which is the local archimedian version of
\eqref{eq:degid}.

\begin{proposition}
\label{prop:geolsw}
Assume that $n=m$ and that $T=Q(x)$ is invertible. Then
\[
\int_{\calD} \varphi_{KM}^n(x,z) =\begin{cases}2e^{-2\pi\tr T},& \text{if $T$ is positive definite,}\\
0,&\text{if $T$ is not positive definite.}
\end{cases}
\]
Moreover, in both cases this is equal to
\[
{2}^{1-\frac{n}{2}(\kappa+\frac{3}{2})}{\pi }^{-n\kappa}
 i^{n\kappa}
\Gamma_n(\kappa) |T|^{-\frac{1}{2}} W_T(1,\frac{1}{2},\Phi_\kappa),
%
\]
where $\Phi_\kappa\in I_n(s,\chi_V)$ is the weight $\kappa$ standard section, that is, the unique standard section whose restriction to $K_G$ is the character $\det(\uk)^{\kappa}$.
\end{proposition}


\begin{proof}
The first statement is a direct consequence of the current equation \eqref{eq:currentn} applied to the constant function with value $1$.
The second statement can be deduced from the first one by means of the formulas of \cite{Sh}. Since we do not need it here, we omit the proof.
\end{proof}

Throughout the rest of this subsection we assume that $n=m+1$.
Then $\xi^{n}_0(x,z)$ is a top degree current, which can be evaluated at the constant function $1$.
For $x\in V^n$ we define the archimedian local height function by
\begin{align}
\label{eq:deflh}
\Ht_\infty(x)=\frac{1}2 \int_{\calD} \xi_0^n(x,z).
\end{align}
In this section we prove the archimedian arithmetic local Siegel Weil formula Theorem~\ref{thm:intro-alsw}, relating $\Ht_\infty(x)$
to the {\em derivative} of a Whittaker function in genus $n$.
It can be viewed as an arithmetic analogue of Proposition \ref{prop:geolsw}.
We restate the Theorem for convenience.

\begin{theorem}
\label{thm:alsw}
Let $x\in V^n(\R)$ such that the moment matrix $T=Q(x)$
is invertible.
Then we have
\begin{align}
\label{eq:i-lsw}
\Ht_\infty(xv^{1/2}) \cdot q^T
=-B_{n,\infty}\det(v)^{-\kappa/2}\cdot
W_{T}'(g_\tau,0,\Phi_{\kappa}),
\end{align}
where $B_{n,\infty}$ is the constant in Proposition \ref{prop:LocalSiegel-WeilR},
and $\Phi_\kappa\in I_n(s,\chi_V)$ is the weight $\kappa$ standard section, that is, the unique standard section whose restriction to $K_G$ is the character $\det(\uk)^{\kappa}$. The derivative of the Whittaker function is taken with respect to $s$.
\end{theorem}

Let $x\in V^n$ and assume that $T=Q(x)$ is invertible.
To prove Theorem \ref{thm:alsw} we employ the recursive formula \eqref{eq:wr} for the star product. It implies that
\begin{align}
\label{eq:lh2}
\Ht_\infty(x)=\Ht_\infty^\main(x)
+\frac{1}{2}\int_{\calD_{x_1}} \xi^{n-1}_0((x_2,\dots,x_n),z),
\end{align}
where we write
\begin{align}
\label{eq:lh1}
\Ht_\infty^\main(x)= \frac{1}{2}\int_{\calD} \xi_0(x_1,z)\wedge \varphi_{KM,0}^{n-1}((x_2,\dots,x_n),z)
\end{align}
for the main term
of the local height function.
The second summand on the right hand side of \eqref{eq:lh2}
vanishes when $Q(x_1)\leq 0$, in which case $\calD_{x_1
}$ is empty.  When $Q(x_1)>0$, this quantity
is a local height function in genus $n-1$ for the quadratic space $V_1=x_1^
\perp\subset V$ of signature $(m-1,2)$.
The divisor
$\calD_{x_1}$ is naturally isomorphic to the Grassmannian of $V_1$.
%
Let
\[
\pr:V\to V_1, \quad \pr(y)=y-\frac{(y,x_1)}{(x_1,x_1)}x_1
\]
be the orthogonal projection
and put
$\tilde x = (\pr(x_2),\dots,\pr(x_n))\in V_1^{n-1}$.
If we write $T$ in block form as in \eqref{eq:varblock}
then the
moment matrix of $\tilde x$ is
\begin{align}
\label{eq:tt}
\tilde T_2=
Q(\tilde x) = T_2-{}^tT_{12} T_1^{-1} T_{12}.
\end{align}

\begin{lemma}
\label{lem:term2}
Assume the above notation.
If $z\in \calD_{x_1}$, we have
\[
\xi_0^{n-1}((x_2,\dots,x_n),z)= \xi^{n-1}_0(
\tilde x,z)
\]
and
\begin{align*}
\frac{1}{2}\int_{\calD_{x_1}} \xi^{n-1}_0((x_2,\dots,x_n),z)= \Ht_\infty(\tilde x).
\end{align*}
Here the height function on the right hand side is for the tuple
$\tilde x\in V_1^{n-1}$.
\end{lemma}

\begin{proof}
If $z\in \calD_{x_1}$ and $y\in V$, then $R(y,z)=R(\pr(y),z)$. Hence the assertion is a direct consequence of the definition of $\xi_0(y,z)$.
\end{proof}

The following result gives a formula for the main term of the local height function.

\begin{theorem}
\label{thm:pre-alsw}
Assume that $n=m+1$.
Let $x\in V^n$ such that $T=Q(x)$
is invertible, and put $\tilde T_2=Q(\tilde x)$ as in \eqref{eq:tt}.
Then we have
\[
e^{-2\pi \tr T} \cdot \Ht_\infty^\main(x)=-B_{n,\infty} \cdot W_T'(1,0,\Phi_\kappa)
+B_{n-1,\infty} e^{-2\pi (\tr T-\tr\tilde T_2)} \cdot W'_{\tilde T_2}(1,0,\Phi_{\kappa-1/2}).
\]
If $n=1$, the second summand on the right hand side is interpreted as $0$.
\end{theorem}

The proof of Theorem~\ref{thm:pre-alsw} will be given in the next three subsections.

\begin{proof}[Proof that Theorem~\ref{thm:pre-alsw} implies
Theorem \ref{thm:alsw}]
Recall that for $\tau\in \H_n$ we have put $g_\tau=\kzxz{1}{u}{}{1}
\kzxz{a}{}{}{{}^ta^{-1}}\in \Sp_n(\R)$ with $a\in \GL_n^+(\R)$ and $a\,{}^ta=v$.
Using the transformation behavior \eqref{eq:i1} of the Whittaker function, we find that
\begin{align}
\label{eq:i11}
\det(v)^{-\kappa/2}W_{T}(g_\tau,s,\Phi_\kappa)&=e^{2\pi i\tr(T u)}
|a|^{-s} \cdot W_{{}^{t}aTa}(1,s,\Phi_\kappa).
\end{align}
Since the signature of $V$ is $(m,2)$ and $n=m+1$, the matrix $T$ must have negative eigenvalues. By Proposition \ref{prop:whitt0}, the Whittaker function $W_{T}(g,s,\Phi_\kappa)$ vanishes at $s=0$. Employing \eqref{eq:i11}, we see that
\eqref{eq:i-lsw} is equivalent to
\begin{align*}
\Ht_\infty(xa) \cdot e^{-2\pi \tr {}^{t}aTa}
=-B_{n,\infty}
\cdot W'_{{}^{t}aTa}(1,0,\Phi_\kappa).
\end{align*}
Consequently, it suffices to prove \eqref{eq:i-lsw} for $\tau=i 1_n$, that is
\begin{align}
\label{eq:ii-lsw}
\Ht_\infty(x) \cdot e^{-2\pi \tr T}
=-B_{n,\infty}
\cdot W'_{T}(1,0,\Phi_\kappa).
\end{align}

We show \eqref{eq:ii-lsw} by induction on $n$.
If $n=1$, then $\Ht_\infty^\main(x)=\Ht_\infty(x)$, and we have nothing to show. Assume now that $n>1$.
According to \eqref{eq:lh2} and Lemma \ref{lem:term2} we have
\begin{align*}
\Ht_\infty(x)=\Ht_\infty^\main(x) + \Ht_\infty(\tilde x).
\end{align*}
By Theorem \ref{thm:pre-alsw}, we obtain
\begin{align*}
e^{-2\pi \tr T}\cdot \Ht_\infty(x)&=-  B_{n,\infty}\cdot W_T'(1,0,\Phi_\kappa)+ B_{n-1,\infty} e^{-2\pi (\tr T-\tr\tilde T_2)} \cdot W'_{\tilde T_2}(1,0,\Phi_{\kappa-1/2})\\
&\phantom{=}{} + e^{-2\pi \tr T}\cdot \Ht_\infty(\tilde x).
\end{align*}
If we use Theorem \ref{thm:alsw} in genus $n-1$ to compute the last term on the right hand side, we get the assertion.
\end{proof}

\subsection{The main term of the local height}
In this subsection we assume again that $n=m+1$.
We give a first formula for the main term of the local height in terms of a certain Whittaker function.
We begin by rewriting the Green function
$\xi(x,z)$ defined in \eqref{eq:defxi} in terms of the Gaussian
\begin{align}
\varphi_G(x,z)&=e^{-\pi(x,x)_z}\in S(V),\\
\varphi_{G,0}(x,z)&= \varphi_G(x,z) \cdot e^{\pi(x,x)} = e^{-2\pi R(x,z)}.
\end{align}

\begin{lemma}
\label{lem:xigauss}
If $x\in V$ and $z\in \calD$, we have
\[
\xi_0(x,z) =
\int_{t=1}^\infty \varphi_{G,0}(\sqrt{t}x,z) \,\frac{dt}{t}.
\]
\end{lemma}

\begin{proof}
The statement follows from the  integral representation
\[
-\Ei(-z) = \int_1^\infty \frac{e^{-zt}}{t}\,dt
\]
by inserting the definitions of $\xi_0(x,z)$ and $R(x,z)$.
\end{proof}

By our assumption on $m$, the Schwartz form  $\varphi_{KM}^{n-1}$ is a top degree differential form on $\calD$. We write it as
\[
\varphi_{KM}^{n-1}(y,z)=\varphi_{KM}^{n-1,*}(y,z)\cdot \Omega^{n-1}.
\]
For $x\in V^n$, we define a Schwartz function\footnote{Later we will also define a $\Sym_n(\R)$-valued Schwartz function $\psi$ whose $(1,1)$-component will by $\psi_{11}$.}
by
\begin{align}
\label{eq:psi11}
\psi_{11}^*(x,z)=\varphi_G(x_1,z)\cdot \varphi_{KM}^{n-1,*}((x_2,\dots,x_n),z).
\end{align}

\begin{proposition}
\label{prop:3.8}
Let $x\in V^n$, put $T=Q(x)$, and write $T=\kzxz{T_1}{T_{12}}{{}^tT_{12}}{T_2}$ as in \eqref{eq:varblock}. The main term
of the local
height function
is given by
\[
\Ht_\infty^\main(x) = B_{n,\infty} e^{2\pi \tr T_2}
\int_1^\infty W_{d(a)T d(a)}(1,0,\lambda(\psi_{11}^*))\cdot e^{2\pi Q(x_1 a)
}\,\frac{da}{a},
\]
where $B_{n,\infty}$ denotes the constant in Proposition \ref{prop:LocalSiegel-WeilR}
and $d(a)$ is given by \eqref{eq:da}.
\end{proposition}

\begin{proof}
Using Lemma \ref{lem:xigauss} and \eqref{eq:psi11}, we see
\begin{align*}
\Ht_\infty^\main(x)&= \frac{1}{2}\int_{\calD} \xi_0(x_1,z)\wedge \varphi_{KM,0}^{n-1}((x_2,\dots,x_n),z)\\
&= \frac{1}{2}\int_{t=1}^\infty
\int_{\calD} \varphi_{G,0}(\sqrt{t}x_1,z) \varphi_{KM,0}^{n-1}((x_2,\dots,x_n),z)\,\frac{dt}{t}\\
&=e^{2\pi \tr T_2}\cdot \int_{a=1}^\infty
\int_{\calD} \psi_{11}^*(x d(a),z)\, \Omega^{n-1}\cdot
e^{2\pi Q(x_1 a)}\frac{da}{a}.
\end{align*}
By the local Siegel Weil formula, Proposition \ref{prop:LocalSiegel-WeilR},
we have
\[
\int_{\calD} \psi_{11}^*(x ,z) \,\Omega^{n-1} = B_{n,\infty}\cdot
W_T(1,0,\lambda(\psi_{11}^*)).
\]
Inserting this, we obtain the assertion.
\end{proof}

\subsection{Some Lie algebra computations}

In this subsection, we temporarily drop the assumption that $n=m+1$. We compute the Whittaker function $W_T(1,s,\lambda(\psi_{11}^*))$
more explicitly. We begin by recalling from \cite[Section 5]{BFK} some facts about the Lie algebra of $G$.
Let
\[
\frakg = \frakk +\frakp_++\frakp_-
\]
be the Harish-Chandra decomposition of $\frakg=\Lie(G)\otimes_\R\C$.
Let $S=\Sym_n(\R)$. Then there are isomorphisms
\begin{align}
\label{eq:piso}
p_\pm: S_\C\longrightarrow \frakp_\pm, \quad X\mapsto  p_\pm(X) = \frac{1}{2}\zxz{X}{\pm iX}{\pm i X}{-X}.
\end{align}
The group $K_G$ acts on $\frakg$
by the adjoint representation, $\operatorname{Ad}(k)g= kgk^{-1}$,  and on
$S_\C$ by $\uk . X = \uk X {}^t\uk$ for $k\in K_G$. For the isomorphism
\eqref{eq:piso} we have
\begin{align}
\operatorname{Ad}(k)p_+(X)&=p_+(\uk . X),\\
\operatorname{Ad}(k)p_-(X)&=p_-(\bar \uk . X).
\end{align}
The trace pairing
\[
\langle p_+(X),p_-(Y)\rangle = \tr(XY)
\]
is invariant under the action of $K_G$, and therefore $\frakp_\pm^*\cong \frakp_\mp$ as $K_G$-modules.
Let $(e_\alpha)$ be a basis of $S$, and write $(e_\alpha^\vee)$ for the dual basis with respect to the trace form. Then $(p_-(e_\alpha^\vee))$ is a basis of $\frakp_-$, and we write $(\eta'_\alpha)$ for the dual basis of $\frakp_-^*$. We identify $\frakp_-^*$ with $S_\C$ by the map
\begin{align}
\label{eq:dualid}
\psi=\sum_\alpha \psi_\alpha \eta_\alpha'\mapsto \sum_\alpha \psi_\alpha e_\alpha.
\end{align}

Recall that the Lie algebra $\mathfrak{gl}_n(\C)\cong \Mat_n(\C)$ is isomorphic to $\frakk$ via the map
\begin{align}
\label{eq:kiso}
k: \Mat_n(\C)\longrightarrow \frakk, \quad Y\mapsto  k(Y) = 
\frac{1}{2}\zxz{Y-{}^t Y}{-i(Y+{}^tY)}{i(Y+{}^tY)}{Y-{}^t Y}.
\end{align}
Let  $E_{jk}\in \Mat_n(\C)$ be the elementary matrix having the entry $1$ at the position $(j,k)$ and all other entries $0$.
Then the matrices 
\begin{align}
\label{eq:yjk}
Y_{jk}= k(E_{jk})
=\frac{1}{2}\zxz{E_{jk}-E_{kj}}{-i(E_{jk}+E_{kj}}{i(E_{jk}+E_{kj})}{E_{jk}-E_{kj}},
\end{align}
for $1\leq j,k\leq n$, form a basis of $\frakk$.

We denote by  $\C(\ell)$ the $K_G$-module given by
the action of $K_G$ on $\C$ by multiplication with $\det(\uk)^\ell$.
Recall that the space of differential forms $A^{p,q}(\H_n)$ on $\H_n$ can be described by the isomorphism
\[
A^{p,q}(\H_n)\longrightarrow \left[C^{\infty}(G)\otimes \wedge^p (\frakp_+^*)\otimes \wedge^q (\frakp_-^*)\right]^{K_G}.
\]
Here, the operator corresponding to $\bar\partial$ on $A^{p,q}(\H_n)$ is given by
\begin{align}
\label{eq:D}
D=\sum_\alpha p_-(e_\alpha^\vee)\otimes \eta'_\alpha,
\end{align}
where $\eta'_\alpha$ acts on $\wedge^\cdot (\frakp^*)$ by exterior multiplication.

The following result, which describes the action of $K_G$ on $\psi_{11}$, is taken from the unpublished manuscript \cite{BFK2}.
We thank Jens Funke and Steve Kudla for allowing to include it here.

\begin{proposition}
Identify $\frakp_-^*$ with $S_\C$ as in \eqref{eq:dualid}.
There is a Schwartz form
\[
\psi\in [S(V^n)\otimes A^{n-1,n-1}(\calD)\otimes \frakp_-^*\otimes \C(-\kappa)]^{K_G}
\]
with diagonal components
\[
\psi_{rr}(x)= \varphi_G(x_r)\cdot\varphi_{KM}(x_1)\wedge\dots\wedge \widehat{\varphi_{KM}(x_r)}\wedge\dots\wedge
\varphi_{KM}(x_n) ,
\]
which satisfies $\psi(0)= \Omega^{n-1} \cdot 1_n$ and
\begin{align*}
\omega(k)\psi(x) &= \det(\uk)^\kappa\cdot \uk^{-1} \psi(x)\,
 {}^t\uk^{-1} \qquad\text{for $k\in K_G$}
.
\end{align*}
\end{proposition}

\begin{proof}
To prove this result we use the Fock model realization of the Weil representation as described in the appendix of \cite{FM-AJM}. Let $\calF=\calF(\C^{(m+2)\times n})$ be the space of polynomial functions on $V_\C^n\cong  \C^{(m+2)\times n}$. As in \cite{FM-AJM} we denote the variables by $z_{\alpha j}$, $z_{\mu j}$, where $\alpha=1,\dots, m$, $\mu= m+1,m+2$, and $j=1,\dots,n$. The Lie algebra $\frakg\times \mathfrak{so}(V)_\C$ acts on $\calF$ via the Weil representation.

Let $\mathfrak{so}(V)_\C= \frakk_H\oplus \frakp_H$ be the Cartan decomposition as in Section \ref{sect:3.1}. Let $X_{\alpha\mu}$ be the standard basis of $\frakp_H$ and denote  by $\omega_{\alpha\mu}$ the corresponding dual basis of $\frakp_H^*$.
The Kudla-Millson Schwartz forms can be viewed as elements of
\[
\left[ \calF\otimes \wedge^\cdot \frakp_H^*\right]^{K_H}.
\] 
We have 
\[
\varphi_{KM}^n =\varphi_{KM,(1)}\wedge \dots \wedge \varphi_{KM,(n)},
\]
where
\[
\varphi_{KM,(j)} = -\frac{1}{8\pi^2}\cdot\sum_{\alpha,\beta=1}^m z_{\alpha j} z_{\beta j} \otimes \omega_{\alpha m+1}\wedge \omega_{\beta m+2}.
\]
The Gaussian $\varphi_{G,(j)}$ corresponds to the constant polynomial $1$ for every $j$.

We define the Schwartz form $\psi=(\psi_{jk})$ in the Fock model by putting
\begin{align*}
\psi_{kk} &= \varphi_{G,(k)}\wedge \prod_{l\neq k} \varphi_{KM,(l)}\quad\text{and}\\
\psi_{jk} &= -\frac{1}{8\pi^2} \left(-\frac{1}{2}\sum_{\alpha,\beta=1}^m (z_{\alpha j} z_{\beta k}+z_{\alpha k} z_{\beta j})
 \otimes \omega_{\alpha m+1}\wedge \omega_{\beta m+2}\right)\wedge \prod_{l\neq j,k} \;\varphi_{KM,(l)}\quad \text{for $j\neq k$.}
\end{align*}
This has the desired diagonal components. Using the intertwining operator between the Schr\"odinger and the Fock model of the Weil representation, it is easily checked that $\psi_{jk}(0)=0$ for $j\neq k$. On the other hand, by \eqref{eq:omega1}, we have $\psi_{jj}(0)=\Omega^{n-1}$, and therefore $\psi(0)=   \Omega^{n-1}\cdot 1_n$.

To verify the transformation law under 
$K_G$, we compute the action of the Lie algebra $\frakk$ under the Weil representation. Recall that the basis element $Y_{jk}$ defined in \eqref{eq:yjk} acts by 
\[
\omega(Y_{jk}) = \frac{1}{2}(m-2)\delta_{jk}+\sum_{\alpha=1}^m z_{\alpha j} \frac{\partial}{\partial z_{\alpha k}} - \sum_{\mu=m+1}^{m+2} z_{\mu k} \frac{\partial}{\partial z_{\mu j}}. 
\]
In fact, since  the element $Y_{jk}$ corresponds to  $\frac{1}{2i}(w_k'\circ w_j'')$ in the notation of \cite{FM-AJM}, this claim follows from \cite[Lemma A.1]{FM-AJM}.

Now a direct computation shows
\begin{align*}
\omega(Y_{jj}) \psi_{ll} &= \begin{cases} (\kappa-2)\cdot \psi_{ll}&\text{if $j=l$,}\\
\kappa\cdot \psi_{ll}&\text{if $j\neq l$,}
\end{cases}\\
\omega(Y_{jk}) \psi_{ll} &= \begin{cases} -2\psi_{kl}&\text{if $j=l$ and $j\neq k$,}\\
0 &\text{if $j\neq l$ and $j\neq k$,}
\end{cases}\\
\omega(Y_{jk}) \psi_{jk} &= \;\;\;-\psi_{kk} \quad\text{if $j\neq k$.}
\end{align*}
%
This implies that the $\psi_{jk}$ generate an irreducible representation of $K_G$, which has $\psi_{nn}$ as a highest weight vector, and which is isomorphic to $\det^\kappa\otimes \Sym^2(\C^n)^\vee$. Hence, we obtain the claimed transformation law. 
\end{proof}


%

The intertwining operator $\lambda: S(V^n)\to I(s_0,\chi_V)$ (see \eqref{eq:deflambda}) induces a map
\begin{align*}
[S(V^n)\otimes A^{n-1,n-1}(\calD)\otimes \frakp_-^*\otimes \C(-\kappa)]^{K_G}\longrightarrow [I(s_0,\chi_V)\otimes A^{n-1,n-1}(\calD)\otimes \frakp_-^*\otimes \C(-\kappa)]^{K_G},
\end{align*}
which we also denote by $\lambda$.
We define $\Psi\in [I(s_0,\chi_V)\otimes \frakp_-^*\otimes \C(-\kappa)]^{K_G}$
by
\begin{align}
\Psi\cdot \Omega^{n-1} = \lambda( \psi),
\end{align}
and write $\Psi(g,s)$ for the corresponding extension to a standard section.

\begin{corollary}
\label{cor:psitrafo}
Identify $\frakp_-^*$ with $S_\C$ as in \eqref{eq:dualid}.
For $k\in K_G$ and $g\in G$ we have
\[
\Psi(gk,s) = \det(\uk)^\kappa\cdot \uk^{-1} \Psi(g,s)\,
 {}^t\uk^{-1}.
\]
Moreover,
\begin{align*}
\Psi(1,s)= 1_n.
\end{align*}
\end{corollary}

This corollary characterizes $\Psi$ uniquely. We now use the action of $\frakp_-$ in the induced representation to find a different expression for $\Psi$.

\begin{proposition}
\label{prop:rd}
Let $D$ be the operator defined in \eqref{eq:D}, and let
\[
r(D): I(s,\chi_V)\longrightarrow I(s,\chi_V)\otimes \frakp_-^*
\]
be the induced operator on the induced representation. Then
\[
r(D)\Phi_\kappa(g,s) =\frac{1}{2} (s+\rho_n-\kappa)\Psi(g,s).
\]
\end{proposition}

\begin{proof}
We first show that $r(D)\Phi_\kappa(g,s)$ has the same $K_G$-type as $\Psi$.
Via the isomorphism \eqref{eq:dualid}, the operator $D$ induces an operator
\[
\tilde D : I(s,\chi_V)\longrightarrow I(s,\chi_V)\otimes S_\C, \quad
\tilde D=\sum_\alpha p_-(e_\alpha^\vee)\otimes e_\alpha.
\]
It satisfies
\[
\operatorname{Ad}(k)\tilde D = \uk^{-1}.\tilde D
\]
for $k\in K_G$, where the action on the left hand side is on the first factor of the tensor product and the action on the right hand side on the second factor. In fact, we have
\begin{align*}
\operatorname{Ad}(k)\tilde D &= \sum_\alpha p_-(\bar\uk.e_\alpha^\vee)\otimes e_\alpha\\
&= \sum_\alpha \sum_\beta \tr(\bar\uk.e_\alpha^\vee e_\beta)\cdot p_-(e_\beta^\vee)\otimes e_\alpha\\
&=\sum_\beta p_-(e_\beta^\vee)\otimes \sum_\alpha \tr(e_\alpha^\vee ({}^t\bar\uk.e_\beta))\cdot e_\alpha\\
&=\sum_\beta p_-(e_\beta^\vee)\otimes  \uk^{-1}.e_\beta
\end{align*}
(see also \cite[Lemma 5.1]{BFK}).
But this implies, again using the identification \eqref{eq:dualid}, that $r(D)\Phi_\kappa$ has the transformation law
\begin{align*}
r(k)r(D)\Phi_\kappa(g,s)&= r(\operatorname{Ad}(k) D) r(k) \Phi_\kappa(g,s)\\
&= \det(\uk)^\kappa \cdot \uk^{-1} \Phi_\kappa(g,s) {}^t\uk^{-1}.
\end{align*}
In other works, it has the same $K_G$-type as $\Psi$.

Since the different $K_G$ types in $I(s,\chi_V)$ have multiplicity one, there exists a constant $c(s)$ such that
\begin{align}
\label{eq:uptoc}
r(D)\Phi_\kappa(g,s) =c(s)\Psi(g,s).
\end{align}
To determine the constant, we evaluate at the unit element. According to
Corollary \ref{cor:psitrafo}, we have $\Psi(1,s)=1$. We now consider
$r(D)\Phi_\kappa$. For $X\in S_\C$ we compute $r(p_-(X))\Phi_\kappa$.
In the Lie algebra $\frakg$ we write
\begin{align}
\label{eq:xsplit}
p_-(X) = \frac{1}{2}\zxz{X}{- iX}{- i X}{-X}= \frac{1}{2} \zxz{X}{0}{0}{-X}
+\frac{i}{2} \zxz{0}{X}{-X}{0} -i \zxz{0}{X}{0}{0}.
\end{align}
We compute the actions of the three summands individually.
We have
\begin{align*}
\frac{1}{2}dr\zxz{X}{0}{0}{-X}\Phi_\kappa(1,s)&=
\frac{1}{2}\frac{d}{dt}\Phi_\kappa\left(m(e^{tX}),s\right)\mid_{t=0}\\
&=\frac{1}{2}\frac{d}{dt}\det(e^{tX})^{s+\rho_n}\Phi_\kappa(1,s)\mid_{t=0}\\
&=\frac{1}{2}(s+\rho_n)\tr(X).
\end{align*}
Next, we compute, using the action of $K_G$,
\begin{align*}
\frac{i}{2}dr\zxz{0}{X}{-X}{0}\Phi_\kappa(1,s)&=
\frac{i}{2}\frac{d}{dt}\Phi_\kappa\left(\exp t \zxz{0}{X}{-X}{0}
,s\right)\mid_{t=0}\\
&=\frac{i}{2}\frac{d}{dt}\Phi_\kappa\left(\zxz{\cos(tX)}{\sin(tX)}{-\sin(tX)}{\cos(tX)}
,s\right)\mid_{t=0}\\
&=\frac{i}{2}\frac{d}{dt}\det(e^{itX})^{\kappa}\Phi_\kappa(1,s)\mid_{t=0}\\
&=-\frac{1}{2}\kappa\tr(X).
\end{align*}
Finally, we notice that
\begin{align*}
-i dr\zxz{0}{X}{0}{0}\Phi_\kappa(1,s)&=
-i\frac{d}{dt}\Phi_\kappa\left(\exp t \zxz{0}{X}{0
}{0}
,s\right)\mid_{t=0}\\
&=-i\frac{d}{dt}\Phi_\kappa\left(n(tX)
,s\right)\mid_{t=0}\\
&=0.
\end{align*}
Putting the terms together, we obtain
\begin{align*}
r(p_-(X))\Phi_\kappa(1,s)&= \frac{1}{2}(s+\rho_n-\kappa)\tr(X),\\
r(D)\Phi_\kappa(g,s) &= \frac{1}{2}(s+\rho_n-\kappa)\cdot 1_n.
\end{align*}
This shows that the constant $c(s)$ in \eqref{eq:uptoc} is equal to $\frac{1}{2}(s+\rho_n-\kappa)$.
\end{proof}



\begin{corollary}
\label{cor:e11}
Let $e_{11}\in S$ be the matrix whose upper left entry is $1$ and whose other entries are all $0$.
We have
\begin{align*}
r(p_-(e_{11}))\Phi_\kappa(g,s)&=\frac{1}{2} (s+\rho_n-\kappa)\lambda(\psi_{11}^*)(g,s),\\
r(p_-(e_{11}))W_{T}(g, s, \Phi_\kappa)&=\frac{1}{2} (s+\rho_n-\kappa)W_T(g,s,
\lambda(\psi_{11}^*)).
\end{align*}
\end{corollary}

\begin{proof}
The first equality is a direct consequence of Proposition \ref{prop:rd}. It implies the second equality, since the Whittaker integral is an intertwining map of $(\frakg,K)$-modules.
\end{proof}

\begin{proposition}
\label{prop:e11}
For $a\in \Gl_n(\R)$ we have
\begin{align*}
r(p_-(e_{11})) W_T\left(m(a),s,\Phi_\kappa\right)=\left(2\pi \tr ( T a e_{11} \,{}^ta) -\frac{\kappa}{2}+\frac{1}{2}
\sum_{i=1}^n a_{i1}\cdot
\frac{\partial}{\partial a_{i1}}\right)
W_T\left(m(a),s,\Phi_\kappa\right).
\end{align*}
\end{proposition}

\begin{proof}
For the proof we put $X=e_{11}$ and split $p_-(X)$ as in \eqref{eq:xsplit}.
We compute the action of the three terms individually. We  have
\begin{align*}
\frac{1}{2}dr\zxz{X}{0}{0}{-X}W_T(m(a),s,\Phi_\kappa)&=
\frac{1}{2}\frac{d}{dt}W_T\left(m(a)m(e^{tX}),s,\Phi_\kappa\right)\mid_{t=0}\\
&=\frac{1}{2}\frac{d}{dt}W_T\left(m(a+taX),s,\Phi_\kappa\right)\mid_{t=0}\\
&=\frac{1}{2}\sum_{i=1}^n a_{i1}\cdot
\frac{\partial}{\partial a_{i1}} W_T(m(a),s,\Phi_\kappa).
\end{align*}
Next, we compute, using the action of $K_G$,
\begin{align*}
\frac{i}{2}dr\zxz{0}{X}{-X}{0}W_T(m(a),s,\Phi_\kappa)&=
\frac{i}{2}\frac{d}{dt} W_T\left(m(a)\exp t \zxz{0}{X}{-X}{0}
,s,\Phi_\kappa\right)\mid_{t=0}\\
&=\frac{i}{2}\frac{d}{dt}\det(e^{itX})^{\kappa}W_T(m(a), s, \Phi_\kappa)\mid_{t=0}\\
&=-\frac{\kappa}{2} \cdot  W_T(m(a), s, \Phi_\kappa) .
\end{align*}
Finally, we notice that
\begin{align*}
-i dr\zxz{0}{X}{0}{0}W_T(m(a),s,\Phi_\kappa)&=
-i\frac{d}{dt}W_T\left(m(a) n(tX)
,s,\Phi_\kappa\right)\mid_{t=0}\\
&=-i\frac{d}{dt}W_T\left(n(taX\,{}^ta)m(a)
,s,\Phi_\kappa\right)\mid_{t=0}\\
&=-i\frac{d}{dt}e(t\tr ( T aX\,{}^ta)) \mid_{t=0} W_T\left( m(a)
,s,\Phi_\kappa\right)\\
&=2\pi \tr ( T a e_{11} \,{}^ta)\cdot W_T\left( m(a)
,s,\Phi_\kappa\right).
\end{align*}
Putting the terms together, we obtain
\begin{align*}
r(p_-(e_{11})) W_T\left(m(a),s,\Phi_\kappa\right)=\left(2\pi \tr ( T a e_{11} \,{}^ta) -\frac{\kappa}{2}+\frac{1}{2}
\sum_{i=1}^n a_{i1}\cdot
\frac{\partial}{\partial a_{i1}}\right)
W_T\left(m(a),s,\Phi_\kappa\right).
\end{align*}
 \end{proof}

\begin{corollary}
\label{cor:key}
Assume that $n=m+1$ and $\det(T)\neq 0$.
Write $T$ in block form as in \eqref{eq:varblock}, and recall the definition \eqref{eq:da} of $d(a)$. For $a\in \R_{>0}$ we have
\begin{align*}
W_{T }(m(d(a)),0,\lambda(\psi_{11}^*))
&=2 \left( 2\pi T_1 a^2-\frac{\kappa}{2}+\frac{a}{2}\frac{\partial}{\partial a}\right) W'_T(m(d(a)),0,\Phi_\kappa).
\end{align*}
\end{corollary}

\begin{proof}
Using Corollary \ref{cor:e11}
and Proposition \ref{prop:e11},
 we see
\begin{align*}W_T(m(d(a)),s,\lambda(\psi_{11}^*))
&=2(s+\rho_n-\kappa)^{-1}\cdot
\left(2\pi T_1 a^2-\frac{\kappa}{2}+\frac{a}{2}
\frac{\partial}{\partial a}\right)
W_T\left(m(d(a)),s,\Phi_\kappa\right)
.
\end{align*}
Since $n=m+1$, we have $\rho_n=\kappa$. Moreover, because of the signature of $V$, the matrix $T$ is not positive definite. Hence, according to Proposition \ref{prop:whitt0}, the Whittaker function on the right hand side vanishes at $s=0$. This implies the assertion.
\end{proof}


\subsection{The main term of the local height revisited}
Here we combine the results of the previous two subsections with the asymptotic properties of Whittaker functions derived in Section~\ref{sect:4.2}.

\begin{proof}[Proof of Theorem \ref{thm:pre-alsw}]
Recall that $n=m+1$, $x\in V^n$, and that $T=Q(x)$ is invertible. We have to show that
\[
e^{-2\pi \tr T}\cdot \Ht_\infty^\main(x)=-B_{n,\infty}\cdot W_T'(1,0,\Phi_\kappa)
+B_{n-1,\infty} e^{-2\pi (\tr T -\tr \tilde T_2)} \cdot W'_{\tilde T_2}(1,0,\Phi_{\kappa-1/2}),
\]
where $\tilde T_2$ is defined by \eqref{eq:tt}.
According to Proposition \ref{prop:3.8} we know that
\[
\Ht_\infty^\main(x) = B_{n,\infty} e^{2\pi \tr T_2}
\int_1^\infty W_{d(a)T d(a)}(1,0,\lambda(\psi_{11}^*))\cdot e^{2\pi Q(x_1 a)
}\,\frac{da}{a}.
\]
Inserting \eqref{eq:i1} and the formula of Corollary \ref{cor:key}, we obtain
\begin{align*}
\Ht_\infty^\main(x) &= 2 B_{n,\infty} e^{2\pi \tr T_2}
\\&\phantom{=}{}\times
\int_1^\infty
a^{-\rho_n}\left( \left( 2\pi T_1 a^2-\frac{\kappa}{2}+\frac{a}{2}\frac{\partial}{\partial a}\right) W'_T(m(d(a)),0,\Phi_\kappa)\right)
e^{2\pi Q(x_1 a)}\,\frac{da}{a}.
\end{align*}
Noticing that
\begin{align*}
&2a^{-\rho_n-1}e^{2\pi Q(x_1 a)} \left( 2\pi T_1 a^2-\frac{\kappa}{2}+\frac{a}{2}\frac{\partial}{\partial a}\right) W'_T(m(d(a)),0,\Phi_\kappa)\\
&= \frac{\partial}{\partial a} \left( W'_T(m(d(a)),0,\Phi_\kappa)  e^{2\pi Q(x_1 a)} a^{-\rho_n}\right),
\end{align*}
we find
\begin{align*}
\Ht_\infty^\main(x) &=  B_{n,\infty} e^{2\pi \tr T_2}
\int_1^\infty \frac{\partial}{\partial a} \left( W'_T(m(d(a)),0,\Phi_\kappa)  e^{2\pi Q(x_1 a)} a^{-\rho_n}\right)\,da\\
&=  B_{n,\infty} e^{2\pi \tr T_2}\left(
-W_T'(1,0,\Phi_\kappa)e^{2\pi Q(x_1)}+\lim_{a\to \infty}
W'_T(m(d(a)),0,\Phi_\kappa)  e^{2\pi Q(x_1 a)} a^{-\rho_n}\right)\\
&= -B_{n,\infty} e^{2\pi \tr T}\cdot W_T'(1,0,\Phi_\kappa)
\\ &\phantom{=}{}
+  B_{n,\infty} e^{2\pi \tr T_2}\left(\lim_{a\to \infty}
W'_T(m(d(a)),0,\Phi_\kappa)  e^{2\pi Q(x_1 a)} a^{-\rho_n}\right).
\end{align*}
We now employ Corollary
\ref{cor:wderivasy} to evaluate the limit on the right hand side.
We obtain
\begin{align*}
\Ht_\infty^\main(x)&= -B_{n,\infty} e^{2\pi \tr T}\cdot W_T'(1,s_0,\Phi_\kappa) \\
&\phantom{=}{}+
\begin{cases} \displaystyle\frac{B_{n,\infty} i^{-n}(2\pi)^{\rho_n}}{\Gamma(\rho_{n})}
\cdot e^{2\pi (\tr T_2 -{}^tT_{12}T_1^{-1}T_{12})}\,
W'_{\tilde T_2 }
(1,0,\Phi_{\kappa-1/2}) ,&\text{if $T_1>0$,}\\[2ex]
0,&\text{if $T_1\leq 0$.}
\end{cases}
\end{align*}
Hence the claim follows from \eqref{eq:Bquotient}.
\end{proof}

\subsection{An alternative proof of Proposition \ref{prop:LocalSiegel-WeilR}}

\label{sect:5.1}

Here we use Corollary \ref{cor:e11} and Proposition \ref{prop:geolsw} to give an alternative way of computing the constant $B_{n,\infty}$ appearing in Proposition
~\ref{prop:LocalSiegel-WeilR}.
Assume that $n=m+1$.
Let $\phi_\infty(x, z) \in S(V^n)\otimes C^\infty(\calD)$ with $\phi_\infty(hx, hz) = \phi_\infty(x, z)$ for all $z \in \mathcal D$, $x \in V^n$ and $h \in H(\R)$.
Then by Theorem \ref{theo:LocalSiegel-Weil} we know that
\begin{align}
\label{eq:lsw-1}
\int_{\mathcal D} \phi_\infty(x, z) \,\Omega^m =B_{n, \infty}  \cdot  W_{T}(1, 0, \lambda(\phi_\infty))
\end{align}
for {\em some} non-zero constant $B_{n,\infty}$, which is independent of $\phi_\infty(x,z)$ and $T=Q(x)$. To compute $B_{n,\infty}$ we pick the special Schwartz function
\[
\psi_{11}^*(x,z)=\varphi_G(x_1,z)\cdot \varphi_{KM}^{n-1,*}((x_2,\dots,x_n),z)
\]
as in \eqref{eq:psi11}.
Evaluating \eqref{eq:lsw-1} in the limit $x_1\to 0$ and using the fact that $\varphi_G(0,z)=1$, we obtain
\begin{align}
\label{eq:lsw-2}
\int_{\mathcal D} \varphi_{KM}^{m}((x_2,\dots,x_n),z) =B_{n, \infty}  \cdot  W_{T}(1, 0, \lambda(\psi_{11}^*)),
\end{align}
where $T=\kzxz{0}{0}{0}{T_2}$ and $T_2=Q((x_2,\dots,x_n))$.
The left hand side of \eqref{eq:lsw-2} is given by Proposition \ref{prop:geolsw}. If $T_2>0$ we have
\begin{align}
\label{eq:deg0}
\int_{\mathcal D} \varphi_{KM}^{m}((x_2,\dots,x_n),z) = 2e^{-2\pi \tr T_2}.
\end{align}
We now compute the right hand side of \eqref{eq:lsw-2}.

\begin{lemma}
\label{lem:singwhitt}
For $T=\kzxz{0}{0}{0}{T_2}$ with $T_2>0$ as above, we have
\begin{align*}
W_{T}(1, 0, \lambda(\psi_{11}^*))&=-\frac{
2^{\frac{n}{2}(\rho_n +1)}}{\Gamma_n(\rho_n) }\left(\frac{\pi}{i}\right)^{n\rho_n}
e^{-2\pi \tr T_2}.
\end{align*}
\end{lemma}

\begin{proof}
We first show that for $T=\kzxz{0}{0}{0}{T_2}$, we have
\begin{align}
\label{eq:deg1}
W_{T}(1, s, \lambda(\psi_{11}^*)) = -W_{T}(1, s, \Phi_\kappa).
\end{align}
In fact,
using the notation of Corollary \ref{cor:e11} and Proposition \ref{prop:e11}, we have
\begin{align*}W_T(1,s,\lambda(\psi_{11}^*))
&=\frac{2}{s}
\cdot \big(
r(p_-(e_{11}))W_{T}(g, s, \Phi_\kappa)\big)\mid_{g=1}\\
&=\frac{2}{s}\cdot
\left(2\pi \tr ( T e_{11} \,{}^t) -\frac{\kappa}{2}+\frac{1}{2}
\frac{\partial}{\partial a_{11}}\right)
W_T\left(m(a),s,\Phi_\kappa\right)\mid_{a=1}
.
\end{align*}
Here we have also used that $n=m+1$ and therfore $\rho_n=\kappa$.
Taking into account the  special form of $T$ and the transformation law \eqref{eq:i1}, we deduce
\begin{align*}
W_T(1,s,\lambda(\psi_{11}^*))
&=\frac{2}{s}\cdot
\left(-\frac{\kappa}{2}+\frac{1}{2}
\frac{\partial}{\partial a_{11}}\right)\big(a_{11}^{\rho_n-s}
W_T\left(1,s,\Phi_\kappa\right)\big)\mid_{a_{11}=1}\\
&=
- W_T\left(1,s,\Phi_\kappa\right)
.
\end{align*}

Next we compute $W_T(1,0,\Phi_\kappa)$ for $T=\kzxz{0}{0}{0}{T_2}$. 
According  \cite[(4.6.K)]{Sh} and \eqref{eq:weta} the function $\omega(g,h,\alpha,\beta)$ of \cite{Sh} satisfies
\begin{align*}
\omega(2\cdot 1_n,\pi T,\alpha,\beta)&=
2^{-(n-1)\alpha} \frac{|2\pi T_2|^{\rho_n-\alpha}|2\cdot 1_n|^{\alpha+\beta-\rho_n}}{\Gamma_{n-1}(\beta-\frac{1}{2})\Gamma_1(\alpha+\beta-\rho_n)}\eta(2\cdot 1_n,\pi T,\alpha,\beta)\\
&=
2^{-(n-1)\alpha} \frac{|2\pi T_2|^{\rho_n-\alpha}|2\cdot 1_n|^{\alpha+\beta-\rho_n}}{c_n(\alpha,\beta)\Gamma_{n-1}(\beta-\frac{1}{2})\Gamma_1(\alpha+\beta-\rho_n)}W_T(1_n,s,\Phi_\kappa)\\
&= 2^{-\rho_n(\frac{n}{2}+1)+n\beta-\frac{n}{2}+\alpha(2-n)}
\pi^{\alpha(1-n)-1} i^{n(\alpha-\beta)} {{|T_2|}}^{\rho_n-\alpha}\\
&\phantom{=}{}\times\frac{\Gamma(\beta)\Gamma_n(\alpha)}{\Gamma(\alpha+\beta-\rho_n)} W_T(1_n,s,\Phi_\kappa).
\end{align*}
Here, in the latter equality, we have also used \eqref{eq:mg2}.
We find
\begin{align*}
\omega(2,\pi T,\rho_n,0)&= 2^{-\frac{3}{2}n\rho_n +\rho_n-\frac{n}{2}}
\pi^{\rho_n(1-n)-1} i^{n\rho_n}\Gamma_n(\rho_n) W_T(1,0,\Phi_\kappa).
\end{align*}
On the other hand, by \cite[Theorem 4.2]{Sh} and  \cite[(4.35.K)]{Sh}, we have
\begin{align*}
\omega(2,\pi T,\rho_n,0)&= \omega(2,\pi T, \rho_n+\frac{1}{2}, \frac{1}{2})=  2^{-(n-1)\rho_n}\pi^{\frac{n-1}{2}}e^{-2\pi \tr T} ,
\end{align*}
and therefore
\begin{align}
\label{eq:deg2}
e^{-2\pi \tr T} = 2^{-\frac{n}{2}(\rho_n +1)}
\pi^{-n\rho_n } i^{n\rho_n}\Gamma_n(\rho_n) W_T(1,0,\Phi_\kappa).
\end{align}
Putting this identity into \eqref{eq:deg1}, we obtain the assertion.
\end{proof}

Combining \eqref{eq:lsw-2}, \eqref{eq:deg0},  and Lemma \ref{lem:singwhitt}, we find
\begin{align*}
B_{n,\infty} = -
2^{\frac{4-n^2-3n}{4}}\left(\frac{i}{\pi}\right)^{n\rho_n}\Gamma_n(\rho_n) .
\end{align*}
In particular, we have $B_{n,\infty}/B_{n-1,\infty}= i^n\frac{\Gamma(\rho_n)}{(2\pi)^{\rho_n}}$, and $B_{1,\infty}= \frac{1}{\pi i}$, and $B_{2,\infty}= \frac{i}{4\sqrt{2}\pi^2}$.


\section{The local arithmetic Siegel-Weil formula at an
odd prime $p$}
\label{sect:LASW-finite}


In this section we assume that $p \ne 2$ is a prime.
Let $W=W(\kay)$ be the Witt ring of $\kay$ and $\K= W_{\Q} $ be the fraction field of $W$,  which is the completion of the maximal unramified extension of $\Q_p$. Let $\sigma$ be the Frobenius of $W$ (such that its reduction to $\kay$ is the Frobenius $x\mapsto x^p$).

Let $L$ be a unimodular quadratic lattice over $\Z_p$ of rank $n+1$ and put $V=L_{\Q_p}$. Let $C(L)$ be the Clifford algebra of $L$, and let  $D(L) =\Hom (C(L), \Z_p)$ be its dual. We write
$\HH =\GSpin(L)$ for the general Spin group over $\Z_p$, and notice that $\HH(\Z_p) \subset C(L)^\times$ acts on $C(L)$ via left multiplication and and thus acts on $D_L$.
Let $\iota$ be the main involution on  $C(V)$ which fixes $V$ point-wise. If $\delta \in  C(V)^\times$ with $\delta^\iota=-\delta$, then $\psi_\delta(x, y) =\tr (x \delta y^\iota)$ defines a  non-degenerate  symplectic form on $C(V)$. We will require that $\delta \in  C(L)$ and $\delta \delta^\iota \in \Z_p^\times$, which implies  that $C(L)$ is unimodular under this symplectic form. This induces an embedding
\begin{equation} \label{eq:embedding}
i= i_\delta: \,   \HH \rightarrow  \hbox{GSp}(C(L)).
\end{equation}
It is also known that  $\HH$ is `cut out' by a family tensors
 $(s_\alpha)$,  $s_\alpha \in C(L)^\otimes$, in the sense  that for any $\Z_p$ algebra $R$ we have
$$
\HH(R) =\{ h \in  \GSp(C(L))(R)\mid\; h s_\alpha =s_\alpha\}.
$$

\subsection{The local unramified Shimura datum and the Rapoport-Zink space  associated to $\tilde H$}

\label{sect:RZSpace}

Here we set up some notation for the rest of this section. We recall the construction of an unramified local Shimura datum for $\tilde H$ due to Howard and Pappas, and the associated Rapoport-Zink space. We refer to \cite{HP} for details.

Choose a $\Z_p$-basis $e =\{ e_1, \dots, e_{n+1}\}$ of $L$ with Gram matrix
 \begin{equation} \label{eq6.4}
( (e_i, e_j)) = \diag(I_{n-2}, \epsilon_L,  \begin{pmatrix} 0 &1 \\ 1& 0 \end{pmatrix}),
\end{equation}
with $\epsilon_L =-\det L$.  Define
\begin{align}
\mu&:  \mathbb G_m \rightarrow \HH, \quad t \mapsto  \mu(t) = t^{-1} e_{n} e_{n+1} + e_{n+1} e_n \in \HH,  \label{eq6.5}
\\
 b&=  e_{n-1} (p^{-1} e_n + e_{n+1}) \in \HH(\Q_p) \subset \HH(\K). \label{eq6.6}
\end{align}
Then $(\tilde H, [b], \{\mu\}, C(L))$ is the local unramified Shimura datum constructed by Howard and Pappas in \cite[Section 4]{HP} for $\tilde H$. Here $\{\mu\}$ is the conjugacy class of the cocharacter $\mu$ under $\tilde H(\K)$, and $[b]$ is the $\sigma$-conjugacy class of the basic element $b$, i.e., the set of elements $h^\sigma b h^{-1}$ with $h \in  \HH(\K)$.  Associated to $b$ there are two isocrystals
$$
(V_\K=V\otimes_{\Q_p}\K, b \circ \sigma),   \quad \hbox{ and } \quad  (D_\K = D(L)\otimes_{\Z_p}\K, b \circ \sigma).
$$
Let
$$
\newV =V_\K^{b\circ \sigma},  \quad \hbox{ and } \quad \newL = (L \otimes_{\Z_p} W)^{b \circ \sigma}.
$$
A direct calculation shows that
 $\newL$ has a $\Z_p$-basis $e'=\{ e_1', \dots, e_{n+1}'\}$ with Gram matrix
\begin{equation} \label{eq:LatticeLambda}
( (e_i', e_j')) = \diag( I_{n-2}, \epsilon_{\newL}, p, -p u)
\end{equation}
where $u \in \Z_p^\times$ with $(p, u)=-1$, and $ -u\epsilon_{\newL}= \epsilon_L$. We can actually take $e_i'=e_i$ for $i\le n-2$.  In particular, $\newV=\newL \otimes_{\Z_p} \Q_p$
is a quadratic space over $\Q_p$ with the same dimension and the same determinant, but with  opposite Hasse invariant as $V$.  

According to \cite[Lemma 2.2.5]{HP}, there is a unique $p$-divisible group
\[
\mathbb X_0=\mathbb X_0(\tilde H, [b], \{ \mu \}, C(L))
\]
over $\kay$  whose contravariant Dieudonn\'e module is $\mathbb D (\mathbb X_0)(W)  \cong D_W =D(L) \otimes_{\Z_p} W$ with Frobenius
$F=b \circ \sigma$. Moreover, the Hodge filtration  on $\mathbb D(\mathbb X_0)(\kay)$ is induced by $\mu_{\kay}$ (up to conjugation). The symplectic form $\psi_\delta$ induces a principal polarization $\lambda_0$ on $\mathbb X_0$.

Let $\RZ(\mathbb X_0, \lambda_0)$ be the Rapoport-Zink space  associated to $\hbox{GSp}(C(L), \psi_\delta)$, see \cite{RZ} and \cite[Section 2.3]{HP}.
It  is
a smooth formal scheme over $\hbox{Spf}(W)$ representing  the moduli problem over $\hbox{Nilp}_W$ of triples $(X, \lambda, \rho)/S$,   where $S$ is a formal scheme over $W$ on which  $p$ is Zariski locally nilpotent, $(X, \lambda)$ is a $p$-divisible group with principal polarization $\lambda$, and $\rho$ is a quasi-isogeny
$$
\rho:  \mathbb X_0 \times_{\kay} \bar{S} \dashrightarrow X\times_S \bar S, \quad  \bar S = S\times_W \kay,
$$
which respects polarization  up to a scalar,  in the sense that Zariski locally on $\bar S$, we have
$$
\rho^\vee \circ \lambda \circ \rho= c(\rho)^{-1} \lambda_0, \quad c(\rho) \in \Q_p^\times.
$$
Let
$\RZ= \RZ(\tilde H,  [b], \{\mu\},  C(L))$ be  the GSpin Rapoport-Zink space constructed in \cite[Section 4]{HP}. This space comes with a closed immersion $\RZ\to \RZ(\mathbb X_0, \lambda_0)$, and by restricting the universal object one obtains a
is a universal triple $(X^{\univ}, \lambda^{\univ} , \rho^{\univ})$ over $\RZ$.  The universal quasi-isogeny preserves the polarization  only up to scalar, which  induces  a decomposition of $\RZ$ as a union of  open and closed formal subschemes
\begin{equation}
\RZ =\bigsqcup_l \RZ^{(l)},
\end{equation}
where $\RZ^{(l)} \subset \RZ$ is cut out by the condition $\ord_p c(\rho^{\univ}) =l \in \Z$.
 According to \cite[Section 7]{HP} (see also Section \ref{sect:ASW-finite} here), $\RZ$  can be used to uniformize the supersingular locus at $p$ of some Shimura variety associated with $(\tilde H,\calD)$.

Notice that $V_{\K}$ acts on $C(V)_{\K}$ via right multiplication, which induces an action on the isocristal $D_{\K} $.  This gives an embedding  $V_{\K} \subset  \End (D_{\K})$.
Moreover,   $\newV=V_{\K}^{b \circ \sigma} \subset \End (D_{\K})$  commutes with the Frobenius $F= b\circ \sigma $. Since  $D_\K\cong \mathbb D(\mathbb X_0)(\K)$,  we obtain an embedding  $\newV  \subset \End^0(\mathbb X_0)$. We call $\newV$ the {\em special endomorphism space} of $\mathbb X_0$ following \cite{HP}.

Let $\tilde\newH$ be the algebraic group $\GSpin(\newV)$. Then
 $\tilde\newH(\Q_p)=\{ h\in \tilde H(\K)\mid \; hb =b\sigma(h)\}$ acts by automorphisms on $D_\K $, giving rise to a quasi-action  on  $\mathbb X_0$. This quasi-action has the property
$$
c(h \rho) = \mu_{\tilde \H}(h) c(\rho),
$$
where $\mu_{\tilde \H}$ is the spin character of $\tilde \H$.
So $h \in \tilde\H(\Q_p)$ induces an isomorphism $\RZ^{(l)} \cong \RZ^{(l+\ord_p \mu_{\tilde\H}(h))}$.  In  particular, we have
\begin{equation}
p^\Z \backslash \RZ \cong \RZ^{(0)} \bigsqcup \RZ^{(1)}.
\end{equation}
According to \cite[Corollary 7.8]{Shen},  $\overline{\RZ}=p^\Z \backslash \RZ$ is exactly the Rapoport-Zink space of $H$ associated to the basic local unramified Shimura datum induced from the datum $(\tilde H, [b],\{\mu\}, C(L))$.

Finally, let $ J \subset \newV$ be an  integral $\Z_p$-submodule of rank $ 1 \le r \le n$. We define the special cycle $\mathcal Z(J)$, following Soylu \cite{Cihan-thesis}, as
  the formal subscheme of $\RZ$ cut out by the condition
\begin{equation} \label{eq:SpecialCycle}
\rho \circ J \circ \rho^{-1} \subset  \End(X) .
\end{equation}
Here, for an $S$-point $\alpha:  S \rightarrow \RZ$,  $X = \alpha^*(X^{\univ})$ and $\rho =\alpha^*(\rho^{\univ})$ are the pull-backs of  the universal objects.

If $J$ has a ordered $\Z_p$-basis $x =(x_1, \dots, x_r) \in \newV^r$,
we also denote  $\mathcal Z(J) =\mathcal Z(x)$. The moment matrix
$T =Q(x) = \frac{1}2( (x_i, x_j)) $ in $\Sym_{r}(\Q_p)$ is
determined by $J$ up to $\Z_p$-equivalence. Soylu gave an explicit formula of the dimension of the reduced scheme $\mathcal Z(J)^{\red}$ underlying
$\mathcal Z(J)$ in terms of $T$ and $L$, see \cite[Section 4.2]{Cihan-thesis}.

The purpose of this section is to prove a local arithmetic Siegel-Weil formula for $\mathcal Z(J)$. We show that when
$\mathcal Z(J)$ is $0$-dimensional, the local height of each  point $P \in \mathcal Z(J)$ depends only on $T$, not on the choice of the point $P$, and is equal to the central derivative of some local Whittaker function (Theorem \ref{theo:LASW-finite}). 

\subsection{Dual vertex lattices and decomposition of the Rapoport-Zink space}
\label{sect:6.2}

A $\Z_p$-lattice $\Lambda \subset \newV=V_\K^{b\circ\sigma}$ is called a \emph{dual vertex} lattice if $p\Lambda' \subset \Lambda \subset \Lambda'$, i.e., its dual   $\Lambda'$ is a vertex lattice in the sense of  \cite{HP}. Let
$$
\Omega_\Lambda=\Lambda'/\Lambda, \quad Q_\Lambda(x) =pQ(x) \mod \Z_p,
$$
be the associated quadratic space over $\mathbb F_p$. Then $t_\Lambda =\dim_{\mathbb F_p} \Omega_\Lambda$ is called the type  number of $\Lambda$.
Let  $\bar{\Omega}_\Lambda =\Omega_\Lambda \otimes_{\mathbb F_p}\kay$.
According to \cite[Section 5.3]{HP},
 there is a projective variety $S_{\Lambda}$ over $\kay$ such that
$$
S_\Lambda(\bar{\mathbb F}_p) = \left\{ \mathfrak L \subset \bar \Omega_\Lambda:\; \text{maximal isotropic and  $\dim(\mathfrak L + \hbox{Frob}(\mathfrak L)) =t_\Lambda/2 +1$}\right\}.
$$
 Moreover, $S_\Lambda =S_{\Lambda}^+ \cup S_{\Lambda}^-$ has two connected components, both of which are smooth  and projective of dimension $t_{\Lambda}/2 -1$

For a dual vertex lattice $\Lambda$ of $\newV$, let $\RZ_\Lambda$ be the closed formal subscheme of $\RZ$  defined  by the condition
$$
\rho \circ \Lambda \circ \rho^{-1} \subset \End(X).
$$
The following theorem summarizes some of the basic properties of $\RZ_\Lambda$ and $\RZ$.  Assertions (1), (3), (4) are due to Howard and Pappas (see \cite{HP}, Proposition~5.1.2, Section~6.5, Remark~6.5.7). The second assertion is due to Li and Zhu \cite[Theorem 4.2.11]{LZ}.

\begin{theorem} The following are true.
\begin{enumerate}

\item
For a dual vertex lattice $\Lambda$, $t_\Lambda$ is even and
 $$
 t_\Lambda \le t_{\text{\rm max}} =
  \begin{cases}
    n &\ff n \hbox{ is even},
    \\
    n-1  &\ff n \hbox{ is odd},  \det L = (-1)^{\frac{n+1}2} ,
    \\
    n+1  &\ff n \hbox{ is odd},  \det L \ne (-1)^{\frac{n+1}2} .
    \end{cases}
 $$
Moreover,  $\Omega_\Lambda$ is the unique non-split space over $\mathbb F_p$ of dimension $t_\Lambda$, and every dual vertex lattice contains a `minimal' dual vertex lattice with $t_\Lambda =t_{\text{\rm max}}$. Moreover, $\RZ_\Lambda^{\red}$ is of  dimension $\frac{t_\Lambda}2-1$.

\item
The formal scheme $\RZ_\Lambda$ is reduced.

\item
One has
$$
\RZ= \bigcup_{t_\Lambda =t_{\text{\rm max}} } \RZ_\Lambda,
$$
and
$$
\RZ^{\red} =\bigsqcup_{\Lambda} \operatorname{BT}_\Lambda, \quad  \operatorname{BT}_\Lambda =\RZ_\Lambda- \bigcup_{\Lambda_1 \subsetneq \Lambda} \RZ_{\Lambda_1}.
$$

\item
Let $\overline{\RZ}_\Lambda =p^\Z \backslash \RZ_\Lambda$. Then
$$
\overline{\RZ}_\Lambda \cong S_\Lambda
$$
as projective varieties over $\kay$.
\end{enumerate}
\end{theorem}

\begin{proposition}
\label{prop:SpecialLattice}  Up to $\Z_p$-isomorphism, there is a unique  dual vertex lattice $\Lambda(t)$ of type number $t$ for every even integer  $ 1< t=2r \le t_{\text{\rm max}}$, which is given by     $\Lambda(t) = \bigoplus \Z_p f_i$ with Gram matrix
\begin{equation}  \label{eq: StandardBasis}
((f_i, f_j)) = \diag(I_{n-t}, \alpha,  pI_{t-1}, p \beta )
\end{equation}
with  $\alpha, \beta \in \Z_p^\times/(\Z_p^{\times})^2$ satisfying
$$
(p,  (-1)^r \beta) =-1,  \quad  \hbox{ and  } \quad  \alpha \beta  =\det L  \mod (\Z_p^{\times})^2.
$$
In particular,  $\Lambda(2)$ is the lattice given by  (\ref{eq:LatticeLambda}).
\end{proposition}
\begin{proof} Since
$$
p\Lambda' \subset \Lambda \subset \Lambda',
$$
we see that $\Lambda = \oplus \Z_p e_i$ with
$$
Q(\sum x_i e_i )= \sum \alpha_i p^{a_i} x_i^2
$$
with  $\alpha_i \in \Z_p^\times$ and $ 0 \le a_1 \le a_2 \le \dots\le a_{n+1}  \le 1$. The condition  $t_\Lambda =\dim_{\mathbb F_p} \Lambda'/\Lambda =t$ implies $a_1=\dots =a_{n-t+1} =0$, and $a_{n-t+2} =\dots =a_{n+1} =1$. So we can change the basis to make  (\ref{eq: StandardBasis}) true.  Since $L$ is unimodular, $V$ has Hasse invariant $1$, and hence $\newV$ has Hasse invariant $-1$, i.e.,
$$
-1 =(2p, (2p)^{t-1} \beta) (2p, (2p)^{t-2}\beta) \cdots (2p, 2p\beta) = (2p, (2p)^r \beta) = (p, (-1)^r \beta).
$$
In particular,  $E=\Q_p(\sqrt{(-1)^r \beta})$ is  the unique unramified quadratic field extension   of $\Q_p$, and $\beta$ is uniquely determined up to a square by this condition.  On the other hand, $\det \newV =\det V$ gives
$$
 \alpha \beta =\det L  \mod (\Z_p^{\times})^2,
$$
which then determines $\alpha$  uniquely up to a square.
\end{proof}

\subsection{Special cycles and local heights}

Recall the definition of the special cycle $\mathcal Z(J)$ at the end of Section \ref{sect:RZSpace}. It is not hard to see \cite[Section 4.2]{Cihan-thesis} that
$$
\mathcal Z(J)^{\red} = \bigcup_{\substack{\text{$\Lambda$ dual vertex lattice}\\ J \subset \Lambda}}  \RZ_\Lambda^{\red}.
$$

 The following theorem  is part of \cite[Theorems 4.13, 4.16 and Proposition  4.15 ]{Cihan-thesis}. (Recall our convention that $\dim V=n+1$ and notice that our $2T$ is Soylu's $T$.)

\begin{theorem} (Soylu) \label{theo:Soylu}
Let $\mathcal Z(J)^{\red}$ be the reduced scheme of $\mathcal Z(J)$ and assume that $J=J(x_1, \dots, x_n)\subset \newV $ has rank $n$ and is integral.  Assume that  $T =Q(x)$
is $\Z_p$-equivalent to $\diag(T_1, T_2)$ where $T_1$ is unimodular of rank $r=r(T)$ (which is also the rank of $T\pmod p$ over $\kay$), and $T_2 \in p\Sym_{n-r}(\Z_p)$.
Then  $\mathcal Z(J)^{\red}$ is $0$-dimensional  if and only if  one of  the following conditions  holds:
\begin{enumerate}
\item  $r(T)=n-1, n-2$.

\item    $r(T)= n-3$, and  $\det (2 T_1)  = \det L$.
\end{enumerate}
In such a case,
$$
\mathcal Z(J)^{\red} =\bigsqcup_{\substack{ J \subset \Lambda \\ t_\Lambda=2}} \RZ_\Lambda = \bigsqcup_{\substack{ J \subset \Lambda \\ \Lambda \cong \Lambda(2) }} \RZ_\Lambda.
$$
\end{theorem}
\begin{proof} We give a sketch of the proof in this special case to give a rough idea what is involved in the general theorems of
Soylu \cite[Section 4]{Cihan-thesis}. Choose a $\Z_p$-basis $e=\{ e_1, \dots, e_n\}$ of $J$ with $\frac{1}2 (e_i, e_j) =\diag(T_1, T_2)$, and let $M_1$ be the submodule of $J$ generated by $e_1, \dots, e_r$, which is unimodular.  To have $J \hookrightarrow \Lambda$, one has to have $t_\Lambda  \le n-r+1$. In the case $r=n-1, n-2$, one has $t_\Lambda =2$. and $\RZ_\Lambda =\RZ_\Lambda^{\red}$ is reduced of dimension $0$. So
$$
\mathcal Z(J)^{\red}  =\bigsqcup_{\substack{ J \subset \Lambda \\ t_\Lambda=2}} \RZ_\Lambda .
$$

In the case $r(T) =n-3$, one might have $t_\Lambda =2$ or $4$.  If $t_\Lambda =4$, then  (as $M_1$ is unimodular)
$$
\Lambda \cong M_1 \oplus \Lambda_2,
$$
where $\Lambda_2$ has a $\Z_p$-basis with Gram matrix $p\diag(1, 1, 1, \epsilon)$. Since the Hasse invariant of $\newV$ is $-1$, we see $(p, -\epsilon) =-1$. On the other hand, $\det \newV =\det  V$ forces
$$
\epsilon = \det M_1  \det L =\det (2T_1) \det L,  \hbox{ i.e., } \det (2T_1) \ne \det L .
$$
Therefore, if $\det (2T_1) =  \det L$,  we cannot embed $J$ into a dual vertex lattice $\Lambda$ with $t_\Lambda =4$, and thus $\mathcal Z(J)$ is $0$-dimensional and reduced as argued above.

When $\det(2T_1) \ne \det L$, Soylu proved that there is indeed some embedding $J\subset \Lambda$ with $t_\Lambda =4$. We refer to \cite[Section 4]{Cihan-thesis} for the details.
\end{proof}

Let $M_1$ be a unimodular  quadratic $\Z_p$-lattice of rank $r<n-2$, and assume that there are isometric embeddings $M_1\subset L$ and $ M_1 \subset \Lambda$,
where $\Lambda \subset \newV$ is a dual vertex lattice.  Write
\begin{equation} \label{eq:SplitLattice}
L=M_1 \oplus L_2, \quad  \Lambda =M_1 \oplus \Lambda_2.
\end{equation}
Notice that, choosing proper bases of $M_1$ and $L$, the data $b$ and $\mu$  defined in  \eqref{eq6.5} and \eqref{eq6.6} still make sense for the unimodular lattice $L_2$, so we have a local unramified Shimura datum $(\HH(r), [b], \{\mu\}, C(L_2))$ and its associated Rapoport-Zink space $\RZ(r)$. Here $\HH(r) = \hbox{GSpin}(L_2)$. Moreover,  one can easily check that $\newV_2 ={L_2}_\K^{b \circ \sigma}$ is a direct summand of $\newV$, and $\Lambda_2$ is a dual vertex lattice of $\newV_2$.
The embdding $L_2 \subset L$ induces a closed immersion
\begin{equation}
i(r):  \RZ(r) \hookrightarrow \RZ.
\end{equation}
The following proposition is a direct consequence of \cite[Lemma 3.1.1]{LZ}.

\begin{proposition} (Li-Zhu)  Let the notation be as above and assume $r \le n-3$. 
\begin{enumerate}
\item One has
$$
i(r) \RZ(r) = \mathcal Z(M_1).
$$

\item Assume that $J=M_1 \oplus J_2$ is a $\Z_p$-submodule of $\Lambda$. Then
$$
i(r) \mathcal Z_{\RZ(r)} (J_2) = \mathcal Z_{\RZ}(J).
$$
\end{enumerate}
\end{proposition}
\begin{proof} Assume that $\{ x_1, \dots, x_r\}$ is a basis of $M_1$ with Gram matrix $\diag(\alpha_1, \dots, \alpha_r)$ and $\alpha_i \in \Z_p^\times$. Applying \cite[Lemma 3.1.1]{LZ} $r$-times, we obtain the above proposition. Notice that their lemma still holds with the same proof when the norm of $x_n$ is a unit in $\Z_p$ (not necessarily equal to $1$).
\end{proof}

From now on, we assume that $\mathcal Z(J)$ is $0$-dimensional. For $P \in  \mathcal Z(J)$, its local height index is defined to be
\begin{equation}
\label{eq:htp}
\Ht_p(P) = \hbox{the length of the  formal complete local ring } \widehat{\mathcal O}_{\mathcal Z(J), P}.
\end{equation}

By Theorem \ref{theo:Soylu}, we have $r(T) \ge n-3$. There is a decomposition
$$
J =M_1 \oplus J_2
$$
with $M_1$ unimodular of rank $n-3$.  Furthermore we can choose bases of $M_1$ and $J_2$ so that the  Gram matrix of $J$ becomes $2T$ with $T = \diag(T_1, T_2)$ where $2T_1$ is the Gram matrix of $M_1$ and $T_2= \diag(\alpha_1 p^{a_1}, \alpha_2 p^{a_2}, \alpha_3 p^{a_3})$ is the matrix of $J_2$ with $\alpha_i \in \Z_p^\times$ and $0\le a_1 \le a_2 \le a_3$. We can always embed $M_1$ into $L$. Assume that $\mathcal Z(J)(\kay) $ is not empty. Then there is an embedding $M_1 \subset J \subset \Lambda$ for some dual vertex lattice $\Lambda$.

\begin{corollary} \label{cor:LiZhu}
Let the notation and hypotheses be as above (in particular $\mathcal Z(J)$ is $0$-dimensional).
For  $P \in  \mathcal Z(J)$  we write $P^*= i(n-3)^{-1}P \in \RZ(n-3)$. Then
$$
\Ht_p(P)  = \Ht_p(P^*).
$$
\end{corollary}

 The local height $\Ht_p(P^*)$ has been studied in  \cite{KRHilbert} (the case $a_1=0$ actually follows from \cite{KRY-book} with $n-3$ replaced by $n-2$).
Assume the decompositions (\ref{eq:SplitLattice}).  Notice that $L_2$ is unimodular of rank $4$. There are two cases: either
\begin{enumerate}
\item[(i)]
$\det L_2 =1$ and $L_2 \cong M_2(\Z_p)$ with $Q(x) =\det x$, or
\item[(ii)] $\det L_2= u \in \Z_p^\times$ where
$E=\Q_p(\sqrt u)$ is  the unique unramified quadratic field extension of $\Q_p$,  i.e.  $(p, u)=-1$.
\end{enumerate}
In the second  case, $L_2$ is $\Z_p$-equivalent to $\Z_p^4$ with the  quadratic form $Q(x) =x_1x_2 + x_3^2 -u x_4^2$, or more conceptionally
$$
L_2 \cong \{ A= \kzxz { \alpha} {a} {b} {\alpha'}: \; a, b \in \Z_p, \alpha \in \mathcal O_E\}, \quad Q(A) =\det A,
$$
where $\alpha'$ is the Galois conjugate of $\alpha$. The second case only occurs when $r(T) \ge n-2$, i.e., $a_1=0$. Indeed, if $a_1>0$, i.e., $r(T)=n-3$, then we would  have
$$
\det (2T_1) =\det L = \det M_1 \det L_2 =\det(2T_1) \det L_2,
$$
which implies $1=\det L_2 = u$, a contradiction.  The condition $a_1=0$ is exactly the condition given  in \cite[Theorem 2]{KRHilbert} for $\mathcal Z_{\RZ(n-3)}(J_2)(\kay)$ to be finite. So in both cases, $\RZ(n-3)$ is associated to the supersingular locus at $p$ of the Hilbert modular surface over a real quadratic field $F$ with $p$ split or inert in $F$, $F_p=\Q_p\times \Q_p$ or $E$. In  \cite{KRHilbert}, Kudla and Rapoport considered twisted Hilbert modular surfaces to avoid issues with the boundary. But their localization at $p$, considered in  \cite[Sections 6-12]{KRHilbert}, is  for our $p$ the same as for a regular Hilbert modular surface, and hence their local results apply. We restate it here as the following theorem for convenience. In case (ii), $a_1=0$, it is \cite[Proposition 6.2]{KRHilbert}. In case (i), it is \cite[Proposition 5.4]{GK}, restated in \cite[Proposition 11.2]{KRHilbert} with a minor mistake (it should not assume $a_1=0$ in this case).

\begin{proposition}\label{prop:KR} (Kudla-Rapoport) Let the notation and hypotheses be as above, in particular $\mathcal Z_{\RZ(n-3)}(J_2)$ is $0$-dimensional, and let $P^* \in \mathcal Z_{\RZ(n-3)}(J_2)$. Recall that $T_2$ is $\Z_p$-equivalent to $\diag(\alpha_1 p^{a_1}, \alpha_2 p^{a_2}, \alpha_3 p^{a_3})$ with $0 \le a_1 \le a_2 \le a_3$, and $\alpha_i \in \Z_p^\times$.  Then
$
\Ht_p(P^*) =\nu_p(T_2),
$
where $\nu_p(T_2)$ is given as follows:
\begin{enumerate}
\item  When $a_2 \equiv a_1 \pmod 2$,
$
\nu_p(T_2)
$
is equal to
$$
\sum_{i=0}^{a_1-1} (i+1)(a_1+_2+a_3-3i) p^i
+\!\sum_{i=a_1}^{\frac{a_1+a_2}2 -1} \! (a_1+1) (2a_1+ a_2 +a_3-4i) p^i
+\frac{a_1+1}2(a_3-a_2+1)p^{\frac{a_1+a_2}2}.
$$

\item  When $a_2 \not\equiv a_1 \pmod 2$,
$
\nu_p(T_2)
$
is equal to
$$
\sum_{i=0}^{a_1-1} (i+1)(a_1+_2+a_3-3i) p^i
+\!\sum_{i=a_1}^{\frac{a_1+a_2-1}2}\! (a_1+1) (2a_1+ a_2 +a_3-4i) p^i.
$$
\end{enumerate}
\end{proposition}

\subsection{Local Whittaker functions and the local arithmetic Siegel-Weil formula}  \label{sect:LocalDensity} Let $\psi=\psi_p$ be the `canonical'  unramified additive character of $\Q_p$ used in this paper.
Let $L$ be an integral quadratic lattice over $\Z_p$ of rank $l$, and let $\chi_L =( (-1)^\frac{l(l-1)}2 \det L, \, \cdot)_p$ be the associated quadratic character. For every integer $r\geq 0$ we also consider the lattice $L^{(r)}=L\oplus H^r$, where $H=\Z_p^2$  is the standard hyperbolic plane with the quadratic form $Q(x, y) =xy$. We temporarily allow $L$ to be non-unimodular.

Let $T \in \Sym_n(\Z_p)$ be non-singular with $n \le l$. Then according to \cite[Appendix]{Ku1} and \cite{YaDensity}, there is a local  density polynomial  $\alpha_p(X, T, L)$ of $X$ such that for every integer $r \ge 0$, one has
 $$
 \alpha_p(p^{-r}, T, L) =\int_{\Sym_n(\Q_p)} \int_{L^{(r), n}}  \psi( \tr b (Q(x) -T)) \,dx \,db,
 $$
where $dx$ and $db$ are the standard Haar measures with $\vol(L, dx) =\vol(L^{(r)}, dx) =1$ and $\vol(\Sym_n(\Z_p), db) =1$.  We write  $\varphi_L=\cha(L^n)\in S(L_{\Q_p}^n)$ for the characteristic function of $L$. Then it is easy to see that
$$
W_{T, p}(1, s, \lambda(\varphi_L)) = \left(\frac{\gamma(L)}{\sqrt{[L':L]}}\right)^n \alpha_p(p^{-s}, T, L).
$$
Here $\gamma(L)=\gamma(L\otimes_{\Z_p}\Q_p)$ is the local Weil index.
 We also recall \cite[Section 2]{YaDensity}  that
$
\alpha_p(p^{-r}, T, L)
$
is the local  representation density $\alpha_p(M_T, L^{(r)})=\beta_p(M_T, L^{(r)})$ studied in Kitaoka's book \cite[Section 5.6]{KiBook}. Here $M_T=\Z_p^n$  is the quadratic lattice associated to $T$, i.e., with the quadratic form  $Q(x) ={}^tx T x$.  For a unimodular lattice $L$ of rank $l$, define
\begin{equation} \label{eq:deltaL}
\delta_L =\begin{cases}
  0 &\ff   l \equiv 1 \pmod 2,
  \\
  \chi_L(p) &\ff  l \equiv 0 \pmod 2.
  \end{cases}
\end{equation}


\begin{lemma} \label{lem6.6} Let $T \in \Sym_n(\Z_p)$ with $\det T \in \Z_p^\times$, i.e., $T$ is unimodular.
\begin{enumerate}
\item   Assume that $L=L_1 \oplus L_0$
 is an integral lattice  over $\Z_p$ such that $Q(x) \in p\Z_p$ for every $x \in L_0$. Then
$$
\alpha_p(X, T, L) =\alpha_p(X, T, L_1).
$$

\item Assume that $L$ is $\Z_p$-unimodular of rank $l \ge n$. Then
$$
\alpha_p(X, T, L)=(1- \delta_{L}   p^{-\frac{l}2} X ) (1+\delta_{M_T\oplus L^-}  p^{-\frac{l-n}2}X ) \prod_{\frac{l-n+1}2 \le e \le  \frac{l-1}2 } (1-p^{-2e} X^2).
$$
Here $L^-$ denotes the lattice $L$ with the rescaled quadratic form $Q^-(x) =-Q(x)$.
\end{enumerate}
\end{lemma}

\begin{proof}
Let $\tilde L=L/pL$ with the $\F_p$-valued quadratic form $\tilde Q(x) =Q(x) \mod p $ for an integral quadratic $\Z_p$-lattice $L$. Replacing $L$ by $L^{(r)}$, we may assume $X=1$ in the proof.

For (1),   write $l$ and  $l_i$ for the rank of $L$ and $L_i$ respectively with  $l =l_0+l_1$. Notice that $\tilde L_0$ is a  zero quadratic space of dimension $l_0$.  Every isometry from $\tilde M_T$ to $\tilde L$ splits into the sum of an isometry from $\tilde M_T$ to $\tilde L_1$ and a homomorphism from $\tilde M_T$ to $\tilde L_0$.    So \cite[p.~99, exercise]{KiBook} gives
\begin{align*}
\alpha_p(1, T, L)
&=p^{\frac{n(n+1)}2- nl} |\{ \hbox{isometries from $\tilde M_T$ to } \tilde L\}|
\\
&=p^{\frac{n(n+1)}2- nl} |\{ \hbox{isometries from $\tilde M_T$ to } \tilde L_1\}| \cdot |\Hom(\tilde M_T, \tilde L_0)|
\\
&=p^{\frac{n(n+1)}2- n l_1} |\{ \hbox{isometries from $\tilde M_T$ to } \tilde L_1\}|
\\
&=\alpha_p(1, T, L_1).
\end{align*}

For (2),   \cite[Theorem 1.3.2]{KiBook} and the formula in \cite[p.~99, exercise]{KiBook} imply
$$
\alpha_p(1, T, L) = (1- \chi(\tilde{L})  p^{-\frac{l}2}) (1+\chi(\tilde{M}_T \oplus \tilde{L}^-) p^{-\frac{l-n}2}) \prod_{\substack{l-n+1 \le e \le  l-1 \\ \text{$e$ even}}} (1-p^{-e}),
$$
where $\chi(\tilde M)$ for an unimodular quadratic  $\Z_p$-lattice  $\tilde M$ is defined as follows. When  $l =\dim \tilde M$ is odd, $\chi(\tilde M)=0$. When  $l$ is even, $\chi(\tilde M)$ is $\pm 1$ depending on whether $\tilde M$ is equivalent to a direct  sum of hyperbolic planes  or not. Assume $l=2r+2$ is even.  Since $M$ is unimodular, $M$ is equivalent to $H^r \oplus M_0$ with $M_0 =\Z_p^2$ with $Q(x, y) =x^2 -\epsilon y^2$ for some $\epsilon \in \Z_p^\times$. Then  $\chi(\tilde{M})=1$ if and only if $\tilde{M}_0$  is a hyperbolic plane, which is the same as saying that $\epsilon$ is a square in $\mathbb F_p$, i.e.,  $(\epsilon, p) =1$. On the other hand, it is easy to check that
$$
\chi_M(x) =( (-1)^{r+1} \det M, x) =(\epsilon, x).
$$
So $\chi(\tilde M) =\chi_M(p)$ in this case.  This proves (2).
\end{proof}

\begin{proposition} \label{prop:Surprise} Assume that  $T$ is $\Z_p$-equivalent  to $\diag(T_1, T_2)$ with $T_1$ being unimodular of rank $n-3$ and $T_2 =\diag(\alpha_1 p^{a_1}, \alpha_2 p^{a_2}, \alpha_3 p^{a_3})$ with $\alpha_i \in \Z_p^\times$ and $0\le a_1 \le a_2 \le a_3$. Let $L$ be a  unimodular lattice of rank $n+1$.  Let $M_1$ be the unimodular quadratic lattice with Gram matrix $2T_1$, and fix an embedding $M_1 \hookrightarrow L$, which results in a decomposition $L=M_1\oplus L_2$. Then $W_{T, p}(1, 0, \lambda(\varphi_L))=0$ if and only if $W_{T_2, p}(1, 0, \lambda(\varphi_{L_2}))=0$. In such a case, we have
$$
\frac{W_{T, p}'(1, 0, \lambda(\varphi_L))}{W_{T^u, p}(1, 0, \lambda(\varphi_L))}
=\frac{W_{T_2, p}'(1, 0, \lambda(\varphi_{L_2}))}{W_{T_2^u, p}(1, 0, \lambda(\varphi_{L_2}))}
=\nu_p(T_2).
$$
Here $T^u$ and $T_2^u$ denote any  unimodular symmetric matrices over $\Z_p$ of order $n$ and $3$, respectively, and $\nu_p(T_2)$ is given in Proposition \ref{prop:KR}.
\end{proposition}
\begin{proof} By Lemma \ref{lem6.6},
$W_{T^u, p}(1, 0, \lambda(\varphi_L)) =\gamma(V_p)^n \alpha(1, T^u, L)$ does not depend on the choice of the $p$-unimodular $T^u$. We choose $T^u=\diag(T_1, T^u_2)$ with $T^u_2$  being unimodular.
Now  \cite[Corollary 5.6.1]{KiBook} implies  ($X=p^{-s}$)
\begin{align*}
\frac{W_{T, p}(1,s,  \lambda(\varphi_L))}{W_{T^u, p}(1, s, \lambda(\varphi_L))}
&=\frac{\alpha_p(X, T, L)}{\alpha_p(X, T^u, L)}
\\
&=\frac{\alpha_p(X, T_1, L) \alpha_p(X, T_2, L_2)}{\alpha_p(X, T_1, L) \alpha_p(X, T^u_2, L_2)}
\\
&=\frac{W_{ T_2, p}(1, s, \lambda(\varphi_{L_2}))}{W_{ T^u_2, p}(1, s, \lambda(\varphi_{L_2}))}.
\end{align*}
This proves the first identity and also the claim about the vanishing at $s=0$. Assume $W_{T, p}(1, 0, \lambda(\varphi_L) )=0$. We have by \cite[Propostions 11.5 and 7.2]{KRHilbert},
$$
W_{T_2, p}'(1, 0, \lambda(\varphi_{L_2}) )= \gamma(V_2^3) (1 -p^{-2}) (1-\chi_{L_2}(p) p^{-2})\nu_p(T_2).
$$
On the other hand, Lemma \ref{lem6.6} gives
$$
W_{T_2^u, p}(1, 0, \lambda(\varphi_{L_2}) )= \gamma(V_2^3) (1 -p^{-2}) (1-\chi_{L_2}(p) p^{-2}).
$$
Now the second identity is clear.
\end{proof}

Combining Propositions \ref{prop:KR} and \ref{prop:Surprise} and Corollary \ref{cor:LiZhu}, we obtain the following local arithmetic Siegel-Weil formula.

\begin{theorem} \label{theo:LASW-finite} Let $L$ be a unimodular quadratic $\Z_p$-lattice of rank $n+1$
with $p \ne 2$, and let $\RZ$ be the Rapoport-Zink space
as in Section \ref{sect:RZSpace}.
Let $T \in \Sym_n(\Z_p)$ be of rank $n$ and let $J \subset \newV$ be a $\Z_p$-sublattice of rank $n$  which has a basis with Gram matrix $2T$. Assume $\mathcal Z(J)$ is $0$-dimensional and let $P \in \mathcal Z(J)$. Then $$
\Ht_p(P) \log p = \frac{W_{T, p}'(1, 0, \lambda(\varphi_L))}{W_{T^u, p}(1, 0, \lambda(\varphi_L))},
$$
where $T^u$ is a unimodular matrix in $\Sym_n(\Z_p)$ (i.e., $\det T^u \in \Z_p^\times$).
\end{theorem}


\section{Arithmetic Siegel-Weil formulas} \label{sect:ArithSW}


In this section, we will prove the arithmetic Siegel-Weil formulas as stated in Theorem~\ref{maintheo1} and Remark~\ref{rem:maintheo1}.
Throughout, let $V$ be a quadratic space over $\Q$ of signature $(m, 2)$, and let $H=\SO(V)$.

\subsection{Vanishing of coefficients of Eisenstein series}

Let $n=m+1$ and let $\mathcal C =\otimes_{p\le \infty} \mathcal C_p$  be the incoherent quadratic space over $\A$ defined in the introduction.
Recall the $G_\A$-equivariant map
$$
\lambda=\otimes \lambda_p:  S(\mathcal C^n) \rightarrow  I(0, \chi_V), \quad \lambda(\phi)(g) =\omega(g) \phi(0).
$$
For simplicity, we also write $\lambda(\phi)$ for the associated standard section in  $I(s, \chi_V)$.
Let $\phi_\infty^{\mathcal C}(x)  = e^{- \pi \tr (x, x)} \in S(\mathcal C_\infty^n)$, then  $\lambda_\infty(\phi_\infty) =\Phi_\kappa \in  I(s, \chi_V) $ is the standard section of weight $\kappa=\frac{m+2}{2}$.  Recall that for a standard factorizable section  $\Phi=\prod \Phi_p \in  I(s,\chi_V)$, the Eisenstein series
$$
E(g, s, \Phi) = \sum_{\gamma \in P_\Q \backslash \Sp_n(\Q)} \Phi(\gamma g, s)
$$
has a meromorphic continuation to the whole  complex $s$-plane and is holomorphic at $s=0$. It has a Fourier expansion of the form
\[
E(g, s, \Phi) = \sum_{T\in \Sym_n(\Q)} E_T(g,s,\Phi).
\]
When  $T \in \Sym_n(\Q)$ is non-singular, the $T$-th Fourier coefficient factorizes,
$$
E_T(g, s,  \Phi) =\prod_{p \le \infty} W_{T, p}(g_p, s, \Phi_p),
$$
into  product of local Whittaker functions, see \eqref{eq:Whittaker}.
For every $\phi \in S(V(\A_f)^n) =S(\mathcal C_{\A_f}^n)$, we define the Siegel-Eisenstein series of weight $\kappa$  on the Siegel upper half plane $\mathbb  H_n$ as
\begin{equation}
E(\tau, s,  \lambda(\phi)\otimes \Phi_\kappa) = (\det v)^{-\frac{\kappa}{2}}\cdot E(g_\tau, s,  \lambda(\phi)\otimes  \Phi_\kappa),
\end{equation}
where we write $g_\tau=n(u)m(a)\in G_\R$ with $u=\Re(\tau)\in \Sym_n(\R)$ and $a\in \GL_n(\R)$ such that $a\,{}^ta=v$ as usual. In particular, we have  $g_\tau (i 1_n)= \tau$.
We could choose for $a$ the positive symmetric square root of $v$ but we do not have to.
The Eisenstein series vanishes automatically at $s=0$ due to the incoherence. The arithmetic Siegel-Weil formula, envisioned by Kudla, aims to give arithmetic meaning to its central derivative at $s=0$. From now on, assume $T=\Sym_n(\Q)$ is non-singular, and let
\begin{equation}
\Diff(\mathcal C, T) =\{ p \le \infty: \;\mathcal C_p  \hbox{ does not represent } T\}
\end{equation}
be Kudla's  Diff set  defined in the introduction.
Then  $\Diff(\mathcal C, T)$ is a finite set of odd order, and $\infty \in \Diff(\mathcal C, T)$ if and only if $T$ is not positive definite. Moreover, when $p \in \Diff(\mathcal C, T)$, then $W_{T, p}(g_p, 0, \lambda_p(\phi_p)) =0$.   So
\begin{equation} \label{eq:Order}
\ord_{s=0} E_T(g,s, \lambda(\phi)\otimes \Phi_\kappa) \ge |\Diff(\mathcal C, T)|
\end{equation}
for every $\phi \in S(V(\A_f)^n)$.

\subsection{The arithmetic Siegel-Weil formula at  infinity}
Here we prove Theorem \ref{maintheo1} (2) of the introduction. We begin by recalling the global setup.

For a compact open subgroup $K\subset H(\A_f)$ we consider
the Shimura variety $X_K$ whose associated complex space is
\begin{align*}
X_K(\C)= H(\Q)\bs \calD\times H(\A_f)/K.
\end{align*}
It is a quasi-projective variety of dimension $m$, which has a canonical model over $\Q$.

Given $x =(x_1, \dots, x_n) \in V(\Q)^n$ with $Q(x) =\frac{1}2 (x, x) =\frac{1}2 ( (x_i, x_j)) >0$, let
 $H_x$ be the stabilizer of $x$ in $H$. For $h \in H(\A_f)$, let $K_{h, x} = H_x(\A_f) \cap h K h^{-1}$ be the corresponding compact open subgroup of $H_x(\A_f)$.  Then
$$
H_x(\Q) \backslash  \calD_x \times H_x(\A_f)/K_{h, x} \rightarrow X_K,  \quad [z, h_1]\mapsto [z, h_1 h]
$$
gives rise to a cycle $Z(h, x)$ in $X_K$ of codimension $n$. More generally, given a positive definite $T\in \Sym_n(\Q)$ and any $K$-invariant Schwartz function $\varphi\in S(V^n(\A_f))$, Kudla \cite{Ku:Duke} defines a weighted cycle as follows: If there exists an $x \in V^n(\Q)$ with $Q(x) =T$, put
$$
Z(T, \varphi)  =\sum_{h\in H_x(\A_f)\backslash H(\A_f)/K} \varphi(h^{-1}x) Z(h, x) \in Z^n(X_K).
$$
If there is no such  $x$, set $Z(T, \varphi)=0$.  These weighted cycles behaves well under pull-back (for varying $K$).
Moreover, if $T\in \Sym_n(\Q)$ is regular but not positive definite, we put $Z(T,\varphi)=0$.


If $T\in \Sym_n(\Q)$ is regular, we define a
Green current for the cycle $Z(T,\varphi)$ by
\[
G(T,\varphi,v,z,h) = \sum_{\substack{x\in V^n(\Q)\\Q(x)=T}}
\varphi(h^{-1}x)\cdot \xi^n_0(xa,z),
\]
where $z\in \calD$, $h\in H(\A_f)$, and $a\,{}^ta=v=\Im(\tau)$.
The pair
\[
\widehat Z(T,\varphi,v) =\big( Z(T,\varphi), G(T,\varphi,v)\big)\in \Cha^n_\C(X_K)
\]
defines an arithmetic cycle, which depends on $v$.
For the rest of this section we assume that $n=m+1$.
In this case,
the cycles $Z(T,\varphi)$ are all trivial (in the generic fiber) for signature reasons.
However, for indefinite $T$, the arithmetic cycles $\widehat Z(T,\varphi,v)$ typically have a non-trivial current part.
We are interested in their archimedian arithmetic degree
\[
\widehat\deg_\infty\widehat Z(T,\varphi,v) = \frac{1}{2}\int_{X_K(\C)} G(T,\varphi,v).
\]

We are now ready to prove Theorem \ref{maintheo1} (2) of the introduction, which we restate here in a version which also gives an explicit value for the constant of proportionality.

\begin{theorem}
\label{theo:ArithSW-infinite}
Assume that $T \in \Sym_n(\Q) $ is of signature $(n-j, j)$ with $j>0$ and that $\varphi \in S(V(\A_f)^n)$ is $K$-invariant. Then the arithmetic Siegel-Weil formula holds for $T$, i.e.,
$$
 \widehat{\deg}_\infty\widehat{ Z}(T, \varphi,v) \cdot q^T= C_{n,\infty}\cdot  E_T'(\tau,0, \lambda(\varphi)\otimes \Phi_\kappa),
$$
where the constant $C_{n,\infty}$ is given as follows. Let $L\subset V$ be an integral lattice, and let $d_L h =\prod_{p < \infty} d_{L_p} h$ be the associated Haar measure on $H(\A_f)$, and $C(L) =\prod_{p <\infty} C(L_p)$ be the associated constant given in  Proposition \ref{prop:Measure} (with respect to the unramified additive character $\psi_f$ of $\A_f$).  Then
$$
C_{n,\infty} = -B_{n, \infty} \frac{C(L)}{\vol(K, d_L h)} .
$$
\end{theorem}

\begin{proof}
The archimedian arithmetic degree is given by
\begin{align*}
 \widehat \deg_\infty\widehat Z(T,\varphi,v) &= \frac{1}{2}\int_{X_K(\C)} G(T,\varphi,v)\\
&=  \frac{1}{2}\int\limits_{H(\Q)\bs \calD\times H(\A_f)/K} \sum_{\substack{x\in V(\Q)^n\\ Q(x)=T}}
\varphi(h_f^{-1} x) \cdot \xi_0^n(xa,z)\, dh_f.
\end{align*}
This quantity vanishes if $V(\Q)$ does not represent $T$. Then, by the  Hasse principle, there is at least one finite prime $p$ such that $V(\Q_p)$ does not represent $T$, i.e., $p \in \Diff(\mathcal C, T)$. As $\infty \in \Diff(\mathcal C, T)$, we see that $|\Diff(\calC, T)| >1$ and  that  $E_T'(g,0,\lambda(\varphi)\otimes \Phi_\kappa) =0$. Hence the theorem holds trivially.

We now assume that there exists an $x_0\in V(\Q)^n$ with $Q(x_0)=T$.
Then, by Witt's theorem, any other $x\in V(\Q)^n$ with $Q(x)=T$ is an $H(\Q)$-translate of $x_0$. Let $dh_f$ be any prefixed Haar measure on $H(\A_f)$.  Notice also that that  the point-wise stabilizer $H_{x_0}(\Q)$ of $x_
0$ is trivial since $n=m+1$.  By unfolding, the above integral is equal to
\begin{align*}
  \widehat \deg_\infty(\widehat Z(T,\varphi,v)) &=  \frac{1}{2}\vol(K,dh_f)^{-1}\int_{H_{x_0}(\Q)
\bs \calD\times H(\A_f)}
\varphi(h_f^{-1} x_0) \cdot \xi_0^n( x_0 a,z)\, dh_f\\
&=\frac{1}{2}
\vol(K, dh_f)^{-1} \int_{ H(\A_f)}
\varphi(h_f^{-1} x_0)\, dh_f\cdot
\int_{\calD}\xi_0^n( x_0 a,z)  .
\end{align*}
By Theorem \ref{thm:alsw}, the archimedian integral is equal to
$$
\frac{1}{2}\int_{\calD}\xi_0^n( x_0 a,z) =
\Ht_\infty(x_0a)
= -B_{n,\infty}\det(v)^{-\kappa/2}\cdot W'_{T,\infty}(g_\tau,s_0,\Phi_\kappa)\cdot q^{-T}.
$$
On the other hand, the quantity
$$
\vol(K, dh_f)^{-1} \int_{ H(\A_f)}
\varphi(h_f^{-1} x_0)\, dh_f
$$
is clearly independent of the choice of the product Haar measure $dh_f$. We choose $dh_f = d_L h$, then the local Siegel-Weil formula, Proposition~\ref{prop:localSW}, gives
$$
\vol(K, dh_f)^{-1} \int_{ H(\A_f)}
\varphi(h_f^{-1} x_0)\, dh_f  = \frac{C(L)}{\vol(K, d_L h)} W_{T, f}(1, 0, \lambda(\varphi)).
$$
This implies the assertion.
%
\end{proof}

\subsection{The arithmetic Siegel-Weil formula at a finite prime} \label{sect:ASW-finite}

Assume that $n=m+1$. Let $p\ne 2$ be a prime number. Let $L\subset V$ be a $p$-unimodular lattice.  Let  $H=\SO(L)$,  and put $\tilde H=\GSpin(L)$. Let $K=K_pK^p$  be a compact open subgroup of $H(\A_f)$ fixing $L$ with
$$
K_p=H(\Z_p)=\{ h \in H(\Q_p): \;  hL_p =L_p\}.
$$
For convenience, we assume that there is a   compact open subgroup $\tilde K\subset \tilde H(\A_f)$ which  contains $\hat\Z^\times$ and which maps onto $K$. Such a $\tilde K$ always exists if $K$ is
contained in the discriminant kernel subgroup of some even lattice in $V$ (see Remark \ref{rem:7.4}).
%
Under this assumption,  the Shimura variety $X_K$ associated to $(H, K)$ is the same as the Shimura variety associated to $(\tilde H, \tilde K)$. The associated complex spaces are both equal to
$$
\tilde H(\Q) \backslash \mathbb D \times \tilde H(\A_f)/\tilde K \cong H(\Q) \backslash \mathbb D \times H(\A_f)/K.
$$
Let $e,f\in V$ be orthogonal vectors of negative length in $\Z_{(p)}^\times$. Then $\delta =ef\in C(V)$
 with  $\delta^\iota =-\delta$
and $\norm(\delta) =\delta \delta^\iota  \in \Z_{(p)}^\times$.
This determines a symplectic form $\psi_\delta(x, y) =\tr (x \delta y^\iota)$  on $C(V)$, for which
the lattice $C(L)$ is $p$-unimodular.
We obtain an embedding
$$
\HH \rightarrow \hbox{GSp}(C(V))
$$
and a morphism of Shimura varieties over $\Q$ from $X_K$ to the Siegel Shimura variety determined by the symplectic space $(C(V),\psi_\delta)$ and a suitable compact open subgroup.

%
The integral model of the Siegel Shimura variety induces then an integral model $\mathcal X=\mathcal X_K$ of $X_K$ \cite{Ki}, \cite[Section 4]{AGHM}.  Kisin showed that  $\mathcal X$  is smooth over $\Z_{(p)}$, if the compact open subgroup $K^p\subset H(\A_f^p)$ is sufficiently small.

By pulling back the universal abelian scheme, we obtain a
polarized abelian scheme $(\mathcal A^{\text{KS}}, \lambda^{\text{KS}}, \eta^{\text{KS}})$ with level structure over $\mathcal X$, the {\em Kuga-Satake abelian scheme}. It is equipped with a right $C(L)$-action.

Given a $\Z_{(p)}$-scheme $S$ and  an $S$-point   $\alpha:  S \rightarrow \mathcal X$, we obtain a triple  $\mathbf{A}_\alpha =(A, \lambda, \eta ) =\alpha^*(\mathcal A^{\text{KS}}, \lambda^{\text{KS}}, \eta^{\text{KS}})$ by pulling back the Kuga-Satake scheme.  In  particular, $\eta$ is a $\tilde K^{p}$-level structure
$$
\eta:  \mathbf{H}_{\A_f^p} :=\bigotimes_{\substack{l < \infty \\  l\ne p}}\left( H_{l}^1(A)\otimes_{\Z_l} \Q_l\right)   \xrightarrow{\sim} C(V) \otimes_\Q \A_f^p,
$$
sending $\mathbf{V}_{\A_f^p} $ (the \'etale realization of the motive associated to the representation of $\tilde H$ on $V$) onto $V\otimes \A_f^p$.
Let $V(\mathbf{A}_\alpha)\subset \End_{C(L)}(A)_{(p)}$
 be the space of special endomorphisms of $\mathbf A_\alpha$ defined in  \cite[Definition 3.3]{Cihan-thesis}.

Given $T\in \Sym_n(\Q)$ with $\det T \ne 0$, the special cycle $\mathcal Z(T) \rightarrow \mathcal X$ is defined as the stack over $\mathcal X$ with functor of points
$$
\mathcal Z(T)(S) =\{ (\alpha,  x): \;  \alpha  \in \mathcal X(S), \,  x=(x_1, \dots, x_n) \in V(\mathbf{A}_\alpha)^n, \, Q( x) =T, \,\eta \circ x_j \circ  \eta^{-1}  \in \hat{L}^{(p)} \},
$$
where $\hat{L}^{(p)} = \prod_{ l\neq  p} L_l$, and $\hat L=\prod_{l} L_l$. In this subsection we drop the Schwartz function  $\varphi$ from the notation of $\calZ(T,\varphi)$, since we only consider it here for the characteristic function of $\hat L^n$.

Soylu showed in  \cite[Proposition 3.7]{Cihan-thesis} that  the image of the forgetful map
$$
\mathcal Z(T) \rightarrow \mathcal X
$$
sends $\mathcal Z(T)(\kay)$ into the supersingular locus $\mathcal X_{ss} \subset \mathcal X(\kay)$.
According to \cite[Proposition 7.2.3]{HP}, there exists an $\alpha_0 \in \mathcal X_{ss}$ such that the  $p$-divisible group $(X_0, \lambda_0)$ associated to $\mathbf{A}_{\alpha_0}$ is equal the $p$-divisible group $(\mathbb X_0,\lambda_0)$ considered in Section \ref{sect:RZSpace}.

 According to \cite[Theorem 7.2.4]{HP} or \cite[Theorem 1.2]{Shen}, there is an isomorphism  of formal schemes
\begin{equation} \label{eq:FpPoints}
\Theta:  \newH(\Q) \backslash \overline{\RZ} \times \newH(\A_f^p)/K^p  \cong \tilde{\newH}(\Q)\backslash \RZ \times \tilde{\newH}(\A_f^p)/\tilde K^p
   \cong (\widehat{\mathcal X}_W)_{/\mathcal X_{ss}},
\end{equation}
where $(\widehat{\mathcal X}_W)_{/\mathcal X_{ss}}
$ is the completion of  $\mathcal X_W$ along  the supersingular  locus $\mathcal X_{ss}$, and $\overline{\RZ} =p^\Z \backslash \RZ$.
The above discussion implies that for every $(\alpha, x) \in \mathcal Z(T)(\kay)$ the space of special endomorphisms satisfies
$$
V(\mathbf{A}_\alpha) \otimes\Q \cong  \newV,
$$
where $\newV$ is the neighboring quadratic space over $\Q$ associated with $\mathcal C$ at $p$.

\begin{proposition}
\label{prop:Counting}
Let $\newL$ be a fixed lattice of $\newV$ such that $\newL_p $ is  a dual  vertex lattice in $\newV_p$ of type $2$ as in Section~\ref{sect:6.2} and $\hat{\newL}_{q} \cong \hat{L}_{q}$ for $q\ne p$. Let $\varphi_{\newL}=\cha(\hat{\newL}^n)$.  Let $T \in \Sym_n(\Q)$ and assume that it satisfies the conditions of Theorem \ref{theo:Soylu} at the  prime $p$.
Then
\begin{align*}
|\mathcal Z(T)(\kay)|&:=\sum_{x \in \mathcal Z(T)(\kay)} \frac{1}{|\Aut(x)|}
\\
&= 2 \frac{C(\newL)}{\vol(\newK, d_\newL h)} W_{T, f}(1, 0, \lambda(\varphi_\newL))<\infty.
\end{align*}
Here  $\newK=\newK_p K^p$  is the compact open subgroup of $\newH(\A_f)$ with $\newK_p=\SO(\newL_p)\subset \newH(\Q_p)$.
In  particular,  if $\mathcal Z(T)(\kay)$ is not empty, then $ \Diff(\calC, T)=\{p \}$.
\end{proposition}
\begin{proof} Let $\pi: \mathcal Z(T) \rightarrow \mathcal X$ be the forgetful map, and identify via (\ref{eq:FpPoints})
$$
\mathcal X_{ss}(\kay) = \newH(\Q) \backslash \overline{\RZ} \times \newH(\A_f^p)/K^p,
$$
where $\mathcal X_{ss}(\kay)$ denotes the supersingular locus of $\mathcal X(\kay)$. By a result of Soylu \cite[Proposition 3.7]{Cihan-thesis}
 the image of $\mathcal Z(T)(\kay)$ lies in $\mathcal X_{ss}(\kay)$.
Notice that $(\mathbf A,  x) \in  \mathcal Z(T)(\kay)$ implies that  the $p$-divisible group $X$ of $\mathbf A$ belongs to $\mathcal Z(J(x))(\kay)$, where $J(x)$ is the  sublattice of $\newV_p$ generated by the $p$-adic components of $x$ (recall that the stabilizer of $x$ in $\HH_p$ is trivial).

By Proposition \ref{prop:SpecialLattice}, we have
$$
\overline{\mathcal Z (J(x) )}(\kay) =\bigsqcup_{\substack{ t_\Lambda =2 \\ x \in \Lambda^n}} S_\Lambda(\kay)
   =\bigsqcup_{\substack{ h_p \in \newH(\Q_p)/\newK_p \\ x \in  h_p\newL_p}} S_{h_p\newL_p}(\kay),
$$
where $\overline{\mathcal Z(J)}$ is the image of $\mathcal Z (J)$ in $\overline{\RZ}$.  Recall that $|S_{\Lambda}(\kay)|=2$ for any dual vertex lattice $\Lambda \subset \newV_p$ of type $2$. So we find
\begin{align*}
\sum_{x \in \mathcal Z(T)(\kay)} \frac{1}{|\Aut(x)|}
 &=2 \sum_{\substack{ x \in \newV^n  \\ Q(x) =T}}
    \sum_{ h \in \newH (\Q) \backslash \newH(\A_f)/\newK}\frac{1}{|\Gamma_h|} \varphi_{\newL}(h^{-1} x)
 \\
 &=\frac{2}{\vol(\newK, d_{\newL} h)} \int_{\newH(\A_f)} \varphi_{\newL}(h^{-1} x) d_{\newL} h,
\end{align*}
if there is an $ x \in  \newV^n$ with $Q( x) =T$ (otherwise, it is zero). Here
$d_{\newL}$ is the Haar  measure on $\newH(\A_f)$  associated to the lattice $\newL$, and
$
|\Gamma_h|=h^{-1} H(\Q) h\cap K
$.
Now applying the local Siegel-Weil formula, we obtain the proposition.
\end{proof}

Recall that the  arithmetic degree of $\mathcal Z(T)$ at $p$ is defined as
\begin{equation}
\widehat{\deg}_p(\mathcal Z(T)) =\sum_{x \in \mathcal Z(T)(\kay)} \frac{ \Ht_p(x)}{|\Aut(x)|}\cdot \log p
\end{equation}
where $\Ht_p(x)$ is the length of the \'etale local ring $\co_{\mathcal Z(T), x}$ of $\mathcal Z(T)$ at the point $x$.
The following result is a refinement of Theorem \ref{maintheo1} (3).

\begin{theorem}
\label{thm:arithswfin}
Fix a prime number $p \ne 2$.  Let $L\subset V $ be a $p$-unimodular lattice. Let $T \in \Sym_n(\Q)$ such that $T_p$ satisfies the conditions in  Theorem \ref{theo:Soylu}. Then the  arithmetic Siegel-Weil formula holds for $T$ with
$$
\widehat{\deg}_p \mathcal Z(T)\cdot  q^T =C_{n, p} \cdot E_T'(\tau, 0, \lambda(\varphi_L)\otimes \Phi_\kappa),
$$
where
$$
C_{n, p} = - B_{n, \infty}  \frac{C(L)}{\vol(K, d_L h)} .
$$
In particular,  $C_{n, p} =C_{n, \infty}$.
\end{theorem}

\begin{proof}
We have by  Theorem ~\ref{theo:LASW-finite},  Propositions~\ref{prop:Counting}, ~\ref{prop:Surprise}, and ~\ref{prop:Measure} that
\begin{align*}
\widehat{\deg}_p(\mathcal Z(T))
 &= \sum_{x \in \mathcal Z(T)(\kay)} \frac{1}{|\Aut(x)|}\cdot \Ht_p(x) \log p
 \\
   &=\frac{2 C(\newL)}{\vol(\newK, d_\newL h)} W_{T, f}(1, 0, \lambda(\varphi_\newL)) \cdot  \nu_p(T_2)\log p\\
   &= \frac{2 C(L)}{\vol(K, d_L h)}  \frac{\vol(K_p, d_{L_p})}{C(L_p)} \frac{C(\newL_p)}{\vol(\newK_p, d_{\newL_p}h)} W_{T, p}(1, 0, \lambda(\varphi_{\newL_p}))
   \\
   &\qquad \cdot  \Big(\prod_{q\nmid p\infty}W_{T, q}(1, s, \lambda(\varphi_{L_q})\Big)\Big|_{s=0} \cdot\frac{W_{T, p}'(1, 0, \lambda(\varphi_{L_p}))}{W_{\tilde T, p}(1, 0, \lambda(\varphi_{L_p}))}
 \\
  &= \frac{2 C(L)}{\vol(K, d_L h)}\cdot \frac{E_T'(1, 0, \lambda(\varphi_L) \otimes\Phi_\kappa)}{W_{T, \infty}(1, 0, \Phi_\kappa)}.
 \end{align*}
 Here $\tilde T\in \Sym_n(\Z_p)$ is any $p$-unimodular matrix.
Remark \ref{rem2.11}
gives
 \begin{equation}
- \frac{1}{2}B_{n, \infty} \cdot  W_{T, \infty}(1, 0, \Phi_\kappa)= e^{-2 \pi \tr T}
 \end{equation}
 for any positive definite $n\times n$ matrix $T$.  So
 $$
 - \frac{1}{2}B_{n, \infty} \cdot (\det v)^{-\kappa/2} W_{T, \infty}(g_\tau, 0,  \Phi_\kappa)
=  q^T.
 $$
 Hence we obtain the claimed formula.
\end{proof}

\begin{remark}
\label{rem:7.4}
In this subsection we have assumed for convenience that  $\varphi$
is the characteristic function of $\hat L$, and that there is a  compact open subgroup $\tilde K\subset \tilde H(\A_f)$ containing $\hat\Z^\times$ and mapping onto $K$.
Both can be relaxed. First, we can naturally modify the definition $\mathcal Z(T)$  in \cite{Cihan-thesis} to include $\mathcal Z(T, \varphi)$ for all $\varphi =\varphi_p \varphi^p \in S(V(\A_f)^n)^K$ with $\varphi_p=\cha(L_p^n)$.  The proof of Proposition \ref{prop:Counting} goes through without any change.
As already mentioned, the assumption on $K$ is always fulfilled if
there exists an even lattice $M\subset V$ which is stabilized by $K$ and such that $K$ acts trivially on $M'/M$. In other words,
every `sufficiently small'
compact open subgroups $K$ satisfíes the condition. Finally, we indicate how the results can be modified to hold for general $K$.
Take a compact open subgroup $\tilde K_1$ of $\tilde H(\A_f)$ containing $\hat{\Z}^\times$ such that its image $K_1$ in $H(\A_f)$ is contained in $K$.  Then there is a natural projection
$\mathcal X_{K_1}\rightarrow  \mathcal X_K$ and an analogous projection of Rapoport-Zink spaces. 
The $p$-adic uniformization  identity (\ref{eq:FpPoints}) still holds according to \cite[Theorem 1.2]{Shen}. For $\varphi \in S(V(\A_f)^n)^K$, the special cycle $\mathcal Z_{K_1}(T, \varphi)$ is $K$-invariant and descends to a special cycle $\mathcal Z_K(T, \varphi)$ on $\mathcal \calX_K$.
\end{remark}

\begin{remark}
Assume that $\Diff(\mathcal C, T) =\{ p \}$. We observe the following variant of the local arithmetic Siegel-Weil formula:
\begin{align*}
\Ht_p(x) \log p &= \frac{W_{T, p}'(1, 0, \lambda(\varphi_L))}{W_{\tilde T, p}'(1, 0, \lambda(\varphi_L))},  \quad  \text{$p<\infty$, where $\tilde T$ is $p$-unimodular},
\\
\frac{1}2 \Ht_\infty(x)  &=\frac{W_{T, \infty}'(1, 0, \Phi_\kappa)}{W_{\tilde T, \infty}'(1, 0, \Phi_\kappa)},  \quad
\text{$p=\infty$, where $\tilde T$ is positive definite with $\tr \tilde T =\tr T$.}
\end{align*}
Here the extra $\frac{1}2$ makes sense as we do the integral over the whole symmetric domain $\mathcal D$ instead of its connected component $\mathcal D^+$, while at a finite prime $p$, we did it at each  individual point (connected component). This reinterpretation is different from the previous ones used in \cite{KRY-book}, \cite{KRHilbert}, \cite{KRSiegel}, and \cite{HY} among others.
\end{remark}

\end{document}